%% file: lensspacegridhomology.tex
\RequirePackage{etex}
\documentclass[12pt]{amsart}
\usepackage{amsmath, amsthm, amsfonts, amssymb, amscd,comment}
\allowdisplaybreaks[4]
\usepackage[mathscr]{eucal}
\usepackage[dvips, usenames, dvipsnames]{xcolor}
\definecolor{orange2}{HTML}{FF7F2A}
\usepackage[all, knot, arc]{xy}
\usepackage{graphicx}
\usepackage{pinlabel}
\usepackage{textcomp}
\usepackage{float}
\usepackage{stmaryrd}
\usepackage{enumitem}
\usepackage{epstopdf}
\usepackage{tikz}
\usetikzlibrary{arrows,automata}
\usetikzlibrary{matrix,arrows}
\tikzset{cdlabel/.style={above,sloped,%
    execute at begin node=$\scriptstyle,execute at end node=$}}
\tikzset{algarrow/.style={->, thick}}   
\tikzset{alb/.style={->, bend right=25, thick}}
\tikzset{arb/.style={->, bend left=25, thick}}
\tikzset{al/.style={->, bend right=20, thick}}
\tikzset{ar/.style={->, bend left=20, thick}}
\tikzset{als/.style={->, bend right=15, thick}}
\tikzset{ars/.style={->, bend left=15, thick}}
\tikzset{blgarrow/.style={->, thick}}
\tikzset{clgarrow/.style={->, thick}}
\tikzset{tensoralgarrow/.style={double, double equal sign distance, -implies}}
\tikzset{tensorblgarrow/.style={double, double equal sign distance, -implies}}
\tikzset{tensorclgarrow/.style={double, double equal sign distance, -implies}}
\tikzset{tensorelgarrow/.style={double, double equal sign distance, -implies}}
\tikzset{modarrow/.style={->, dashed}}
\tikzset{Amodar/.style={->, dashed}}
\tikzset{Dmodar/.style={->, dashed}}
\tikzset{DAmodar/.style={->, dashed}}

\usepackage{tikz-cd}
\usepackage[noadjust]{cite}
\usepackage{xifthen}
\usepackage{slashbox}
\usepackage{multirow}
\usepackage{makecell}
\usepackage{longtable}
\usepackage[para]{threeparttable}
\usepackage[referable]{threeparttablex}
\usepackage{afterpage}
\usepackage{adjustbox}
\usepackage{caption}
\usepackage{subcaption}
\usepackage{footnote}
\makesavenoteenv{tabular}
\makesavenoteenv{table}
\usepackage[final, colorlinks=true,%
  linkcolor=blue, citecolor=blue, urlcolor=blue]{hyperref}
\usepackage{subcaption}

\usepackage[lmargin=1in,rmargin=1in,tmargin=1in,bmargin=1in]{geometry}

\theoremstyle{definition}
\newtheorem{defn}[equation]{Definition}

\theoremstyle{plain}

\newtheorem{cor}[equation]{Corollary}
\newtheorem{lem}[equation]{Lemma}
\newtheorem{prop}[equation]{Proposition}
\newtheorem{thm}[equation]{Theorem}
\newtheorem{remark}[equation]{Remark}
\numberwithin{equation}{section}
\numberwithin{figure}{section}
\numberwithin{table}{section}
\newcommand{\vphi}{\varphi}
\renewcommand{\phi}{\vphi}

\DeclareMathOperator{\Id}{Id}

\DeclareMathOperator{\Cone}{Cone}

\newcommand{\genG}{\mathcal{G}}
\newcommand{\opminus}[1]{#1^-}
\newcommand{\ophat}[1]{\widehat{#1}}
\newcommand{\optilde}[1]{\widetilde{#1}}

\newcommand{\opunorr}[1]{#1 {}^u}

\DeclareMathOperator{\CFK}{CFK}
\DeclareMathOperator{\HFK}{HFK}

\newcommand{\CFKm}{\opminus{\CFK}}

\newcommand{\HFKm}{\opminus{\HFK}}
\newcommand{\HFKh}{\ophat{\HFK}}
\newcommand{\HFKt}{\optilde{\HFK}}

\newcommand{\CTu}{\opunorr{\CTu}}

\newcommand{\curves}[1]{\overrightarrow{#1}}
\newcommand{\alphas}[1][]{%
  \ifthenelse{\equal{#1}{}}{\curves{\alpha}}{\curves{\alpha^{#1}}}
}
\newcommand{\betas}[1][]{%
  \ifthenelse{\equal{#1}{}}{\curves{\beta}}{\curves{\beta_{#1}}}
}
\newcommand{\Xs}[1][]{%
  \ifthenelse{\equal{#1}{}}{\curves{\mathbb{X}}}{\curves{\mathbb{X}{#1}}}
}
\newcommand{\Os}[1][]{%
  \ifthenelse{\equal{#1}{}}{\curves{\mathbb{O}}}{\curves{\mathbb{O}{#1}}}
}

\newcommand{\markers}[1]{\mathbb{#1}}
\newcommand{\OO}{\markers{O}}
\newcommand{\XX}{\markers{X}}

\newcommand{\gen}[1]{\mathbf{#1}}

\newcommand{\x}{\gen{x}}
\newcommand{\y}{\gen{y}}
\newcommand{\z}{\gen{z}}

\DeclareMathOperator{\Rect}{Rect}

\DeclareMathOperator{\Pent}{Pent}
\DeclareMathOperator{\Hex}{Hex}

\newcommand{\emptypoly}[2][]{%
  #2^{\ifthenelse{\equal{#1}{}}{\circ}{\circ, #1}}
}
\newcommand{\eRect}{\emptypoly{\Rect}}

\newcommand{\ePent}{\emptypoly{\Pent}}

\input{mycommands}

\begin{document}
\title[On grid homology for lens space links: combinatorial invariance]{On grid homology for lens space links:\\ combinatorial invariance and integral coefficients}
\author{Samuel Tripp}
\address {Department of Mathematics, Dartmouth College, Hanover, NH 03755, 
  USA}
\email{\href{mailto:samuel.w.tripp.gr@dartmouth.edu}{samuel.w.tripp.gr@dartmouth.edu}}
\urladdr{\url{http://www.math.dartmouth.edu/~stripp/}}

%
\begin{abstract}
Following the approach to grid homology of links in $S^3$, we prove combinatorially that the grid homology of links in lens spaces defined by Baker, Grigsby, and Hedden is a link invariant. Further, using the sign assignment defined by Celoria, we prove that the generalization of grid homology to integral coefficients is a link invariant. 
\end{abstract}
\maketitle

\section{Introduction}
\input{Sections/intro}


\section{Twisted toroidal grid diagrams}\label{sec:griddiagrams}
\input{Sections/griddiagrams}

\section{Grid homology}\label{sec:gridhomology}
\input{Sections/gridhomology}

\section{Gradings}\label{sec:gradings}
\input{Sections/gradings}

\section{Invariance}\label{sec:invariance}
\input{Sections/invariance}

\section{Sign assignments}\label{sec:signassignments}
\input{Sections/signassignments}

\section{Integral knot Floer homology}\label{sec:integralinvariance}
\input{Sections/integralinvariance}

\bibliographystyle{plain}
\bibliography{master}

\end{document}

%% file: mycommands.tex
\newcommand{\lpq}{L(p,q)}

\DeclareMathOperator{\XSW}{X:SW}
\DeclareMathOperator{\XSE}{X:SE}
\DeclareMathOperator{\XNE}{X:NE}
\DeclareMathOperator{\XNW}{X:NW}
\DeclareMathOperator{\OSE}{O:SE}
\DeclareMathOperator{\OSW}{O:SW}
\DeclareMathOperator{\ONE}{O:NE}
\DeclareMathOperator{\ONW}{O:NW}
\newcommand{\bI}{\mathbf{I}}

\newcommand{\tbI}{\tilde{\bI}}
\newcommand{\tbN}{\tilde{\mathbf{N}}}

\newcommand{\saS}{\mathcal{S}}
\newcommand{\I}{\mathcal{I}}

\newcommand{\dxminus}{\partial_{\XX}^-}
\newcommand{\dalpha}{\partial_{\alpha}}
\newcommand{\dbeta}{\partial_{\beta}}
\newcommand{\dxminusprime}{\partial_{\XX}^{-'}}

\newcommand{\dii}{\partial_{\tbI}^{\tbI}}
\newcommand{\din}{\partial_{\tbI}^{\tbN}}
\newcommand{\dnn}{\partial_{\tbN}^{\tbN}}

\newcommand{\phixnisigned}{\Phi_{X_n,\saS}^{\tbI}}
\newcommand{\phionsigned}{\Phi_{O_n,\saS}}
\newcommand{\phionxnsigned}{\Phi_{O_n,X_n,\saS}}

\newcommand{\numset}[1]{\mathbb{#1}}

\newcommand{\Z}{\numset{Z}}
\newcommand{\F}{\numset{F}}
\newcommand{\Ztwo}{\Z/2\Z}
\newcommand{\Zp}{\Z/p\Z}
\newcommand{\R}{\numset{R}}

\newcommand{\poly}[1]{\mathscr{#1}}

\newcommand{\HH}{\poly{H}}

\DeclareMathOperator{\PG}{Par}
\newcommand{\ePG}{\PG^{\circ}}
\DeclareMathOperator{\Poly}{Poly}
\newcommand{\ePoly}{\Poly^{\circ}}
\newcommand{\pbg}{\Pent_{\beta\gamma}}

\newcommand{\epbg}{\Pent^{\circ}_{\beta\gamma}}

\DeclareMathOperator{\Int}{Int}
\DeclareMathOperator{\spinc}{Spin^c}
\newcommand{\hexbgb}{\Hex_{\beta\gamma\beta}}
\newcommand{\Hbgb}{\HH_{\beta\gamma\beta}}

\newcommand{\ehexbgb}{\Hex^{\circ}_{\beta\gamma\beta}}
\newcommand{\epi}{\pi^{\circ}}

\newcommand{\spincS}{\mathbf{S}}
\newcommand{\maslov}{\mathbf{M}}
\newcommand{\alex}{\mathbf{A}}

\newcommand{\tx}{\tilde{\x}}
\newcommand{\tOO}{\tilde{\OO}}
\newcommand{\ty}{\tilde{\y}}
\newcommand{\tXX}{\tilde{\XX}}
\newcommand{\tO}{\tilde{O}}
\newcommand{\tc}{\tilde{c}}
\newcommand{\tz}{\tilde{\z}}

\newcommand{\gop}{\llbracket}
\newcommand{\gcl}{\rrbracket}

\newcommand{\tgen}{\mathbf{t}}
\newcommand{\w}{\mathbf{w}}

  \newcommand{\figheight}{2.5cm}
  
  \newcommand{\gkprime}{\mathbb{G}'}
  \newcommand{\gridg}{\mathbb{G}}
  \newcommand{\gridgl}{\mathbb{G}_L}

%% file: Sections/intro.tex
Heegaard Floer homology is a package of invariants for closed, oriented 3-manifolds first introduced by Ozsváth and Szabó \cite{Ozsvath00holomorphicdisks}. Knot Floer homology is an extension of this invariant to null-homologous knots in closed, oriented 3-manifolds, introduced independently by Ozsváth and Szabó \cite{Ozsvath2002HolomorphicDA} and Rasmussen \cite{Rasmussen2003FloerHA}. In \cite{Ozsvath_2008}, Ozsváth and Szabó extended knot Floer homology to nullhomologous links in closed, oriented 3-manifolds. The theory was later generalized to non-nullhomologous links in rational homology spheres by Ni \cite{Ni_2009} and Ozsváth and Szabó \cite{Ozsv_th_2010}. 

The above invariants are computed from a Heegaard diagram for a link in the 3-manifold, along with further analytic choices. To this data, one associates a chain complex whose differential counts holomorphic disks in the $g$-fold symmetric product of the Heegaard surface, where $g$ is the genus of the surface. The conditions these disks must satisfy depend on which version of the above invariant one is computing. The simplest form yields a bigraded chain complex over $\Ztwo$. The quasi-isomorphism type of the chain complex does not depend on the choice of diagram for the link or on the choice of analytic data, and the bigraded Euler characteristic of this chain complex is the Alexander polynomial of the link.

In \cite{Sarkar2006AnAF}, Sarkar and Wang showed that the simplest flavors of Heegaard Floer and knot Floer homology admit a combinatorial description. This description relies on a Heegaard diagram for the link in the manifold which is \emph{nice}. They showed that any nullhomologous link in a closed, oriented 3-manifold has a nice Heegaard diagram. On such diagrams, the differential counts embedded rectangles and bigons, and does not require any analytic choices. Soon after, Manolescu, Ozsváth, Szabó, and Thurston used \emph{grid diagrams} for links in $S^3$ to give a combinatorial description of all flavors of knot Floer homology \cite{Manolescu_2007}. In \cite{Manolescu_2007}, the authors also provide a standalone combinatorial proof of invariance without appealing to the equivalence with the holomorphic theory. Additionally, they give a way to assign signs to the rectangles counted by the differential, allowing the combinatorial theory to be generalized to integral coefficients. 

In \cite{Baker_2008}, Baker, Grigsby, and Hedden undertook a similar combinatorial program for links in lens spaces. They introduced \emph{twisted toroidal grid diagrams}, the combinatorics of which capture the topology of links in lens spaces. To a twisted toroidal grid diagram $\gridg$ for a link $L$, they associate a chain complex $(C^-(\gridg),\dxminus)$ over $\Ztwo[V_0,\ldots,V_{n-1}]$, where $n$ is the size of the grid. Each variable $V_i$ is associated to a marking on the grid, and the actions of $V_i$ and $V_j$ are homotopic if the markings they are associated to correspond to the same component of the link. As a result, $(C^-(\gridg),\dxminus)$ can be viewed as a chain complex over $\Ztwo[U_1,\ldots,U_{\ell}]$, where $\ell$ is the number of components of $L$. Let each $U_i$ act by multiplication by $V_{i_j}$, where $V_{i_j}$ corresponds to any marking on the $i^{\text{th}}$ component of $L$. 

In \cite{Baker_2008}, it is shown that this chain complex is quasi-isomorphic to one that computes the knot Floer homology of the link. By \cite{Ozsvath_2008}, they conclude the quasi-isomorphism type of $C^-(\gridg)$ is an invariant of the link. A combinatorial proof that the quasi-isomorphism type of $C^-(\gridg)$ is a link invariant is the first main result of this paper. 

\begin{thm}\label{thm:maintheoremz2}
Suppose $\gridg$ and $\gkprime$ are two twisted toroidal grid diagrams associated to the same $\ell$-component link $L$ in $\lpq$. Then  $(C^-(\gridg),\dxminus)$ and $(C^-(\gkprime),\dxminusprime)$ are quasi-isomorphic as chain complexes over $\Ztwo[U_1,\ldots,U_{\ell}]$ via a combinatorially defined quasi-isomorphism. 
\end{thm}

Following the program for links in $S^3$, Celoria gives a way to assign signs to the embedded parallelograms in twisted toroidal grid diagrams \cite{celoria2015note}. Celoria also proves that for a fixed twisted toroidal grid diagram $\gridg$, the \emph{sign assignment} is unique up to quasi-isomorphism. Utilizing this sign assignment, one can generalize the above complex to integral coefficients. We prove that this generalization is a link invariant; below, let $(C^-(\gridg),\dxminus;\Z)$ be the chain complex incorporating the sign assignment of \cite{celoria2015note}. 

\begin{thm}\label{thm:mainresultz}
Suppose $\gridg$ and $\gkprime$ are two twisted toroidal grid diagrams associated to the same $\ell$-component link $L$ in $ \lpq$. Then $(C^-(\gridg),\dxminus;\Z)$ and $(C^-(\gkprime),\dxminusprime;\Z)$ are quasi-isomorphic as chain complexes over $\Z[U_1,\ldots,U_{\ell}]$. Moreover, there is a combinatorially defined quasi-isomorphism between them.  
\end{thm}

Let $\gridg$ be a twisted toroidal grid diagram associated to an $\ell$-component link $L$ in $\lpq$. Let $\CFKm(L,\lpq;\Z)$ be the quasi-isomorphism type of $(C^-(\gridg),\dxminus;\Z)$ viewed as a chain complex over $\Z[U_1,\ldots,U_\ell]$. Define the \emph{grid homology $\HFKm(L,\lpq;\Z)$ of $L$} to be the homology of $(C^-(\gridg),\dxminus;\Z)$. 

\begin{cor}\label{cor:invariant}
The quasi-isomorphism type $\CFKm(L,\lpq;\Z)$ and the $\Z[U_1,\ldots,U_\ell]$ module $\HFKm(L,\lpq;\Z)$ are invariants of the link $L$. 
\end{cor}

We note that Theorem~\ref{thm:mainresultz} did not follow from an identification with the holomorphic theory. While there is an extension of the holomorphic theory for links to integral coefficients, it is not known if the two definitions agree. For example, it is not understood how the local orientation systems necessary to define the holomorphic theory with integral coefficients change under combinatorial moves.

We also mention that from these invariants, one could extract the simpler $\HFKt$ and $\HFKh$ invariants, by collapsing some of the variables.

There remains an interesting open question of whether there can be torsion found in these integral coefficient grid homology groups. No examples have been produced yet of links in $S^3$ or lens spaces which have torsion in these groups, but that is not expected to hold in general \cite{gridhomology}. 
For knots in $S^3$, another complex homotopy equivalent to grid homology, but more computationally feasible, was described in \cite{beliakova2010simplification} and implemented in \cite{droz2008effective}. An adaptation of this approach to links in lens spaces would allow for greater computation, and is one possible direction of future inquiry, as more computational evidence would help investigate the potential presence of torsion.

\subsection*{Organization} In Section~\ref{sec:griddiagrams} we recall the definition and some properties of twisted toroidal grid diagrams. In Section~\ref{sec:gridhomology}, we recall the definition of grid homology introduced in \cite{Baker_2008} for links in lens spaces. We then take a detour in Section~\ref{sec:gradings} to introduce the gradings on grid homology, and provide some combinatorial proofs of their properties. 

In Section~\ref{sec:invariance} we prove directly that the grid homology of a lens space link is well-defined, using some results about the combinatorics of twisted toroidal grid diagrams. Section~\ref{sec:signassignments} introduces the definition of sign assignments of \cite{celoria2015note}, which is used in Section~\ref{sec:integralinvariance} to define the grid homology of links in lens spaces with integral coefficients, and prove that the quasi-isomorphism type of this is well-defined. 

\subsection*{Acknowledgements}
The author thanks Ina Petkova for suggesting the project, and both Ina Petkova and C.-M. Michael Wong for comments on the draft of this paper.

%% file: Sections/griddiagrams.tex
In this section we recall the definition of twisted toroidal grid diagrams introduced in \cite{Baker_2008}. The combinatorics of these diagrams capture the topology of links in lens spaces. We will assume throughout that $p$ and $q$ are relatively prime integers with $p$ positive and $-p<q<p$.

\begin{defn}[{\cite[Definition 2.1]{Baker_2008}}]
A \emph{(twisted toroidal) grid diagram} $\gridg$ with grid number $n$ for $L(p,q)$ consists of a five-tuple $\gridg=(T^2,\alphas,\betas,\XX,\OO)$, where:

\begin{itemize}
    \item $T^2$ is the standard oriented torus $\R^2/\Z^2$, identified with the quotient of $\R^2$ (with its standard orientation) by the $\Z^2$ lattice generated by the vectors $(1,0)$ and $(0,1)$. 
    \item $\alphas=\{\alpha_0,\ldots,\alpha_{n-1}\}$ are the $n$ images $\alpha_i$ in $T^2$ of the lines $y=\frac{i}{n}$ for $i\in\{0,\ldots,n-1\}$. Their complement $T^2-\alpha_0-\ldots-\alpha_{n-1}$ has $n$ connected annular components, which we call the \emph{rows} of the diagram. 
    \item $\betas=\{\beta_0,\ldots,\beta_{n-1}\}$ are the $n$ images $\beta_i$ in $T^2$ of the lines $y=-\frac{p}{q}(x-\frac{i}{pn})$ for $i\in\{0,\ldots,n-1\}$. Their complement $T^2-\beta_0-\ldots-\beta_{n-1}$ has $n$ connected annular components, which we call the \emph{columns} of the diagram. 
    \item $\OO=\{O_0,\ldots,O_{n-1}\}$ are $n$ points in $T^2-\alphas-\betas$ with the property that no two $O$'s lie in the same row or column. 
    \item $\XX=\{X_0,\ldots,X_{n-1}\}$ are $n$ points in $T^2-\alphas-\betas$ with the property that no two $X$'s lie in the same row or column. 
\end{itemize}
\end{defn}

Associated to such a twisted toroidal grid diagram, there is a unique isotopy class of link $L$ in $L(p,q)$. First, note that $(T^2,\alphas,\betas,\XX)$ is a multi-pointed Heegaard diagram for $\lpq$. The lens space $\lpq$ is comprised of two solid tori $V_{\alpha}$ and $V_{\beta}$, glued along their common boundary $T^2$. We view $V_\alpha$ as laying below $T^2$, and $V_\beta$ as laying above it. The curves $\alphas$ are meridians of $V_{\alpha}$, and $\betas$ are meridians of $V_{\beta}$. A link $L$ in $\lpq$ can then be gotten from a grid diagram $G$ by connecting each marking $X_i$ to the unique marking $O_j$ in the same row as it via an oriented embedded arc in $T^2$ disjoint from $\alphas$. Next, each marking $O_j$ is connected to the unique marking $X_k$ in the same column as it via an oriented embedded arc in $T^2$ disjoint from $\betas$. Pushing the interiors of the arcs disjoint from $\alphas$ down into $V_{\alpha}$ and the interiors of the arcs disjoint from $\betas$ up into $V_{\beta}$, we get an oriented link $L$. As a result, we occasionally denote a grid diagram $\gridg$ by $\gridgl$, to remind ourselves of the link associated to the grid diagram. 

It is also true that for any link $L$ in a lens space, there is a grid diagram which represents it (see~{\cite[Proposition 4.3]{Baker_2008}}). 

Furthermore, \cite{baker2008grid} makes explicit the sense in which the combinatorics of twisted toroidal grid diagrams capture the topology of links in lens spaces: 

\begin{thm}[{\cite[Theorem 5.1]{baker2008grid}}]\label{thm:cromwell}
Let $\gridg$ and $\gkprime$ be grid diagrams representing smooth links $L$ and $L'$ in $L(p,q)$. Then $L$ and $L'$ are smoothly isotopic if and only if $\gridg$ and $\gkprime$ are related by a sequence of grid moves: stabilizations, destabilizations, and commutations. 
\end{thm}

We will describe these moves in Section~\ref{sec:invariance}. 

We can choose any number of fundamental domains for $T^2$. 
In general, we will look at one specific class of fundamental domains. Throughout the rest of this paper, whenever we fix a fundamental domain, it is assumed to be one of this form. 
Each such is a choice of a subset of $\R^2$ bounded vertically between $y=\frac{k}{n}$ and $y=\frac{k}{n}+1$, with those two edges identified, for some $k\in\Z$, and horizontally between $y=-\frac{p}{q}(x-\frac{l}{pn})$ and $y=-\frac{p}{q}(x-\frac{l}{pn}-1)$, with those two edges identified, for some $l\in\Z$. 
Choosing a fundamental domain for $T^2$ allows us to order the $\alphas$ bottom-to-top and the $\betas$ left-to-right. In a fundamental domain, each $\alpha_i$ and $\beta_j$ intersect $p$ times. Other than in the definition of a twisted toroidal grid diagram, when we refer to an indexed curve $\alpha_i\in\alphas$ or $\beta_j\in\betas$, this index is with respect to the ordering induced by choice of fundamental domain. 

In depicting these fundamental domains, we will often use a different model, which we will call the \emph{straight model}. This model is bounded by $y=0$, $x=0$, $y=1$, and $x=1$. The left and right boundary are identified directly, and the top and the bottom are identified with a twist: $(x,0)\sim(x+p/q\pmod{1}, 1)$. Morally, this model is just straightening out the previously described fundamental domain, while remembering the underlying twisting. 

Figure~\ref{fig:fundomain} depicts an example of a choice of fundamental domain, as well as the straight model for it. 

\begin{figure}[h]
    \centering
    \begin{tikzpicture}
    \node at (0,0) {\includegraphics[height=6cm]{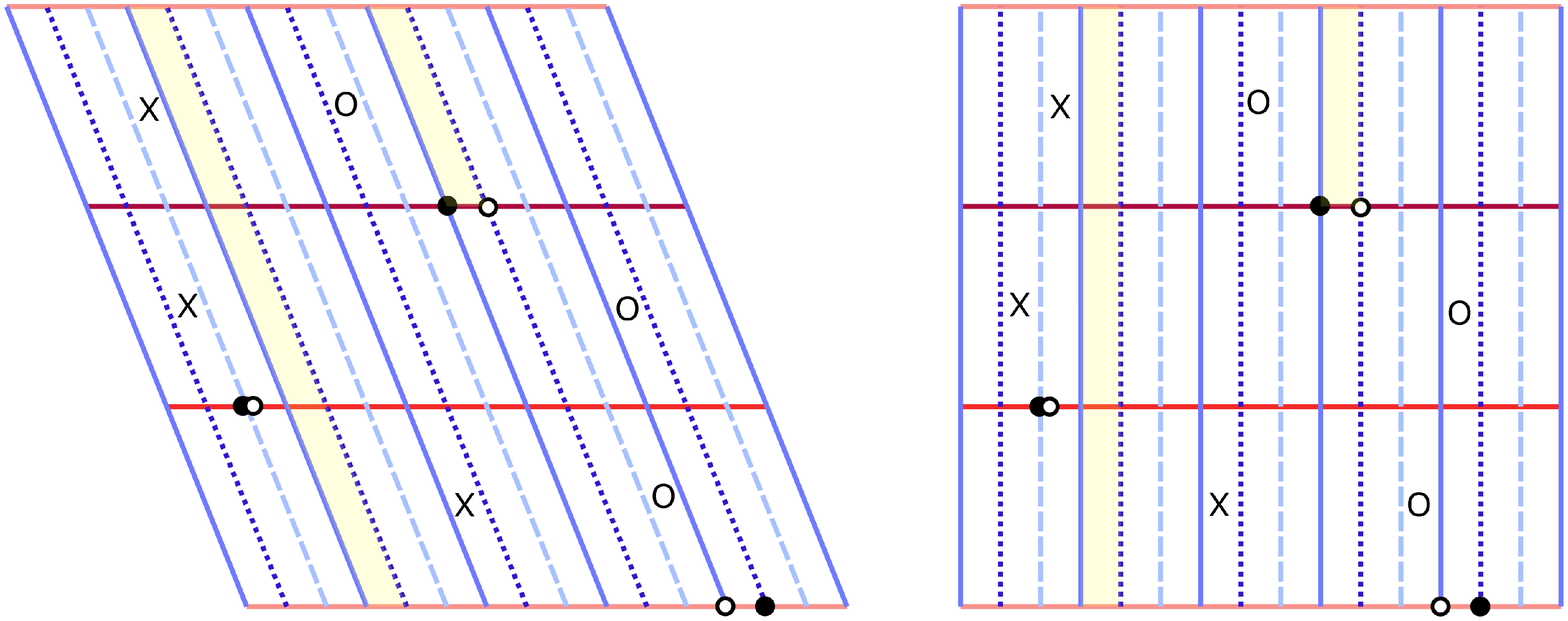}};
    \node at (-6.3,-1) {$\alpha_1$};
    \node at (-7.1,1) {$\alpha_2$};
    \node at (-5.55,-2.85) {$\alpha_0$};
    \node at (-5.2,-3.12) { $\beta_0$};
    \node at (-4.8,-3.12) { $\beta_1$};
    \node at (-4.4,-3.12) { $\beta_2$};
    \node at (1.8,-3.12) { $\beta_0$};
    \node at (2.2,-3.12) { $\beta_1$};
    \node at (2.6,-3.12) { $\beta_2$};
    \node at (1.45,-1) {$\alpha_1$};
    \node at (1.45,1) {$\alpha_2$};
    \node at (1.45,-2.85) {$\alpha_0$};
    \end{tikzpicture}
    \caption{A fundamental domain for $\text{L}(5,2)$ with grid number 3, and the corresponding straight model.}
    \label{fig:fundomain}
\end{figure}

We also note that without loss of generality, we assume the \emph{markings}, the points $X_i\in\XX$ and $O_i\in\OO$, lie in the center of the region of $T^2-\alphas-\betas$ which they occupy.

%% file: Sections/gridhomology.tex
In \cite{Baker_2008}, to a grid diagram $\gridg$ the authors associate a module $C^-(\gridg)$ and an endomorphism $\dxminus$ on $C^-(\gridg)$. They show that $(C^-(\gridg),\dxminus)$ is quasi-isomorphic to the knot Floer complex of \cite{Ozsvath_2008}. This equivalence is powerful: it proves that the endomorphism is a differential, and that the homology with respect to the differential is an invariant of the link. In \cite{celoria2015note}, the author provides a combinatorial proof that the endomorphism is a differential. 

After recalling the definition of the chain complex associated to a grid diagram, we will recreate the combinatorial proof from \cite{celoria2015note} that it is indeed a chain complex, as the case analysis will prove fruitful in later proofs. A combinatorial proof of the fact that the homology is a link invariant is the content of Section~\ref{sec:invariance}. 

\subsection{Preliminaries}

Let $L$ be a link with $\ell$ components in $L(p,q)$, and let $\gridg$ be a grid diagram of size $n$ for $L$. Define $\genG$ to be the set of unordered $n$-tuples of intersection points in $\alphas\cap\betas$ with exactly one point on each curve in $\alphas$ and exactly one point on each curve in $\betas$. Let $\x\in\genG$. Given a fundamental domain for $T^2$, and thus an ordering of $\alphas$ and $\betas$, $\x$ corresponds to a bijection between $\alphas$ and $\betas$, which can be thought of as an element $\sigma_{\x}\in S_n$; namely, each component of $\x$ is in $\alpha_i\cap \beta_{\sigma_{\x}(i)}$ for some $0\le i<n$. Denote by $x_i$ the component of $\x$ in the intersection of $\alpha_i$ and $\beta_{\sigma_{\x}(i)}$. 

The choice of a fundamental domain also gives an ordering of the $p$ points of $\alpha_i\cap\beta_{\sigma_{\x}(i)}$ from left-to-right for each $0\le i <n$. To each component $x_i$ of $\x$ we can assign an element of $\Zp$ recording on which intersection point of $\alpha_i$ and $\beta_{\sigma_{\x}(i)}$ the component lies. 

Thus, a choice of a fundamental domain for $T^2$ gives us a bijection between $\genG$ and $S_n\times(\Zp)^n$. For a generator $\x\in\genG$ and a choice of a fundamental domain, we can then write $\x$ as a pair $(\sigma_{\x},(a^{\x}_0,\ldots,a^{\x}_{n-1}))$. We call the first component the \emph{permutation associated to $\x$}, and the second component the \emph{$p$-coordinates of $\x$}. As an example, the generator given by darkened circles in Figure~\ref{fig:fundomain} has permutation $\sigma_{\x}=  \bigl(\begin{smallmatrix}
    0 & 1 & 2 \\
    1 & 2 & 0
  \end{smallmatrix}\bigr)$, and $p$-coordinates $(4,0,3)$. 

\begin{defn}
Fix $\x,\y\in\genG$. A \emph{parallelogram $r$ from $\x$ to $\y$} is a disk embedded in $T^2$ along with the data of the generators $\x$ and $\y$ such that \begin{itemize}
    \item the boundary of $r$ (which is oriented) lies on $\alphas\cup\betas$,  
    \item $\mid\x\setminus\y\mid=2=\mid\y\setminus\x\mid$, and 
    \item $\partial(\partial_{\alpha} (r))=\y-\x$ and $\partial(\partial_{\beta}(r))=\x-\y$, where $\partial_{\alpha}(r)=\partial(r)\cap\alphas$ and $\partial_{\beta}(r)=\partial(r)\cap\betas$. 
\end{itemize}
\end{defn}

There is also the more general idea of a domain, a generalization of the above concept. 

\begin{defn}[{\cite[Definition 4.6.4]{gridhomology}}]\label{def:domain1}
Fix $\x,\y\in\genG$. A \emph{domain $\phi$ from $\x$ to $\y$} is a formal linear combination of the closure of regions in $\gridg\setminus(\alphas\cup\betas)$ such that $\partial(\dalpha\phi)=\y-\x$ and $\partial(\dbeta\phi)=\x-\y$, along with the data of the generators $\x$ and $\y$. We will call $\x$ the \emph{initial generator} for the domain, and $\y$ the \emph{terminal generator}. The \emph{support} of a domain $\phi$ from $\x$ to $\y$ is the underlying linear combination. Thus, two domains can have the same support without being equal. 
\end{defn}

Given a domain $\phi_1$ from $\x$ to $\y$ and a domain $\phi_2$ from $\y$ to $\z$, we can form the \emph{juxtaposition} $\phi_1*\phi_2$, which is a domain as follows. The underlying formal linear combination of $\phi_1*\phi_2$ is the sum of those underlying $\phi_1$ and $\phi_2$, and it has as initial generator $\x$ and terminal generator $\z$. In general, the domains we consider will be juxtapositions of two specified types of domains, such as parallelograms or other polygons we will define later. 

Let $\PG(\x,\y)$ be the space of parallelograms from $\x$ to $\y$. Note that $\PG(\x,\y)$ is the empty set unless $\x$ and $\y$ agree at $n-2$ points. Call a parallelogram \emph{empty} if $\Int(r)\cap \x=\emptyset=\Int(r)\cap\y$. Let $\ePG(\x,\y)$ be the space of empty parallelograms from $\x$ to $\y$. Further, let $\Int(\phi)$ denote the interior of the support of $\phi$. Finally, let $\PG(\gridg):=\bigcup_{\x,\y\in\genG}\PG(\x,\y)$, and $\ePG(\gridg):=\bigcup_{\x,\y\in\genG}\ePG(\x,\y)$. 

\begin{remark} \label{rem:parallelogramtransposition}
Note that if an embedded parallelogram $r$ connects $\x$ to $\y$, then $\sigma_{\y}=\tau_{i,j}\sigma_{\x}$, where $\tau_{i,j}$ is the transposition $(i\; j)\in S_n$. Furthermore, we have that $a^{\x}_s=a^{\y}_s$ for all $s\in\{0,\ldots,n-1\}\setminus\{i, j\}$, and $a^{\x}_i-a^{\y}_i=a^{\y}_j-a^{\x}_j=k$ for some $|k|<n$. 
\end{remark}

For an embedded parallelogram $r\in\ePG(\x,\y)$, denote the above transposition by $\tau_r$.

An example of a parallelogram $r$ can be seen in Figure~\ref{fig:fundomain}, connecting the generator given by darkened circles to the generator given by empty circles. In this case, we see $\tau_r=(0\;1)$. 

\subsection{Aside on twisting}\label{sec:twisting}
Before we provide the definition of grid homology, we take a minute to discuss the impact of the twisting of the grid diagrams on the domains we will be discussing. In the case of grid diagrams for $S^3$, it is not possible for an empty rectangle to intersect a curve $\alpha\in\alphas$ multiple times; the rectangle either intersects $\alpha$ in a single segment or not at all. As a result, we can always choose a fundamental domain such that any rectangle embedded in $T^2$ can be embedded in the fundamental domain considered as a subset of $\R^2$. 

This is no longer true for twisted toroidal grid diagrams for links in lens spaces: there are parallelograms which intersect a given curve $\alpha\in\alphas$ in multiple segments. There is no choice of fundamental domain in which they will be an embedded disk, considered as a subset of $\R^2$. We refer to such parallelograms as \emph{wrapping} around the domain, in reference to the fact that their support on any fundamental domain, considered as a subset of $\R^2$, is disconnected.

Figure~\ref{fig:wrappingexplained} provides some examples of parallelograms on the twisted toroidal grid diagrams. We see that there are many options for the parallelogram extending vertically: it can wrap any number of times around the diagram (less than $p$). Later on, when analyzing the combinatorics of decompositions of domains, we will often not display this potential complexity. We will draw domains as embedded, wrapped, or with dashed lines to acknowledge potential wrapping. Regardless of how the domain is drawn, it will stand in for all domains with the same combinatorics governing its decompositions, which will be independent of the number of times the domain wraps around the diagram. 

\begin{figure}[ht]
    \centering
    \includegraphics[height=4cm]{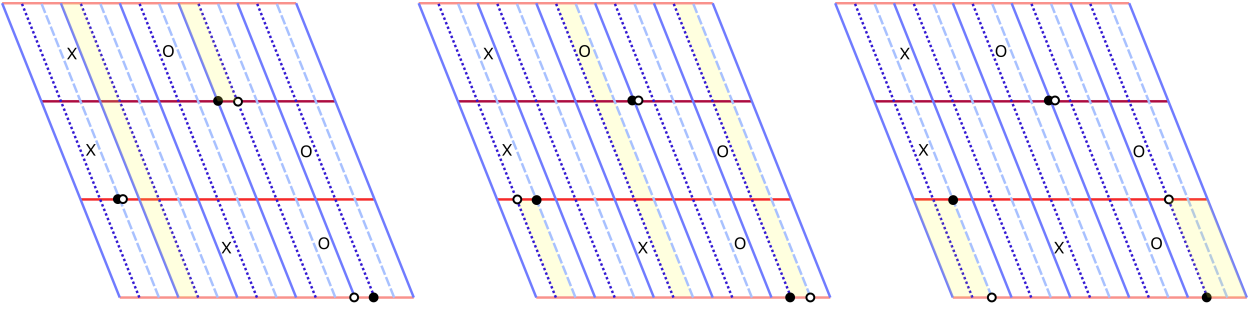}
    \caption{Examples of parallelograms on fundamental domains. The first and second examples wrap around the diagram one and two times respectively. The third example does not wrap around the diagram, as there are many examples of fundamental domains on which the parallelogram would be embedded (as a subset of $\R^2$).  }
    \label{fig:wrappingexplained}
\end{figure}

\subsection{Definition of grid homology}

Let $C^-(\gridg)$ be the free $\F[V_0,\ldots,V_{n-1}]$ module generated by $\genG$. Define an endomorphism of $C^-(\gridg)$ on generators $\x\in\genG$ by:$$
\dxminus(\x)=\sum_{y\in\genG}\sum_{\substack{r\in\ePG(\x,\y)\\r\cap\XX=\emptyset}}(\prod_{i=0}^{n-1}V_i^{n_{O_i}(r)})\y,
$$
where $n_{O_i}(r)=\#(O_i\cap\Int(r))$. Extend this map to $C^-(\gridg)$ linearly; note then that $\dxminus(V_i\x)=V_i\dxminus(\x)$. 

\begin{prop}[{\cite[Proposition 2.2]{Baker_2008}}]\label{prop:cminusiscfk}
$(C^-(\gridg),\dxminus)$ is isomorphic to a chain complex which computes the knot Floer homology $(CF^-(\lpq,K),\dxminus)$. 
\end{prop}

\begin{cor}
The endomorphism $\dxminus$ is a differential on $C^-(\gridg)$. 
\end{cor}

The above corollary was proved combinatorially in \cite{celoria2015note}. We recreate that proof here in detail to lay out cases and notation which will be useful in later sections. 

\begin{prop}[{\cite[Proposition 2.13]{celoria2015note}}]\label{prop:d2=0}
The map $\dxminus$ is a differential on $C^-(\gridg)$, i.e.\ $(\dxminus)^2=0$. 
\end{prop}

\begin{remark}\label{rem:omultiplicity}
Note that if a domain $\phi$ from $\x$ to $\z$ has a decomposition as $\phi=r_1*r_2$, for $r_1\in\PG(\x,\y)$ and $r_2\in\PG(\y,\z)$, we have that $n_{O_i}(\phi)=n_{O_i}(r_1)+n_{O_i}(r_2)$ for any $0\le i <n$. 
\end{remark}

\begin{proof}[Proof of Proposition~\ref{prop:d2=0}]
This is the first of a number of proofs where we analyze the composition of two maps by analyzing domains which can be written as the juxtaposition of two specified polygons.

Let $\ePoly(\x,\z)$ denote the set of domains from $\x$ to $\z$ which can be written as the juxtaposition of two empty parallelograms. For $\phi\in\ePoly(\x,\z)$, let $N(\phi)$ denote the number of decompositions of $\phi$ as the juxtaposition of two empty parallelograms. 

We see that $$(\dxminus\circ\dxminus)(\x)=\sum_{z\in\genG}\sum_{\phi\in\ePoly(\x,\z)}N(\phi)(\prod_{i=0}^{n-1}V_i^{n_{O_i}(\phi)})\y. 
$$

Let $\x,\z\in\genG$ and let $\phi\in\ePoly(\x,\z)$, so $\phi=r_1*r_2$ for some $r_1\in\ePG(\x,\tgen)$, $r_2\in\ePG(\tgen,\z)$, $\tgen\in\genG$. Let $m$ be the number of components where $\x$ and $\z$ differ, and observe that $m\le 4$. For $2\le m\le 4$ we will produce the unique alternate decomposition of $\phi$ as the juxtaposition of two empty parallelograms, proving $N(\phi)=2$. 

Suppose $m=4$. There are two possible cases: either the supports of $r_1$ and $r_2$ are disjoint, or have overlapping interiors but share no corners or edges (except for transverse edge intersections). These two cases are shown in Figure~\ref{fig:dsquared4} (portrayed in schematic in the straight model). In either case, we produce an alternate decomposition of $\phi$ as follows. Let $r_1'$ have the same support as $r_2$, with initial generator $\x$. Let the terminal generator of $r_1'$, denoted $\w$, agree with $\x$ away from the support of $r_1'$ but occupy the complementary two corners of the support of $r_1'$ as $\x$; so $r_1'\in\ePG(\x,\w)$. Let $r_2'$ have the same support as $r_1$, but with initial generator $\w$ and terminal generator $\z$. Thus, $\phi=r_1'*r_2'$ is the unique alternate decomposition of $\phi$, implying $N(\phi)=2$. 

\begin{figure}[ht]
    \centering
    \includegraphics[height=\figheight]{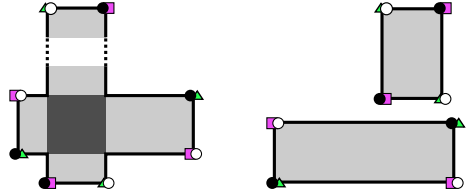}
    \caption{Domains counted in the proof of Proposition~\ref{prop:d2=0} when the initial and terminal generators disagree at four components. Throughout this paper, we use the convention that the initial generator of the domain $\phi$ is given by darkened circles, the terminal generator by empty circles, and the \emph{intermediate generators}, which are the initial or terminal generators of any domain included in a decomposition of $\phi$ (except for the initial and terminal generators of $\phi$ itself), are given by pink squares or green triangles.}
    \label{fig:dsquared4}
\end{figure}

Suppose $m=3$. We will construct the unique alternate decomposition of $\phi$, but this will require introducing some notation. Figure~\ref{fig:dsquared3detail} shows the below notation, which details one of the four combinatorially distinct cases of Figure~\ref{fig:dsquared3}. 

The intermediate generator $\tgen$ must have one component on $\partial(\phi)$ which does not coincide with $\x$ or $\z$. Suppose this component is on an arc in $\dalpha(\phi)$, and denote the component by $d_{\alpha}$. The arc of $\dalpha(\phi)$ which $d_{\alpha}$ occupies terminates in one component of $\x$ and one component of $\z$; denote these by $x_a$ and $z_b$. 

Denote the curve in $\betas$ on which $d_{\alpha}$ lies on by $\beta_j$. Traveling from $d_{\alpha}$ along $\beta_j$ into $\phi$, we reach a corner of $\phi$ which we will denote $d$. Denote the curve in $\alphas$ on which $d$ lies by $\alpha_i$. The point $d$ is characterized by the domain $\phi$ having multiplicity one in three of the four regions surrounding it, and multiplicity zero in the fourth. If we continue following $\beta_j$ along $\dbeta(\phi)$ we reach another corner of $\phi$, a component of $\x$ or $\z$. Call this component $w_1$. If instead at $d$ we travel along $\alpha_i$ into $\phi$, we reach a point we will denote $d_{\beta}$ on $\dbeta(\phi)$. This point is on the same arc of $\dbeta(\phi)$ as either $x_a$ or $z_b$; let $w_2$ be whichever of those two points does not lie on the same arc of $\dbeta(\phi)$ as $d_{\beta}$. 

We now let $\w$ agree with $\tgen$ away from $\phi$, and contain the components $w_1$, $w_2$, and $d_{\beta}$. Then one summand of the support of $\phi$ is the support of an embedded parallelogram $r_3\in\ePG(\x,\w)$, and the complement of that summand is the support of an embedded parallelogram $r_4\in\ePG(\w,\z)$, providing the unique alternate decomposition of $\phi$ as $\phi=r_3*r_4$. 

If, instead, the component of $\tgen$ which is distinct from $\x$ and $\z$ lies on $\dbeta(\phi)$, then interchanging the role of $\alpha$ and $\beta$ in the above proof will generate the unique alternate decomposition. Either way, $N(\phi)=2$. 

\begin{figure}[ht]
    \centering
    \includegraphics[height=\figheight]{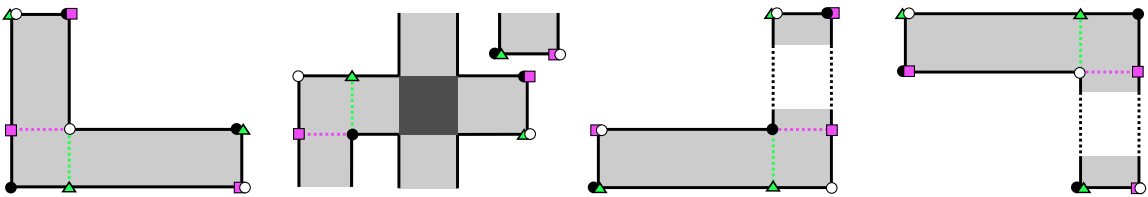}
    \caption{Domains counted in the proof of Proposition~\ref{prop:d2=0} when the initial and terminal generators disagree at three components. See Remark~\ref{rem:wrappingdiagrams} for a discussion of the schematics used here with respect to wrapping. }
    \label{fig:dsquared3}
\end{figure}
\begin{figure}[ht]
\begin{center}
    \begin{tikzpicture}
    \node at (0,0) {\includegraphics[height=\figheight]{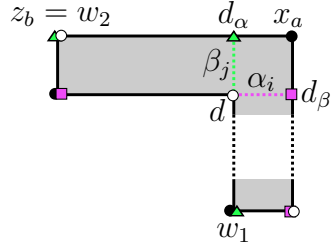}};
    \node at (.8,1.45) {$d_{\alpha}$};
    \node at (1.53,1.4) {$x_a$};
    \node at (-1.5,1.42) {$z_b=w_2$};
    \node at (.58,.78) {$\beta_j$};
    \node at (.58,.18) {$d$};
    \node at (1.17,.55) {$\alpha_i$};
    \node at (.8,-1.45) {$w_1$};
    \node at (1.9,.37) {$d_{\beta}$};
    \end{tikzpicture}
\end{center}
    \caption{An example of the notation in constructing an alternate decomposition for a domain where the initial and terminal generators disagree in three components}
    \label{fig:dsquared3detail}
\end{figure}

\begin{remark}\label{rem:wrappingdiagrams}
In light of Section~\ref{sec:twisting}, let us clarify what the schematics in the figures throughout this paper mean. Each domain, regardless of whether it is displayed as embedded, wrapping, or with dashed lines, stands in for all such domains with the same combinatorics of alternate decompositions. The wrapping does not impact the combinatorics, so this is a reasonable simplification. This wrapping can be disjoint from the rest of the domain, or overlap on the interior with transverse edge intersections. For example, in Figure~\ref{fig:dsquared3}, we display the first domain as embedded, the second domain as wrapping once, and the third and fourth with dashed lines. In practice, any of those domains could wrap any number of times (fewer than $p$) without impacting the combinatorics of the decompositions.  
\end{remark}

Suppose $m=2$. Recall that in grid diagrams for $S^3$, there are no such domains which admit a decomposition as the juxtaposition of two empty rectangles. Remark~\ref{rem:parallelogramtransposition} makes clear, however, that for twisted toroidal grid diagrams, this is possible, as generators are not identified with elements of $S_n$, but with elements of $S_n\times(\Zp)^n$. In either setting, we have that $\sigma_{\x}=\tau_{r_1}\tau_{r_2}\sigma_{\z}$. The only way for $\sigma_{\x}$ and $\sigma_{\z}$ to differ by one or fewer transpositions, then, is for $\tau_{r_1}=\tau_{r_2}$, so $\sigma_{\x}=\sigma_{\z}$. The parallelograms, however, can shift the $p$-coordinates by different amounts, so in the end, $\sigma_{\x}=\sigma_{\z}$, but they differ in two $p$-coordinates. See Figure \ref{fig:dsquared2} for the possible domains. 

Again we will construct the unique alternate decomposition of $\phi$. The intermediate generator $\tgen$ must have two components on $\partial(\phi)$ which do not coincide with $\x$ or $\z$; suppose they are on arcs in $\dalpha(\phi)$, and denote these components by $d_{\alpha}$ and $d_{\alpha}'$. 

Following the curves in $\betas$ on which these components lie into $\phi$ we reach two corners of $\phi$ which we denote $d$ and $d'$. These points are the two points around which the domain has multiplicity one in three regions and multiplicity zero in the fourth. From such points, traveling into $\phi$ on the curves in $\alphas$ which $d$ and $d'$ occupy, we reach two points on $\dbeta(\phi)$ which we denote $d_{\beta}$ and $d_{\beta}'$. Let $\w$ agree with $\tgen$ away from $\phi$, and include $d_{\beta}$ and $d_{\beta}'$. Then one summand of the support of $\phi$ is the support of an embedded parallelogram $r_3\in\ePG(\x,\w)$ and the complement of that summand is the support of an embedded parallelogram $r_4\in\ePG(\w,\z)$, providing the unique alternate decomposition of $\phi$ as $\phi=r_3*r_4$.

If, instead, the components of $\tgen$ which are distinct from $\x$ and $\z$ lie on $\dbeta(\phi)$, then interchanging the role of $\alpha$ and $\beta$ in the above proof will generate the unique alternate decomposition. Either way, $N(\phi)=2$.

\begin{figure}[ht]
    \centering
    \includegraphics[width=12cm]{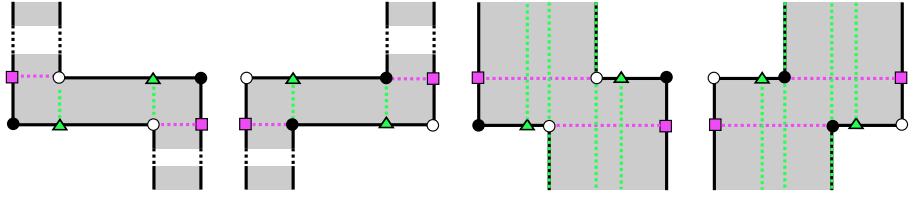}
    \caption{Domains counted in the proof of Proposition~\ref{prop:d2=0} when the initial and terminal generators disagree at two components.}
    \label{fig:dsquared2}
\end{figure}

By virtue of Remark \ref{rem:parallelogramtransposition}, there are no domains with $m=1$. 

Finally, suppose $\x=\z$. The support of $\phi$ must be an annulus. It must be width or height one, by virtue of $\phi$ being empty. The conditions on a grid diagram, however, ensure there is at least one $X\in\XX$ in each row and column, so there are no such annuli counted. 

Thus in each case, we have $N(\phi)\equiv 0\pmod{2}$, so $(\dxminus)^2=0$. 
\end{proof}

While $C^-(\gridg)$ is a chain complex over $\Ztwo[V_0,\ldots,V_{n-1}]$, the link invariant is the quasi-isomorphism type of $C^-(\gridg)$ as a chain complex over $\Ztwo[U_1,\ldots,U_\ell]$, where each variable corresponds to the action of a distinct component of the link. To see this is well-defined, we need the following lemma. 

\begin{lem}[cf. {\cite[Lemma 2.11]{Manolescu_2007}}]
\label{lem:componentvariablesarehomotopic}
Suppose $O_i$ and $O_j$ are markings corresponding to the same component of $L$. Then multiplication by $V_i$ is chain homotopic to multiplication by $V_j$. 
\end{lem}
\begin{proof}
Without loss of generality, suppose $O_i$ is in the same row as $X_a$, which is in the same column as $O_j$. Define $\Phi_{X_a}:C^-(\gridg)\to C^-(\gridg)$ on a generator $\x\in\genG$ by \[
\Phi_{X_a}(\x)=\sum_{\y\in\genG}\sum_{\substack{r\in\ePG(\x,\y)\\\Int(r)\cap\XX=X_a}}V_o^{n_{O_0}(r)}\cdots V_{n-1}^{n_{O_{n-1}}(r)}\y.
\] Extend this linearly to $C^-(\gridg)$. 

We will show that this is the homotopy between $V_i$ and $V_j$. Our goal is to show that $\dxminus\circ\Phi_{X_a}+\Phi_{X_a}\circ\dxminus = V_i-V_j$. The left hand side counts domains which are the juxtapositions of two embedded parallelograms, one of which contains $X_a$. Any such domain has an alternate decomposition given in the proof of Proposition~\ref{prop:d2=0}, regardless of the placement of $X_a$. In that proof we do not need to consider the domains which connect $\x\in\genG$ to $\z=\x\in\genG$, because they consist of a single row or a single column, either of which contains some marking $X\in\XX$. Here, however, domains are allowed to cover $X_a\in\XX$, so the left hand side above counts the annulus of width one covering $X_a$ and the annulus of height one covering $X_a$. Each of these connects $\x$ to $\x$; the row contributes a coefficient of $V_i$ and the column contributes a coefficient of $V_j$, proving the above equality. 

If $O_i$ and $O_j$ are not connected by a single marking $X_a$, then a sequence of the above homotopies, which can be found since $O_i$ and $O_j$ correspond to the same component, proves the actions of $V_i$ and $V_j$ are homotopic. 
\end{proof}

%% file: Sections/gradings.tex
Grid homology as defined comes with three gradings \cite{Baker_2008}. After recalling the definitions of these below, we will prove directly that each of these is well-defined, i.e. independent of the choice of a fundamental domain on which to compute them. Like the fact that $\dxminus$ is a differential, this is proved in \cite{Baker_2008} by appealing to the correspondence with the holomorphic theory. 

Throughout, let $\gridg$ be a grid diagram for an $\ell$-component link $L$.

\subsection{\texorpdfstring{$\spinc$ grading}{spinc grading}}
\input{Sections/Gradings/spinc}

\subsection{Maslov grading}
\input{Sections/Gradings/maslov}

\begin{figure}[ht]
    \centering
    \includegraphics[height=6cm]{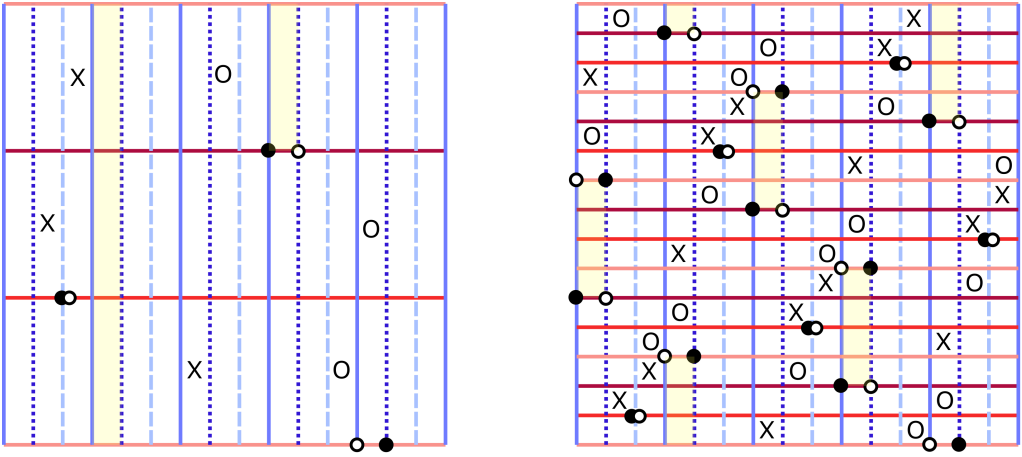}
    \caption{An example of an embedded parallelogram $r$ connecting a generator $\x\in\genG$ (represented by solid dots) to a generator $\y\in\genG$ (represented by empty dots) in the straight model on the left. On the right, the lift of the left to the universal cover. }
    \label{fig:differentialgradings}
\end{figure}

\subsection{Alexander grading}
\input{Sections/Gradings/alexander}

\subsection{Embedded parallelograms}
The above gradings behave well with respect to embedded parallelograms. This was proved in \cite{celoria2015note}, but we recreate the proof here slightly more generally, as we will adapt it to study embedded pentagons later. 

\begin{prop}[cf.~{\cite[Proposition 2.13]{celoria2015note}}]\label{prop:parallelogramgradings}
Suppose $\x,\y\in\genG$ are connected by an embedded parallelogram $r\in \ePG(\x,\y)$. Then:\begin{enumerate}
    \item\label{prop:parallelogramgradingsspinc} $\spincS(\x)=\spincS(\y)$
    \item\label{prop:parallelogramgradingsmaslov} $\maslov(\x)=\maslov(\y)+1-2\sum_{i=0}^{n-1}n_{O_i}(r)$
    \item\label{prop:parallelogramgradingsalexander} $\alex(\x)=\alex(\y)+\sum_{i=0}^{n-1}n_{X_i}(r)-\sum_{i=0}^{n-1}n_{O_i}(r)$. 
\end{enumerate}
\end{prop}

\begin{proof}
To prove (\ref{prop:parallelogramgradingsspinc}), recall that Remark~\ref{rem:parallelogramtransposition} tells us if $\x$ and $\y$ are connected by an empty parallelogram $r$, then $a^{\x}_k=a^{\y}_k$ for all but at most two $p$-coordinates $0\le k<n$. Further, if $\tau_r=(i\; j)$, then $a^{\x}_i-a^{\y}_i=a^{\y}_j-a^{\x}_j$. Plugging these values in to the definition of $\spincS$ tells us that $\spincS(\x)=\spincS(\y)$. 

To simplify notation in the proof of (\ref{prop:parallelogramgradingsmaslov}), let $s=\sum_{i=0}^{n-1}n_{O_i}(r)$. One could hope that the Maslov grading behaves on each of the $p$ strips as it does on a grid diagram for $S^3$, so we would get\begin{align*}
    \I(\tx,\tx)&=\I(\ty,\ty)+p\\
    \I(\tx,\OO)&=\I(\ty,\OO)+ps\\
    \I(\OO,\tx)&=\I(\OO,\tx)+ps.
\end{align*} 
Sadly it does not, as Figure~\ref{fig:differentialgradings} demonstrates. Instead of trying to relate the Maslov gradings of $\x$ and $\y$ one strip at a time, we will do it one empty parallelogram at a time, in the following sense.

Fix a fundamental domain for $T^2$, and consider $\tx$ and $\ty$, which are connected by $p$ parallelograms, denoted $\tilde{r}=\{r_1,\ldots,r_p\}$ as in Figure~\ref{fig:differentialgradings}. Let $\tx^i$ denote the result of applying $\{r_1,\ldots,r_i\}$ to $\tx$, so $\tx^0=\tx$, and $\tx^p=\y$. Note that the $\OO$ and $\XX$ markings are fixed. Finally, we will abuse notation to write $\tilde{\maslov}(\tx^i)=\I(\tx^i,\tx^i)-\I(\tx^i,\tOO)-\I(\tOO,\tx^i)+\I(\tOO,\tOO)$. 

We will tackle three cases, depending on how $r_i$ sits on the fundamental domain. Suppose $r_i$ does not wrap around the diagram, i.e. it is an embedded disk in the fundamental domain as a subset of $\R^2$.  Then $\I(\tx^{i-1},\tx^{i-1})=\I(\tx^i,\tx^i)+1$. Furthermore, we have $\I(\tx^{i-1},\tOO)=\I(\tx^i,\tOO)+s$, and similarly $\I(\tOO,\tx^{i-1})=\I(\tOO,\tx^{i})+s$. As a result, $\tilde{\maslov}(\x^{i-1})=\tilde{\maslov}(\x^i)+1-2s$. 

Suppose on the fundamental domain $r_i$ wraps horizontally around the diagram, i.e. is the disjoint union of two embedded disks which border the curves in $\betas$ which have coordinates $0$ and $pn$. Each of these embedded disks covers the same number of rows (consider on the fundamental domain as a subset of $\R^2$), which we will denote $l$.

Each row must have a component of $\tx^{i-1}$ on it, and since the rectangle is empty, there must be $l-1$ components in the complement of $r_i$ in the horizontal annulus containing those $l$ rows. Each such component counts twice towards $\I(\tx^i,\tx^i)$ but not at all towards $\I(\tx^{i-1},\tx^{i-1})$. Furthermore, the corners of $r_i$ contribute once to $\I(\tx^i,\tx^i)$ but not to $\I(\tx^{i-1},\tx^{i-1})$. Therefore, we get $\I(\tx^{i-1},\tx^{i-1})=\I(\tx^i,\tx^i)-1-2(l-1)$. 

Similarly, there is one marking per row, so there are $l$ markings contained in the horizontal annulus containing those $l$ rows. If a marking is contained in $r_i$, it contributes equally to $\I(\tx^i,\tx^i)$ and $\I(\tx^{i-1},\tx^{i-1})$. There are $l-s$ markings contained in the annulus but not in $r_i$; each contributes to $\I(\tx^i,\tOO)$ and $\I(\tOO,\tx^i)$ but not to $\I(\tx^{i-1},\tOO)$ or $\I(\tOO,\tx^{i-1})$, so we get \begin{align*}\tilde{\maslov}(\tx^{i-1})&=(\I(\tx^i,\tx^i)-1-2(l-1))-(\I(\tx^i,\tOO)-(l-s))-(\I(\tOO,\tx^i)-(l-s))+\I(\tOO,\tOO)\\
&=\tilde{\maslov}(\tx^i)+1-2s
.\end{align*}

Suppose on the fundamental domain $r_i$ wraps vertically around the diagram,  i.e. is the disjoint union of two embedded disks which border the curves in $\alphas$ which have coordinates $0$ and $n$. Then the analogous argument gives us that again $\tilde{\maslov}(\tx^{i-1})=\tilde{\maslov}(\tx^i)+1-2s$. 

If we choose a fundamental domain such that some rectangle has lower left corner at $(0,0)$, we ensure that no rectangle wraps both horizontally and vertically. Thus we have that $\tilde{\maslov}(\x)=\tilde{\maslov}(\tx^0)=\tilde{\maslov}(\tx^p)+p-2ps=\tilde{\maslov}(\y)+p-2ps$, so $\maslov(\x)=\maslov(\y)+1-2s$, as desired.

We can conclude (\ref{prop:parallelogramgradingsalexander}) similarly. We see \[\maslov_{\XX}(\x)=\maslov_{\XX}(\y)+1-2\sum_{i=0}^{n-1}n_{X_i}(r)\] by a similar argument, and combining this with Definition~\ref{eqn:alexdefn} gives us that \begin{equation*}\alex(\x)=\alex(\y)+\sum_{i=0}^{n-1}n_{X_i}(r)-\sum_{i=0}^{n-1}n_{O_i}(r).\qedhere\end{equation*}

\end{proof}

%% file: Sections/Gradings/spinc.tex
The $\spinc$ grading takes values in $\Zp$. Fix a fundamental domain for $T^2$. Let $\x_{\OO}$ be the generator with components occupying the lower-left corner of each region of $T^2-\alphas-\betas$ containing an $O$ marking. For any generator $\x\in\genG$, let $\tilde{\spincS}(\x)=\sum_{i=0}^{n-1}a^{\x}_i-a^{\x_{\OO}}_i\pmod{p}$, and let $\spincS(\x)=\tilde{\spincS}(\x)+q-1\pmod{p}$. Extend this to a grading on all of $C^-(\gridg)$ by setting $\spincS(V_i\x):=\spincS(\x)$. 

As an example, consider the generator given in Figure~\ref{fig:fundomain} by darkened circles, and denote it by $\z$. We noted earlier it has $p$-coordinates $(4,0,3)$. We can see on the diagram that $\x_{\OO}$ has $p$-coordinates $(3,4,2)$. We compute $\tilde{\spincS}(\z)=(4-3)+(0-4)+(3-2)\pmod{3}=3$, so $\spincS(\z)=4$. 

\begin{lem}
The function $\spincS$ is well-defined. 
\end{lem}
\begin{proof}
Consider the fundamental domain given by moving up by $1/n$ in the $y$-direction, which corresponds to vertical rotation of the rows. This will not change the $p$-coordinates of $\x_{\OO}$ or $\x$, so we get that $\spincS$ is well-defined under vertical rotation. 

Consider the fundamental domain given by moving right by $1/pn$ in the $x$-direction, which corresponds to horizontal rotation of the columns. This shift will decrease both $a^{\x}_{\sigma^{-1}(0)}$ and $a_{\sigma^{-1}(0)}^{\x_{\OO}}$ by one, and preserve all the other $p$-coordinates, so will preserve $\spincS(\x)$. 
\end{proof}

%% file: Sections/Gradings/maslov.tex
We will follow \cite{Baker_2008}. To define the \emph{Maslov grading} (and the \emph{Alexander grading}, later), we need several auxiliary functions. It is unsurprising that we need to give coordinates to our generators to compute combinatorial gradings, and the first two functions do that. The third function compares sets of coordinate pairs. 

Fix a fundamental domain for $T^2$, with lower-left corner at $k/pn, l/n$. We use the function \begin{align*}W:\Big\{\text{finite sets}&\text{ of points in $\gridg$}\Big\}\\
&\to\Big\{\text{finite sets of pairs $(a,b)\in[0,pn)\times[0,n)$}\Big\}
\end{align*} which takes a set of points in $\gridg$ and computes its coordinates with respect to the basis for $\R^2$ given by  $\{\overrightarrow{v_1}=(\frac{1}{np},0),\overrightarrow{v_2}=(\frac{-q}{np},\frac{1}{n})\}$, then subtracts $k$ from each $x$ coordinate and $l$ from each $y$ coordinate. In effect, this computes coordinates as if the fundamental domain had lower left corner at the origin. 

We compute this above function on the example generator $\z$ from Figure~\ref{fig:fundomain} given by darkened circles. We have $W(\z)=\{(13,0),(2,1),(9,2)\}$. 

Next, we define a function \begin{align*}
C_{p,q}:\Big\{\text{finite sets}&\text{ of pairs $(a,b)\in[0,pn)\times[0,n)$}\Big\}\\
&\to\Big\{\text{finite sets of pairs $(a,b)\in[0,pn)\times[0,pn)$}\Big\}. 
\end{align*}

It is defined on an $n$-tuple of coordinates $\{(a_i,b_i)\}_{i=0}^{n-1}$ by  \[C_{p,q}(\{(a_i,b_i)\}_{i=0}^{n-1}):=
\{((a_i+nqk)\mod np,b_i+nk)_{i=0,k=0}^{i=n-1,k=p-1}\}.
\]

This corresponds to taking the coordinates of a generator $\x$ and lifting them to the coordinates of a generator $\tilde{\x}$ of the chain complex associated to $\tilde{L}$ in $S^3$, the lift of $L$ in the universal cover of $\lpq$. The effect of this function can be seen in Figure~\ref{fig:differentialgradings}.  

We will often conflate a generator $\x$ and its coordinates $W(\x)$, and we will do the same for $\tilde{\x}$ and $C_{p,q}(W(\x))$.

Note that the output of $C_{p,q}$ is $p$ horizontal strips of points, each with $y$-coordinates between $nk$ and $n(k+1)$ for $0\le k<p$. We will occasionally want to refer to each such strip, so will denote each by $C_{p,q}^k(\x):=\{((a_i+nqk)\pmod{np}, b_i+nk)\}_{i=0}^{i=n-1}$. 

For the generator $\z$ of Figure~\ref{fig:fundomain} given by darkened circles, we can compute algebraically or via inspection of Figure~\ref{fig:differentialgradings} that \begin{align*}C_{p,q}(W(\z))=\{&(13,0),(2,1),(9,2),(4,3),(8,4),(0,5),(10,6),(14,7),\\
&(6,8),(1,9),(5,10),(12,11),(7,12),(11,13),(3,14)\}.\end{align*}

The third and final of these auxiliary functions is $\I$, defined in \cite{Manolescu_2007}, which takes in two sets of coordinate pairs, $A$ and $B$. It is defined by \[
\I(A,B):=\#\{\big((a_1,a_2),(b_1,b_2)\big)\in A\times B:a_1<b_1\text{ and }a_2<b_2\}.
\]

With these auxiliary functions in hand, we can define a function $\tilde{\maslov}$ by \[
\tilde{\maslov}(\x)=\I(\tilde{\x},\tilde{\x})-\I(\tilde{\x},\tilde{\OO})-\I(\tilde{\OO},\tilde{\x})+\I(\tilde{\OO},\tilde{\OO}),\]
and the \emph{Maslov grading} on a generator by $\maslov(\x)=\frac{\tilde{\maslov}(\x)}{p}+d(p,q,q-1)+1$, where $d(p,q,q-1)$ is a recursively defined correction term whose definition can be found in \cite{Baker_2008}. 

As a grading, we can extend $\maslov$ to all of $C^-(\gridg)$ by declaring $\maslov(V_i\x)=\maslov(\x)-2$ for $0\le i<n$.

We can now compute the Maslov grading of the generator $\z$ of Figure~\ref{fig:fundomain}. We get that \begin{align*}
    \I(\tz,\tz)&=52\\
    \I(\tz,\tOO)&=55\\
    \I(\tOO,\tz)&=40\\
    \I(\tOO,\tOO)&=42.
\end{align*}

Asserting that $d(5,2,1)=-\frac{2}{5}$, we get that $\maslov(\z)=\frac{1}{5}(52-55-40+42)-\frac{2}{5}+1=\frac{2}{5}$. 

\begin{lem}\label{lem:maslov}
The function $\maslov$ is well-defined. 
\end{lem}

\begin{proof}
Fix a fundamental domain for $T^2$, and let $\x\in \genG$. Denote by $a$ the component of $\x$ with $W(a)=(r,0)$. We will denote by $\x'$ and $a'$ the same generator and component, but considered with respect to the fundamental domain given by moving up by $1/n$ in the $y$-direction.
Note that for all components of $\x=\x'$, the $x$-coordinate remains the same after this shift, and for all components except $a$, the $y$-coordinate decreases by 1. 
We see that $a'$ has coordinates $(r,n-1)$. 

For any $b\in\x$ with first coordinate less than $r$, we have that $(b,a')$ contributes to $\I(\x',\x')$ but $(b,a)$ does not contribute to $\I(\x,\x)$. Correspondingly, if $b$ has first coordinate greater than $r$, then $(a,b)$ contributes to $\I(\x,\x)$ but $(a',b)$ does not contribute to $\I(\x',\x')$. Thus we have that $\I(\x',\x')=\I(\x,\x)-(n-r-1)+r=\I(\x,\x)+2r+1-n$. In shifting up the fundamental domain by $1/n$, we have also changed the coordinates of one $O\in\OO$. If it originally had coordinates $(l+1/2,1/2)$, then a similar analysis yields $\I(\OO',\OO')=\I(\OO,\OO)+2l+1-n$. For the last two terms, we can break up the analysis into one more step to yield \begin{align*}
    \I(\x',\OO)&=\I(\x,\OO)-(n-r)\\
    \I(\x',\OO')&=\I(\x',\OO)+(l+1)\\
    \I(\x',\OO')&=\I(\x,\OO)+r+l+1-n,
\end{align*}
and 
\begin{align*}
    \I(\OO,\x')&=\I(\OO,\x)+r\\
    \I(\OO',\x')&=\I(\OO,\x')-(n-l-1)\\
    \I(\OO',\x')&=\I(\OO,\x)+r+l+1-n.
\end{align*}

In the end, we get equality in the sum:\begin{align*}
\I(\x',\x')-\I(\OO',\x')-\I(\x',\OO')+\I(\OO',\OO')&=(\I(\x,\x)+2r+1-n)-(\I(\OO,\x)+r+l+1-n)\\&-(\I(\x,\OO)+r+l+1-n)+(\I(\OO,\OO)+2l+1-n)\\
&=\I(\x,\x)-\I(\OO,\x)-\I(\x,\OO)+\I(\OO,\OO).
\end{align*}Further, this sum is well defined for any strip $C_{p,q}^k(\x)$ and $C_{p,q}^k(\OO)$. Finally, because the vertical shift only changes the $y$-coordinates of any point and not the $x$-coordinates, we get that two points in $\tilde{\x}\cup\tilde{\OO}$ from different strips contribute to $\tilde{\maslov}(\x)$ if and only if the corresponding points in $\tilde{\x'}\cup\tilde{\OO'}$ contribute to $\tilde{\maslov}(\x')$, as their $x$-coordinates are the same, and their $y$-coordinates satisfy the same inequality. Thus, $\tilde{\maslov}(\x')=\tilde{\maslov}(\x)$. 

Let's show the same is true for a horizontal shift of the fundamental domain to the right by $1/pn$. Instead of proving this strip by strip, let's look at the lift of the generator and markings to prove this case. There is one component of $\tilde{\x}$ with $x$-coordinate 0; suppose it has $y$-coordinate $s$. Similarly, there is one component of $\tilde{\OO}$ with $x$-coordinate 1/2; suppose it has $y$ coordinate $1/2+t$. We will assume that $s\ge t$. As in the first case, let $\x'$ be the same generator, but with coordinates given with respect to the fundamental domain shifted to the right by $1/pn$. Then $\I(\tilde{\x'},\tilde{\x'})=\I(\tilde{\x},\tilde{x})+s-(n-s-1)=\I(\tilde{\x},\tilde{x})-n+2s+1$. Similarly, for the other terms, \begin{align*}
    \I(\tilde{\OO'},\tilde{\OO'})&=\I(\tilde{\OO},\tilde{\OO})+t-(n-t-1)=\I(\tilde{\OO},\tilde{\OO})-n+2t+1\\
    \I(\tilde{\x'},\tilde{\OO'})&=\I(\tilde{\x},\tilde{\OO})-(n-s-1)+t\\
    \I(\tilde{\OO'},\tilde{\x'})&=\I(\tilde{\OO},\tilde{\x})-(n-t-1)+s.
\end{align*}
In the end, we see that as above, we have $\tilde{\maslov}(\x')=\tilde{\maslov}(\x)$, as desired. The case where $s\le t$ is analogous. 
\end{proof}

%% file: Sections/Gradings/alexander.tex
The third and final grading comes, in its most general form, as a filtration on the chain complex. Throughout this paper, however, we have restricted ourselves to a differential which counts rectangles which do not contain any markings in $\XX$ in their interior, which is the differential on the associated graded complex coming from the filtration. As such, on the complex $(C^-(\gridg),\dxminus)$, we get a grading, called the Alexander grading. 

In an analogous way to what we will here call $\maslov_{\OO}$ defined above, we can define $\maslov_{\XX}$. For $\x\in\genG$, we then define $\tilde{\alex}(\x)=\maslov_{\OO}(\x)-\maslov_{\XX}(\x)$. Recall that $L$ is a link with $\ell$ components. The \emph{Alexander grading} of a generator $\x\in\genG$ is defined by \begin{equation}\label{eqn:alexdefn}\alex(\x)=\frac{\tilde{\alex}(\x)-(n-\ell)}{2}.\end{equation} We extend this to all of $C^-(\gridg)$ by declaring $\alex(V_i\x)=\alex(\x)-1$. It is clear that $\alex$ is well-defined, as $\maslov_{\XX}$ and $\maslov_{\OO}$ are well-defined by Lemma~\ref{lem:maslov}. 

Let us compute the Alexander grading of the generator $\z$ of Figure~\ref{fig:fundomain}. We compute \begin{align*}
    \I(\tz,\tz)&=52\\
    \I(\tz,\tXX)&=67\\
    \I(\tXX,\tz)&=52\\
    \I(\tXX,\tXX)&=62.
\end{align*}

We recall that $d(5,2,1)=-\frac{2}{5}$, so $\maslov_{\XX}(\z)=\frac{1}{5}(52-67-52+62)-\frac{2}{5}+1=-\frac{2}{5}$. Thus we have $\tilde{\alex}(\z)=\frac{2}{5}-(-\frac{2}{5})=\frac{4}{5}$, and $\alex(\x)=\frac{\frac{4}{5}-(3-1)}{2}=-\frac{3}{5}$.

%% file: Sections/invariance.tex
By Theorem~\ref{thm:cromwell}, two grid diagrams represent the same isotopy class of link in $\lpq$ if and only if they are related by a finite sequence of moves called commutations, stabilizations, and destabilizations. After defining these moves, we will provide explicit homotopy equivalences between the chain complexes associated to two grids connected by such a move. Throughout, let $\gridg$ be a grid diagram with grid number $n$ for a $\ell$-component link $L$ in $\lpq$.

\subsection{Commutations}\label{sec:commutation}
\input{Sections/Invariance/commutation}

\subsection{Switches}\label{sec:switches}
\input{Sections/Invariance/switches}

\subsection{Stabilizations and destabilizations}\label{sec:stabilization}

\input{Sections/Invariance/stabilization}

\subsection{Proof of invariance}
We now have the ingredients we need to prove Theorem~\ref{thm:maintheoremz2}. 

\begin{proof}[Proof of Theorem~\ref{thm:maintheoremz2}]
By Theorem~\ref{thm:cromwell}, if $\gridg$ and $\gkprime$ represent the same link, they are connected by a sequence of commutations and (de)stabilizations. Therefore the result follows directly from Propositions~\ref{prop:commutation} and ~\ref{prop:stabilization}. 
\end{proof}

%% file: Sections/Invariance/commutation.tex
We will discuss in detail the case of a commutation of columns, and at the end explain how this can be modified to work for a commutation of rows. Consider a pair of consecutive columns in $\gridg$ with the first column bounded by $\beta_{i-1}$ and $\beta_i$, and the second column by $\beta_i$ and $\beta_{i+1}$. We say the markings from these two columns \emph{do not interleave} if there is some point on $\beta_i$ such that when traveling along $\beta_i$ starting from that point we pass the regions of $T^2-\alphas-\betas$ containing both of the markings from one column, then the regions containing both of the markings from the other column. 

If the markings in these two columns do not interleave, we can perform a \emph{commutation move}, which consists of exchanging the markings in the first column with those in the second. Let $\gkprime$ be the result of applying a commutation move to $\gridg$, and let $L'$ be the link specified by $\gkprime$. By Theorem~\ref{thm:cromwell}, $L$ and $L'$ are smoothly isotopic. The fact that knot Floer homology is a link invariant then tells us that $(CF^-(\lpq,L),\dxminus)\simeq (CF^-(\lpq,L'),\dxminusprime)$. Furthermore, by Proposition~\ref{prop:cminusiscfk} we have that $(C^-(\gridg),\dxminus)\simeq(CF^-(\lpq,L),\dxminus)$ and $(C^-(\gkprime),\dxminus)\simeq(CF^-(\lpq,L'),\dxminusprime)$. We get, then, that $(C^-(\gridg),\dxminus)$ is quasi-isomorphic to $(C^-(\gkprime),\dxminusprime)$. We will prove this directly, without appealing to the holomorphic theory:

\begin{prop}\label{prop:commutation}
Suppose $\gridg$ and $\gkprime$ are grid diagrams for a $\ell$-component link $L$ which are connected by a commutation move. Then $(C^-(\gridg),\dxminus)$ and $(C^-(\gkprime),\dxminusprime)$ are homotopy equivalent as chain complexes over $\Ztwo[U_1,\ldots,U_\ell]$. 
\end{prop} 

Let us set some notation, following the language of \cite{Manolescu_2007}. Let $\betas'$ be the set $\betas$ after applying a commutation of columns to $\gridg$. Let $\beta:=\beta_i$ and $\gamma:=\beta_i'$. We wish to draw $\gridg$ and $\gkprime$ on the same diagram. To do so, perturb $\beta$ and $\gamma$ so they meet transversely at two points which are not on a horizontal circle. Figure~\ref{fig:commutation} depicts both grid diagrams drawn on the same twisted torus. Denote by $\genG$ the set of generators of $C^-(\gridg)$ and by $\genG'$ the set of generators of $C^-(\gkprime)$. 

\begin{figure}[h]
    \centering
    \begin{tikzpicture}
    \node at (0,0) {\includegraphics[height=6cm]{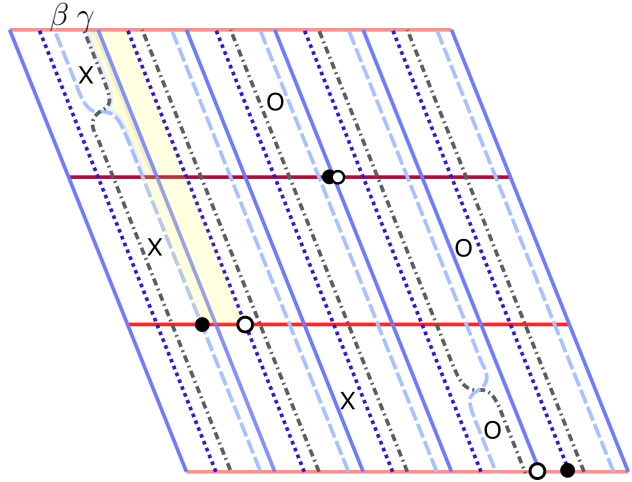}};
    \node at (-3.45,3.15) {$\beta$};
    \node at (-3.1,3.1) {$\gamma$};
    \end{tikzpicture}
    \caption{A commutation of the second and third columns on a grid diagram for a knot in $L(5,2)$ of grid index 3. The darkened circles represent a generator $\x\in\genG$ and the empty circles represent a generator $\y'\in\genG'$. The shaded pentagon is a pentagon connecting $\x$ to $\y'$. }
    \label{fig:commutation}
\end{figure}

We will define a homotopy equivalence $\Phi_{\beta\gamma}:C^-(\gridg)\to C^-(\gkprime)$ counting pentagons in the combined diagram, with homotopy inverse $\Phi_{\gamma\beta}:C^-(\gkprime)\to C^-(\gridg)$. We will define a homotopy $\Hbgb:C^-(\gridg)\to C^-(\gridg)$ which counts hexagons in the combined diagram. We define these polygons here. 

\begin{defn}[cf.~{\cite[Definition 5.1.1]{gridhomology}}]
Fix $\x\in\genG$ and $\y'\in\genG'$. A \emph{pentagon $s$ from $\x$ to $\y'$} is a disk embedded in $T^2$ along with the data of an \emph{initial generator} $\x$ and \emph{terminal generator} $\y'$ such that \begin{itemize}
    \item the boundary of $s$ consists of five arcs lying on $\alphas\cup\betas\cup\betas'$,
    \item four of the corners of $s$ are in $\x\cup\y'$; the last is one of the intersection points of $\beta$ and $\gamma$,
    \item $\partial(\partial_{\alpha}(s))=\y'-\x$, and
    \item around each corner point of $s$, the disk has multiplicity one in one region, and zero in the other three.
\end{itemize} 
\end{defn}

Let $\pbg(\x,\y')$ denote the space of pentagons from $\x$ to $\y'$, noting that $\pbg(\x,\y')$ is empty unless $\x$ and $\y'$ agree at $n-2$ points. Call a pentagon $s$ \emph{empty} if $\Int(s)\cap\x=\emptyset=\Int(s)\cap\y'$, and denote by $\epbg(\x,\y')$ the space of empty pentagons from $\x$ to $\y'$. 

\begin{remark}\label{rem:pentagontransposition}
Just as in the case for parallelograms, if a pentagon $s$ connects $\x$ to $\y'$, then $\sigma_{\y'}=\tau_{i,j}\sigma_{\x}$, for some transposition $\tau_{i,j}\in S_n$, which we denote $\tau_s$. Furthermore, we have that $a^{\x}_r=a^{\y'}_r$ for all $r\in\{0,\ldots,n-1\}\setminus\{i,j\}$, and $a^{\x}_i-a^{\y'}_i=a^{\y'}_j-a^{\x}_j=k$ for some $|k|<n$. 
\end{remark}

As with parallelograms, we can generalize the notion of a pentagon into that of a \emph{domain}. We overload the term \emph{domain} with the below notation, but the meaning will always be clear from context. 

\begin{defn}[cf.~{\cite[Lemma 5.1.4]{gridhomology}}]\label{def:domain2}
Fix $\x\in\genG$, $\y'\in\genG'$. A \emph{domain $\phi$ from $\x$ to $\y'$} is a formal linear combination of the closure of regions of $T^2\setminus(\alphas\cup\betas\cup\betas')$ such that $\partial(\partial_{\alpha}(\phi))=\y'-\x$, along with the data of the generators $\x$ and $\y'$. As before, we will call $\x$ the \emph{initial generator} for the domain, and $\y'$ the \emph{terminal generator}. The \emph{support} of a domain $\phi$ from $\x$ to $\y$ is the underlying linear combination. Let $\dbeta(\phi)=\partial(\phi)\cap(\betas\cup\betas')$. 
\end{defn} 

Remark~\ref{rem:omultiplicity} applies to domains of Definition~\ref{def:domain2} just as to domains of Definition~\ref{def:domain1}: if a domain $\phi$ from $\x$ to $\z$ has a decomposition as the juxtaposition of two domains $\phi=\psi_1*\psi_2$, then $n_{O_i}(\phi)=n_{O_i}(\psi_1)+n_{O_2}(\psi)$.

Pentagons are one specialization of domains; another is hexagons. 

\begin{defn}[cf.~{\cite[Definition 5.1.5]{gridhomology}}]
Let $\x,\y\in\genG$. A \emph{hexagon $h$ from $\x$ to $\y$} is a disk embedded in $T^2$ along with the data of generators $\x$ and $\y$ such that \begin{itemize}
    \item the boundary of $h$ consists of six arcs lying on $\alphas\cup\betas\cup\betas'$, 
    \item four of the corners of $h$ are in $\x\cup\y$; the other two corners are the two intersection points of $\beta$ and $\gamma$,
    \item $\partial(\partial_{\alpha}(h))=\y-\x$, and 
    \item around each corner point of $h$, the disk mas multiplicity one in one region, and zero in the other three. 
\end{itemize}
\end{defn}

Let $\hexbgb(\x,\y)$ denote the space of hexagons from $\x$ to $\y$, noting that $\hexbgb(\x,\y)$ is empty unless $\x$ and $\y$ agree at $n-2$ points. Call a hexagon $h$ empty if $\Int(h)\cap\x=\emptyset=\Int(h)\cap\y$, and denote by $\ehexbgb(\x,\y)$ the space of empty pentagons from $\x$ to $\y$. 

We can now define our chain map from $C^-(\gridg)$ to $C^-(\gkprime)$, its homotopy inverse, and the homotopy.

Define $\Phi_{\beta\gamma}:C^-(\gridg)\to C^-(\gkprime)$ on a generator $\x\in\genG$ by \[\Phi_{\beta\gamma}(\x)=\sum_{\y\in \genG'}\sum_{\substack{s\in \epbg(\x,\y)\\s\cap\XX=\emptyset}}V_0^{n_{O_0}(s)}\cdots V_{n-1}^{n_{O_{n-1}}(s)}\y. \]

Extend this linearly to all of $C^-(\gridg)$.

\begin{lem}\label{lem:phibetagamma}
The map $\Phi_{\beta\gamma}$ is a chain map. 
\end{lem}
\begin{proof}
The analysis of domains is similar to the corresponding proof for knots in $S^3$ (cf.~{\cite[Lemma 5.1.4]{gridhomology}}). 

We wish to prove \begin{equation}\label{eqn:phibetagammachainmap}
    (\Phi_{\beta\gamma}\circ\dxminus)(\x)=(\dxminus\circ\Phi_{\beta\gamma})(\x).
\end{equation} 

Let $\phi$ be a domain connecting $\x\in\genG$ to $\z\in\genG'$. Suppose $\phi$ has a decomposition of one of two types: either as the juxtaposition of a parallelogram and a pentagon, or as the juxtaposition of a pentagon and a parallelogram. For simplicity, we will only consider domains which have a decomposition with a pentagon to the righ of the intersection point; the cases where the pentagon lie to the left of the intersection point follow analogously. 

As in the proof of Proposition~\ref{prop:d2=0}, let $m$ be the number of components where $\x$ and $\z$ differ, and observe that $m\le 4$. 

We will show that for $2\le m\le 4$, $\phi$ has exactly one alternate decomposition, either of the same type, in which case it drops out of (\ref{eqn:phibetagammachainmap}) in characteristic two, or of the opposite type, proving the equality of the two terms of (\ref{eqn:phibetagammachainmap}). 

For $m=1$, there are two cases. In the first, $\phi$ has exactly two decompositions, of opposite types, proving the equality of (\ref{eqn:phibetagammachainmap}). In the other, we will show that $\phi$ has a unique decomposition, but pairs with another domain connecting $\x$ to $\z$ also with $m=1$ either of the same type, in which case it drops out, or of the opposite type, proving the equality of (\ref{eqn:phibetagammachainmap}). 

Note that there are no domains with $m= 0$, as $\x$ and $\z$ lie on different grid diagrams, so we never have $\x=\z$. 

Suppose $m=4$. Then in analogy to Proposition~\ref{prop:d2=0}, we will show that $\phi$ has exactly two decompositions. By assumption, we know that $\phi$ has a decomposition as one of the two types. In the same method as in the proof of Proposition~\ref{prop:d2=0}, we can construct the unique alternate decomposition. 

In this case, the supports of the parallelogram and pentagon are either disjoint or have overlapping interiors but share no corners or edges (except for transverse edge intersections). Suppose $\phi=r*s$ for $r\in\ePG(\x,\tgen)$ and $s\in\ePent(\tgen,\z)$ for some $\tgen\in\genG$. We can construct $\w\in\genG$ by letting it agree with $\tgen$ away from $r$ and $s$, by agreeing with $\x$ on $r$, and agreeing with $\z$ on $s$. Then $s'\in \ePent(\x,\w)$, $r'\in\ePG(\w,\z)$, and we get the unique alternate decomposition is $\phi=s'*r'$. 

If instead $\phi$ was the juxtaposition of a pentagon and a parallelogram, we would construct the unique alternate decomposition in the analogous way. These cases can be seen in Figure~\ref{fig:pentagon4}. 

\begin{figure}[h]
    \centering
    \includegraphics[height=\figheight]{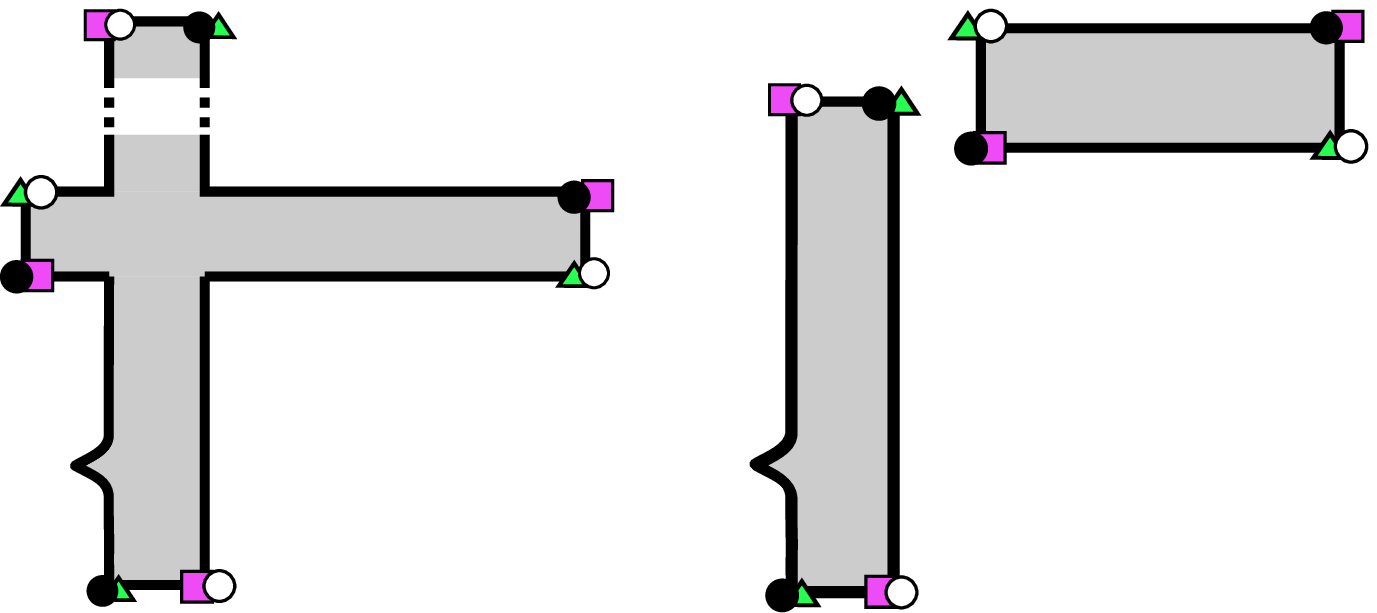}
    \caption{Domains counted in the proof of Lemma~\ref{lem:phibetagamma} when $m=4$. }
    \label{fig:pentagon4}
\end{figure}

Suppose $m=3$. Again, these cases are similar to those analyzed in the proof of Proposition~\ref{prop:d2=0}, and are also similar to domains that were analyzed in \cite{petkova2019skein}. Each such domain has a unique corner at which the domain has multiplicity one in three of the four regions around the corner, and multiplicity zero in the fourth. For the domains in Figure~\ref{fig:pentagon3differentcomponent} and Figure~\ref{fig:pentagon3samecomponent}, this corner is a component of $\x$ or $\z$. 

\begin{figure}[h]
    \centering
    \includegraphics[height=\figheight]{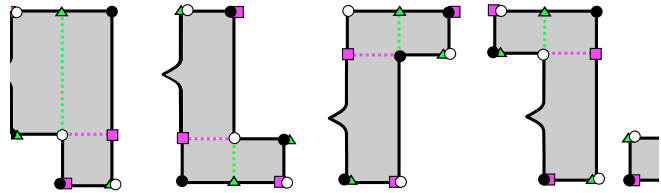}
    \caption{Certain domains counted in the proof of Lemma~\ref{lem:phibetagamma} when $m=3$. One decomposition contributes to each term:  $\Phi_{\beta\gamma}\circ\dxminus$ and $\dxminus\circ\Phi_{\beta\gamma}$.}
    \label{fig:pentagon3differentcomponent}
\end{figure}

\begin{figure}[h]
    \centering
    \includegraphics[height=\figheight]{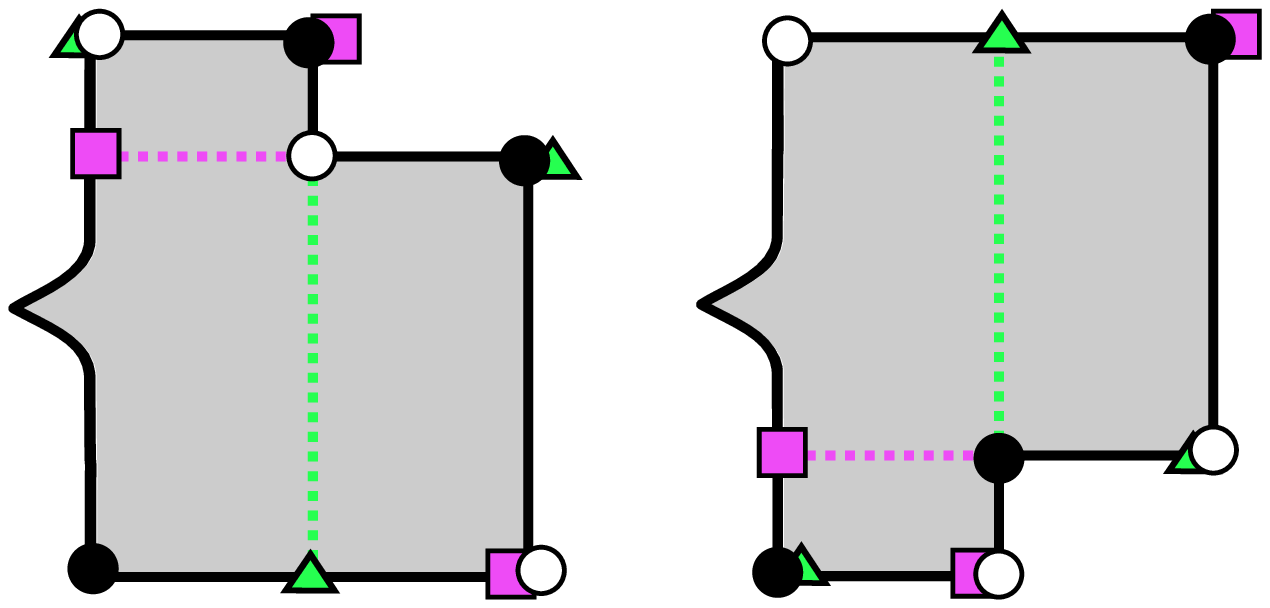}
    \caption{Certain domains counted in the proof of Lemma~\ref{lem:phibetagamma} when $m=3$. For the first domain, both decompositions contribute to $\dxminus\circ\Phi_{\beta\gamma}$. For the second domain, both decompositions contribute to $\Phi_{\beta\gamma}\circ\dxminus$.}
    \label{fig:pentagon3samecomponent}
\end{figure}

Cutting along the curves in $\alphas$ and $\betas$ into the domain yields components of the two intermediate generators, the other components of which are gotten similarly to the $m=3$ case of the proof of Proposition~\ref{prop:d2=0}. To define these other components, in the proof of of Proposition~\ref{prop:d2=0}, we at some point traveled along $\dbeta(\phi)$ via a curve in $\betas$. For these domains, if that curve is $\beta$ or $\gamma$, we note that if we reach an intersection point of $\beta$ and $\gamma$ we should keep following $\dbeta(\phi)$, which might require switching our course from $\beta$ to $\gamma$ or vice-versa.

For the domains in Figure~\ref{fig:pentagon3differentbetagamma}, this unique corner is at the intersection of $\beta$ and $\gamma$. Cutting from this point into $\phi$ along each of $\beta$ and $\gamma$ yields a component of each intermediate generator; these intermediate generators agree away from these components. 

\begin{figure}[h]
    \centering
    \includegraphics[height=\figheight]{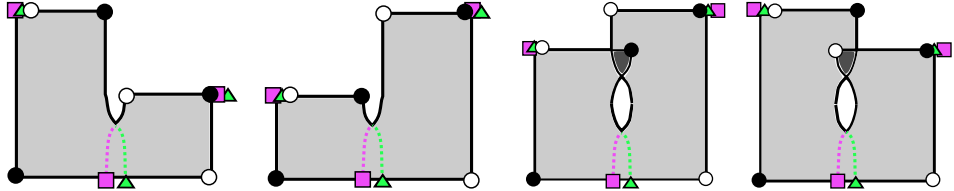}
    \caption{Certain domains counted in the proof of Lemma~\ref{lem:phibetagamma} when $m=3$. One decomposition contributes to each term:  $\Phi_{\beta\gamma}\circ\dxminus$ and $\dxminus\circ\Phi_{\beta\gamma}$.}
    \label{fig:pentagon3differentbetagamma}
\end{figure}

Suppose $m=2$. Each such domain has exactly two alternate decompositions, gotten in the same method as in the proof of Proposition~\ref{prop:d2=0}. As in that proof, $\phi$ has exactly two corners around which the domain $\phi$ has multiplicity one in three of the four regions and multiplicity zero in the fourth. Traveling along the curves in $\alphas$ and the curves in $\betas\cup\betas'$ occupied by these two corners yields the components of the intermediate generators. For the domains of Figure~\ref{fig:pentagon2alternate} these decompositions are of opposite type; for the domains of Figure~\ref{fig:pentagon2same}, these decompositions are of the same type.

\begin{figure}[h]
    \centering
    \includegraphics[height=\figheight]{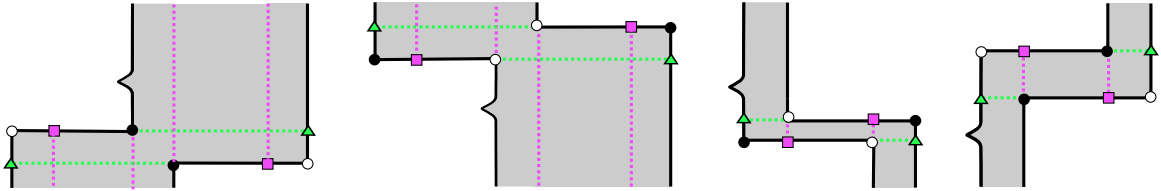}
    \caption{Certain domains counted in the proof of Lemma~\ref{lem:phibetagamma} when $m=2$. One decomposition contributes to each term:  $\Phi_{\beta\gamma}\circ\dxminus$ and $\dxminus\circ\Phi_{\beta\gamma}$.}
    \label{fig:pentagon2alternate}
\end{figure}

\begin{figure}[h]
    \centering
    \includegraphics[height=\figheight]{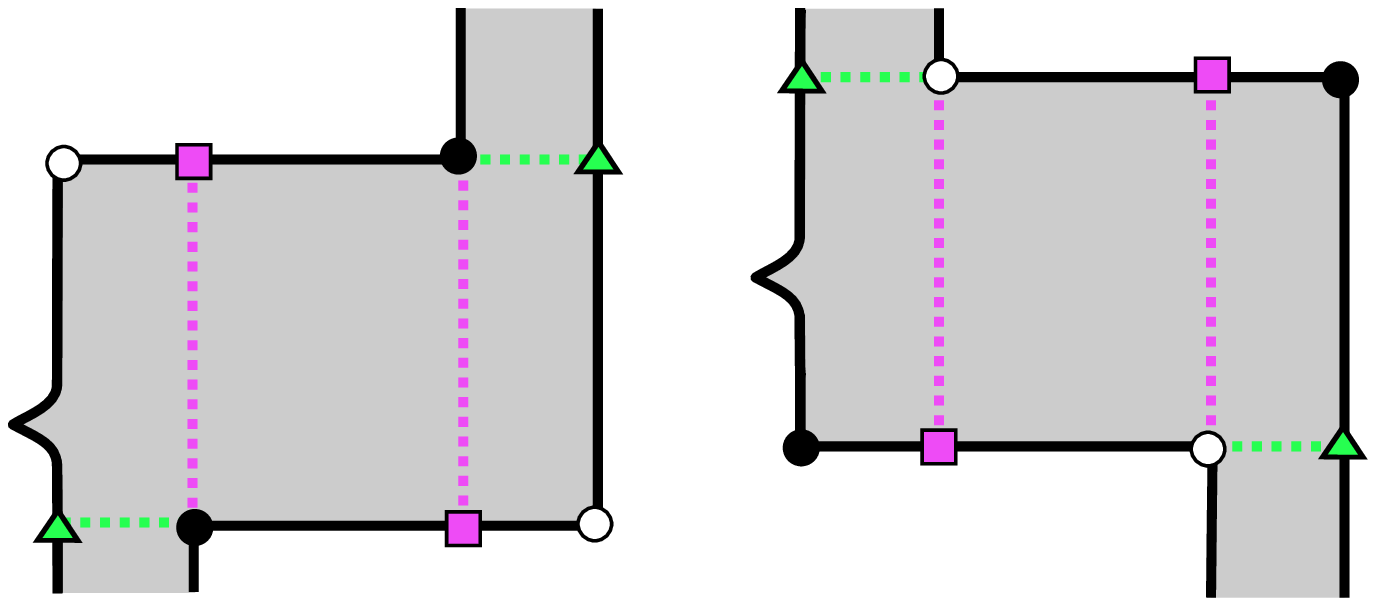}
    \caption{Certain domains counted in the proof of Lemma~\ref{lem:phibetagamma} when $m=2$. For the first domain, both decompositions contribute to $\Phi_{\beta\gamma}\circ\dxminus$. For the second domain, both decompositions contribute to $\dxminus\circ\Phi_{\beta\gamma}$.}
    \label{fig:pentagon2same}
\end{figure}

Finally, we must analyze the case where $m=1$. This is the analogue of the case where $\x=\z$ in Proposition \ref{prop:d2=0}. Remarks~\ref{rem:pentagontransposition} and \ref{rem:parallelogramtransposition} make clear this is only possible if the $p$-coordinates of $\x$ and $\z$ are the same and $\sigma_{\x}=\sigma_{\z}$, with one component of $\x$ moving from $\beta$ to $\gamma$. There are two such possibilities.

In either case, we will need to define on the grid a \emph{triangle}, an embedded disk in $T^2$ with boundary consisting of an arc of $\beta$, an arc of $\gamma$, and an arc of some curve in $\alphas$. 

In the first case, the support of $\phi$ consists of a row of the grid diagram, less a triangle. In this case, cutting into $\phi$ from the intersection of $\beta$ and $\gamma$ on the boundary of $\phi$ along each of $\beta$ and $\gamma$ yields components of each of the intermediate generators. These domains of this type can be seen in Figure~\ref{fig:pentagon1horizontal}. 

\begin{figure}[h]
    \centering
    \includegraphics[height=1.5cm]{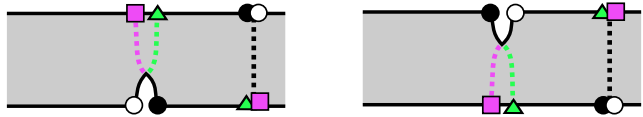}
    \caption{The domains counted in the proof of Lemma~\ref{lem:phibetagamma} when $m=1$ which are a row less a triangle. One decomposition contributes to each term:  $\Phi_{\beta\gamma}\circ\dxminus$ and $\dxminus\circ\Phi_{\beta\gamma}$.}
    \label{fig:pentagon1horizontal}
\end{figure}

Note that $T^2-\alphas-\betas-\betas'$ has two regions which are annuli and border both $\beta$ and $\gamma$. In the second case, the support of $\phi$ is one of these regions which is an annulus, as well as an abutting triangle. Each of these has a unique decomposition, but the domain whose support consists of the same triangle and the opposite annulus, denoted $\phi'$ is another domain with the same initial and terminal generators, and a unique decomposition. Thus, these domains pair together, as desired. There are six combinatorial possibilities when considering the relative positions of the components of a generator on $\beta_{i-1}$ and $\beta_i$. The decomposition of $\phi'$ depends on the relative position of the component of the generator on $\beta_{i+1}$. In Figure~\ref{fig:pentagon1vertical} we display these six combinatorial possibilities, as well as the decompositions of $\phi'$ gotten by an arbitrarily chosen location for the component on $\beta_{i+1}$. 
\end{proof}

\begin{figure}[h]
    \centering
    \includegraphics[height=5.5cm]{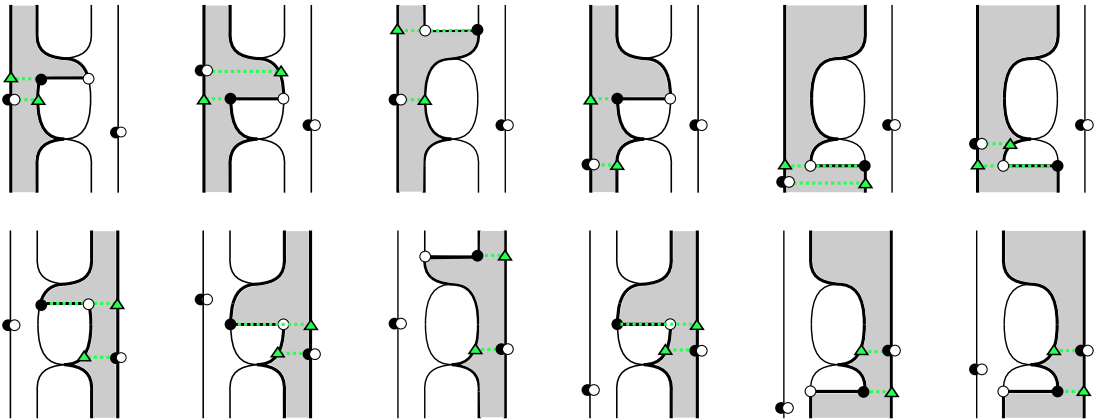}
    \caption{Certain domains counted in the proof of Lemma~\ref{lem:phibetagamma} when $m=1$. Each has a unique decomposition, but pairs with the domain below it to either cancel out (if of the same type) or prove our desired equality (if of the opposite type).}
    \label{fig:pentagon1vertical}
\end{figure}

\begin{lem}\label{lem:pentagongrading}
The map $\Phi_{\beta\gamma}$ preserves all three gradings.  
\end{lem}

\begin{proof}
This is directly analogous to Proposition \ref{prop:parallelogramgradings}, with the preservation of $\spincS$ proved by Remark \ref{rem:pentagontransposition}, and the preservation of $\maslov$ and $\alex$ proved by considering sequences of pentagons in the universal cover. 
\end{proof}

\subsection{Homotopy}
The homotopy between $\Phi_{\gamma\beta}\circ\Phi_{\beta\gamma}$ and the identity will be provided by a map counting hexagons. 

Define $\Hbgb:C^-(\gridg)\to C^-(\gridg)$ by \[\Hbgb(\x)=\sum_{\y\in\genG}\sum_{\substack{h\in\ehexbgb(\x,\y)\\\Int(h)\cap\XX=\emptyset}}V_0^{n_{O_0}(h)}\cdots V_{n-1}^{n_{O_{n-1}}(h)}\y.\]

\begin{lem}\label{lem:hbgb}
The map $\Hbgb$ is a homotopy between $(\Phi_{\gamma\beta}\circ\Phi_{\beta\gamma})$ and the identity $\mathbb{I}$. 
\end{lem}
\begin{proof}
We again must look at decompositions of domains. Repurposing the notation of Proposition \ref{prop:d2=0}, let $\ePoly(\x,\z)$ denote the set of empty domains from $\x\in\genG$ to $\z\in\genG$ which can be written as a juxtaposition of a parallelogram and a hexagon, a hexagon and a parallelogram, or two pentagons. For $\phi\in\ePoly(\x,\z)$, let $N(\phi)$ denote the number of decompositions of $\phi$ in one of the above three ways. In this notation, we have 
\begin{equation}\label{eq:hexhomotopy}
(\Phi_{\gamma\beta}\circ\Phi_{\beta\gamma}+\dxminus\circ \Hbgb+\Hbgb\circ\dxminus)(\x)=\sum_{\z\in\genG}\sum_{\substack{\phi\in\ePoly(\x,\z)\\\Int(\phi)\cap\XX=\emptyset}}N(\phi)V_0^{n_{O_0}(\phi)}\cdots V_{n-1}^{n_{O_{n-1}}(\phi)}\z.
\end{equation}

Our goal is to show that the left hand side of (\ref{eq:hexhomotopy}) equals the identity on $C^-(\gridg)$. We will do this by proving for most domains $\phi\in\ePoly(\x,\z)$, we have $N(\phi)=2$, so they are not counted in characteristic two. The only domains for which this is not true are those in $\ePoly(\x,\x)$, which have $N(\phi)=1$. Again, let $m$ be the number of components where $\x$ and $\z$ disagree, and observe $m\le 4$. 

Suppose $m=4$. Note that $\phi$ does not have a decomposition as the juxtaposition of two pentagons. Supposing it did, the intermediate generator would occupy $\gamma$. For the terminal generator to be in $\genG$, the second pentagon would need to share a corner with a first at this component on $\gamma$, so the pentagons would share at least one corner and some portion of an edge, precluding the case $m=4$. 

This means that, in analogy to the proofs of Lemma~\ref{lem:phibetagamma} and of Proposition~\ref{prop:d2=0}, $\phi$ has exactly two decompositions: one as the juxtaposition of a parallelogram and a hexagon, and one as the juxtaposition of a hexagon and a parallelogram. The parallelograms have the same support as each other, as do the hexagons. The common supports of the parallelograms and the hexagons are either disjoint, or have overlapping interiors but share no corners or edges (except for transverse edge intersections). See Figure~\ref{fig:dsquared4} or Figure~\ref{fig:pentagon4} for the analogous cases. 

Suppose $m=3$. These cases are very similar to those analyzed in Lemma~\ref{lem:phibetagamma}, and also to those in the classical case analyzed in \cite{gridhomology} Lemma 5.1.6. 

For most of these domains, there is a unique corner around which the domain $\phi$ has multiplicity one in three regions and zero in the fourth. Cutting into $\phi$ along each of the two curves this corner occupies yields components of each of the intermediate generators on $\partial(\phi)$. These domains can be seen in Figures~\ref{fig:hex3pentagonscomponent},  \ref{fig:hex3alternate}, and \ref{fig:hex3same}, organized by which terms in (\ref{eq:hexhomotopy}) they contribute to.

\begin{figure}[h]
    \centering
    \includegraphics[height=\figheight]{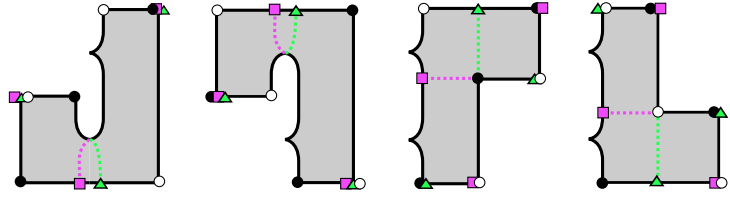}
    \caption{Certain domains counted in the proof of Lemma~\ref{lem:hbgb} when $m=3$. Each domain has one decomposition which contributes to $\Phi_{\gamma\beta}\circ\Phi_{\beta\gamma}$. The alternate decomposition for the first and third domain contributes to $\Hbgb\circ\dxminus$, and the alternate decomposition for the second and fourth domain contributes to $\dxminus\circ\Hbgb$.}
    \label{fig:hex3pentagonscomponent}
\end{figure}

\begin{figure}[h]
    \centering
    \includegraphics[height=\figheight]{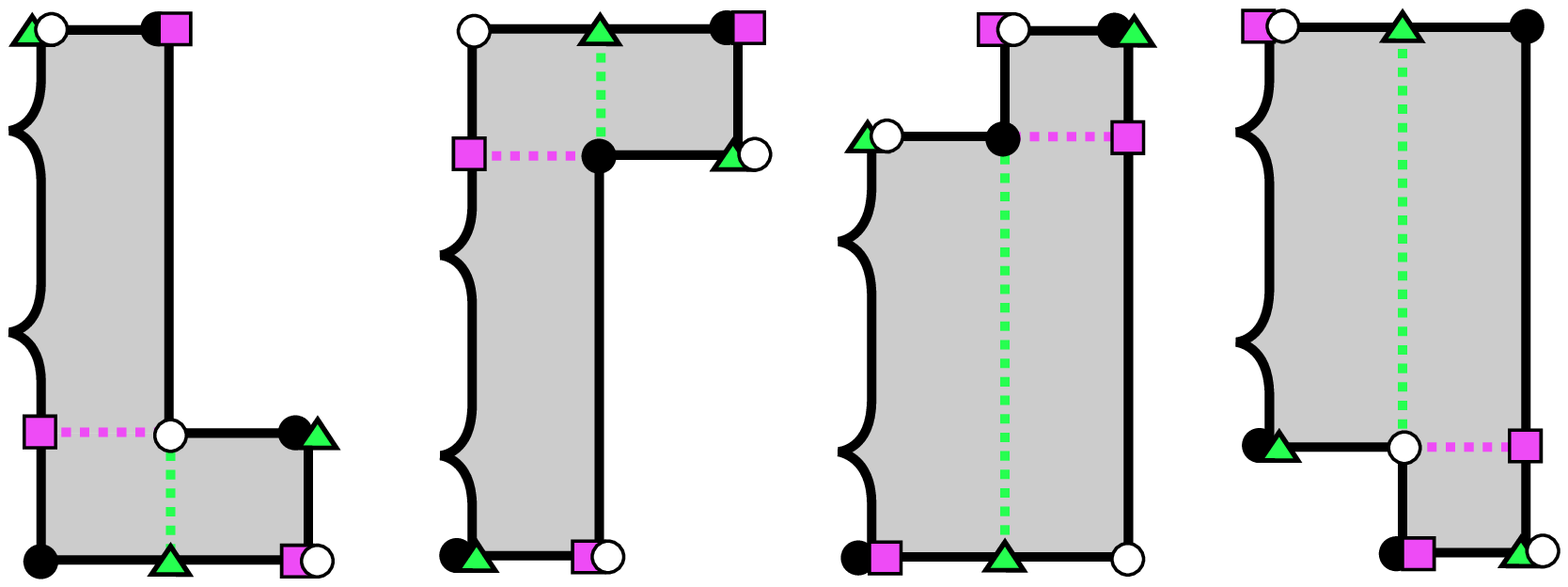}
    \caption{Certain domains counted in the proof of Lemma~\ref{lem:hbgb} when $m=3$. Each domain has one decomposition which contributes to $\Hbgb\circ\dxminus$ and one decomposition which contributes to $\dxminus\circ\Hbgb$.}
    \label{fig:hex3alternate}
\end{figure}

\begin{figure}[h]
    \centering
    \includegraphics[height=\figheight]{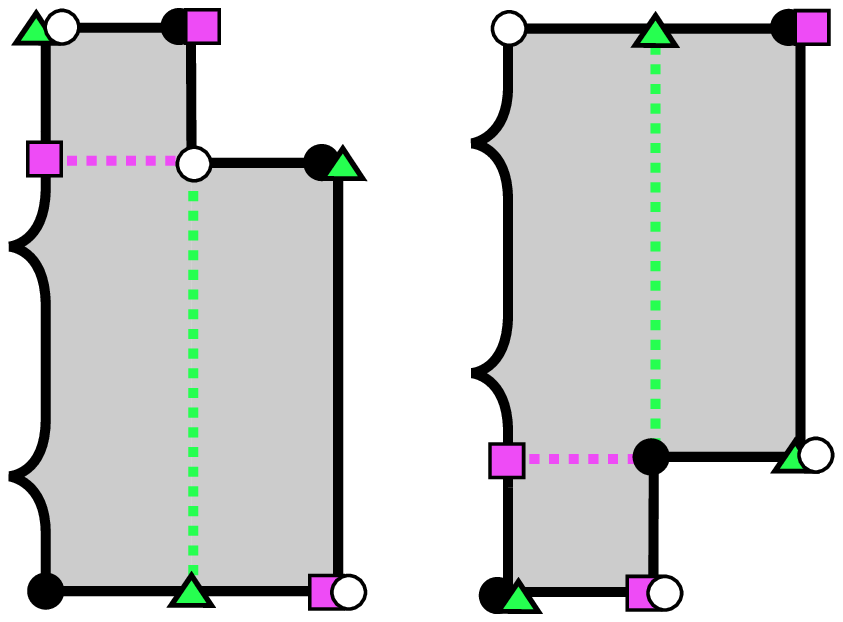}
    \caption{Certain domains counted in the proof of Lemma~\ref{lem:hbgb} when $m=3$. Both decompositions of the first domain contribute to $\dxminus\circ\Hbgb$; both decompositions of the second domain contribute to $\Hbgb\circ\dxminus$.}
    \label{fig:hex3same}
\end{figure}

For two domains, there are three such corners: the two intersection points of $\beta$ and $\gamma$, as well as a component of $\x$ or $\z$. From the intersection point of $\beta$ and $\gamma$ which is in every pentagon counted by $\Phi_{\beta\gamma}$, cutting along each of $\beta$ and $\gamma$ yields a component of each of the intermediate generators. These domains can be seen in Figure~\ref{fig:hex3pentagonsbetagamma}. 

\begin{figure}[h]
    \centering
    \includegraphics[height=\figheight]{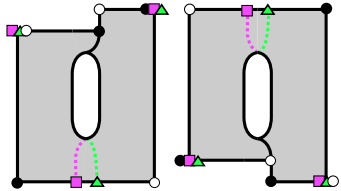}
    \caption{Certain domains counted in the proof of Lemma~\ref{lem:hbgb} when $m=3$. Each domain has one decomposition which contributes to $\Phi_{\gamma\beta}\circ\Phi_{\beta\gamma}$.  The alternate decomposition for the first domain contributes to $\Hbgb\circ\dxminus$, and the alternate decomposition for the second domain contributes to $\dxminus\circ\Hbgb$.}
    \label{fig:hex3pentagonsbetagamma}
\end{figure}

Suppose $m=2$. These cases are very similar to those analyzed in Lemma~\ref{lem:phibetagamma}. Each domain has two points around which the domain has multiplicity one in three of the four regions, and multiplicity zero in the fourth. From these two points cutting into $\phi$ along each of the two curves each of these two corners occupy yield components of each of the intermediate generators on $\partial(\phi)$. These domains can be seen with their alternate decompositions in Figures~\ref{fig:hex2pp},~\ref{fig:hex2alternating}, and ~\ref{fig:hex2same}, organized by which terms in (\ref{eqn:phibetagammachainmap}) they contribute to. 

\begin{figure}[h]
    \centering
    \includegraphics[height=\figheight]{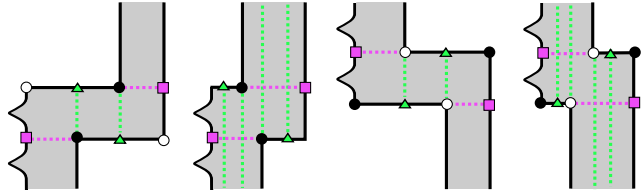}
    \caption{Certain domains counted in the proof of Lemma~\ref{lem:hbgb} when $m=2$. Each domain has one decomposition which contributes to $\Phi_{\gamma\beta}\circ\Phi_{\beta\gamma}$.  The alternate decomposition for the first two domains contributes to $\Hbgb\circ\dxminus$, and the alternate decomposition for the second domain contributes to $\dxminus\circ\Hbgb$.}
    \label{fig:hex2pp}
\end{figure}

\begin{figure}[h]
    \centering
    \includegraphics[height=\figheight]{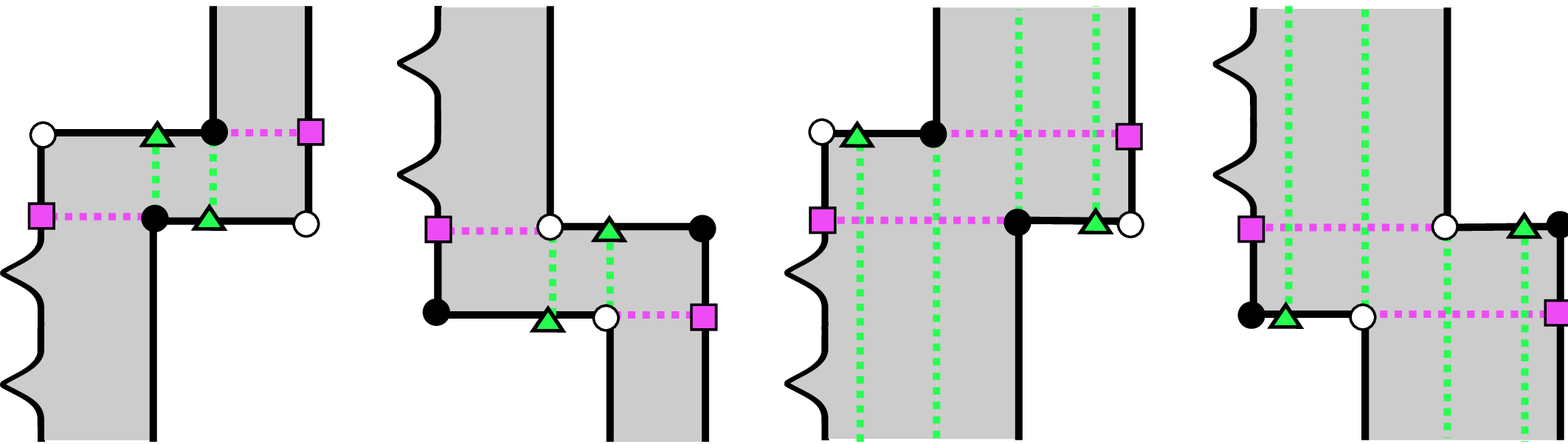}
    \caption{Certain domains counted in the proof of Lemma~\ref{lem:hbgb} when $m=2$. Each domain has one decomposition which contributes $\Hbgb\circ\dxminus$ and one decomposition which contributes to $\dxminus\circ\Hbgb$.}
    \label{fig:hex2alternating}
\end{figure}

\begin{figure}[h]
    \centering
    \includegraphics[height=\figheight]{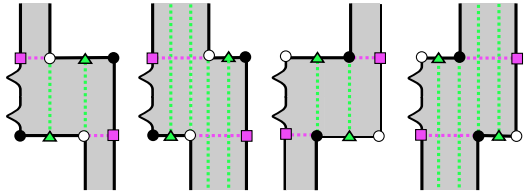}
    \caption{Certain domains counted in the proof of Lemma~\ref{lem:hbgb} when $m=2$. The two decompositions of the first two domains each contribute to $\dxminus\circ\Hbgb$, and the two decompositions of the third and fourth domains each contribute to $\Hbgb\circ\dxminus$.}
    \label{fig:hex2same}
\end{figure}

As is the case for Proposition~\ref{prop:d2=0}, there are no domains with $m=1$. This is true by Remarks \ref{rem:parallelogramtransposition} and \ref{rem:pentagontransposition}, and the analogue for hexagons. 

Suppose $\x=\z$. There is exactly one domain in $\ePoly(\x,\x)$, and the type of its decomposition is given by the relative location of $\x$, $\beta$, and $\gamma$. The support of each such domain is one of the annuli of $T^2-\alphas-\betas-\betas'$ which borders both $\beta$ and $\gamma$. It cannot also contain any other annulus, as then it would contain a marking $X\in\XX$ and a component of $\x$. Therefore this domain is counted with coefficient one in (\ref{eq:hexhomotopy}). 

\begin{figure}[h]
    \centering
    \includegraphics[height=3cm]{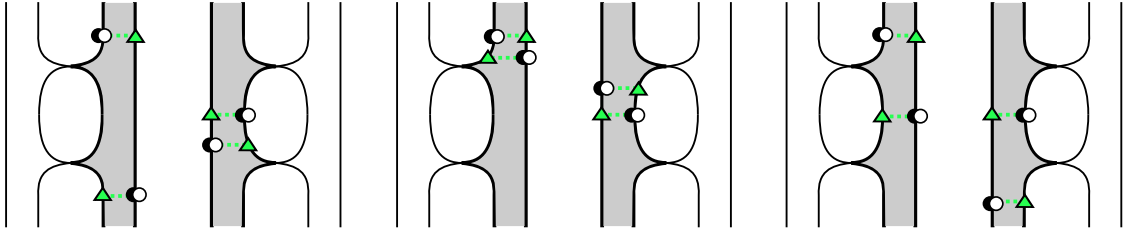}
    \caption{Domains counted in the proof of Lemma~\ref{lem:hbgb} when $m=0$. The first and second domains each contribute to $\Hbgb\circ\dxminus$. The third and fourth domains each contribute to $\dxminus\circ\Hbgb$. The fifth and sixth domains each contribute to $\Phi_{\gamma\beta}\circ\Phi_{\beta\gamma}$. }
    \label{fig:hex0}
\end{figure}

As a result, any domain connecting $\x$ to $\z$ for $2\le m\le 4$ has exactly two decompositions, so cancels in Equation~\ref{eq:hexhomotopy}. Furthermore, there are no domains connecting $\x$ to $\z$ for $m=1$, and there is a unique domain connecting $\x$ to $\x$, counted with coefficient 1, so we have $\Phi_{\gamma\beta}\circ\Phi_{\beta\gamma}+\dxminus\circ \Hbgb+\Hbgb\circ\dxminus$ is the identity map. 
\end{proof}

\begin{proof}[Proof of Proposition~\ref{prop:commutation}]
Suppose $\gridg$ and $\gkprime$ differ by a commutation of columns. Then Lemma~\ref{lem:phibetagamma} says $\Phi_{\beta\gamma}$ as defined above is a chain map. Lemma~\ref{lem:hbgb} says that $\Phi_{\beta\gamma}$ has as homotopy inverse $\Phi_{\gamma\beta}$, via the homotopy $\Hbgb$. That $\Phi_{\beta\gamma}$ is trigraded is the content of Lemma~\ref{lem:pentagongrading}. Therefore $C^-(\gridg)$ and $C^-(\gkprime)$ are homotopy equivalent as chain complexes over $\Ztwo[V_0,\ldots,V_{n-1}]$, and so by Lemma~\ref{lem:componentvariablesarehomotopic}, as chain complexes over $\Ztwo[U_1,\ldots,U_\ell]$. If instead, $\gridg$ and $\gkprime$ differ by a commutation of rows, analogous proofs to the above with the diagrams rotated $90^{\circ}$ provide the homotopy equivalence. 
\end{proof}

%% file: Sections/Invariance/switches.tex
 Another move that can be done to a grid diagram that preserves the knot type is a \emph{switch}. Let $\gridg$ be a grid diagram representing a link $L$. Consider a pair of consecutive columns in $\gridg$ with the first column bounded by $\beta_{i-1}$ and $\beta_i$, and the second column bounded by $\beta_i$ and $\beta_{i+1}$. If the $X\in\XX$ marking in one of these columns occurs in the same row as the $O\in\OO$ marking in the other, and these markings are in adjacent regions of $T^2-\alphas-\betas$, we call this pair of columns \emph{special}, in analogy to \cite{gridhomology}. A \emph{switch move of columns} consists of exchanging the markings in a special pair of columns. Similarly, if two markings from two consecutive columns are in the same row (which implies they are in adjacent regions of $T^2-\alphas-\betas$), we call this pair of rows \emph{special}, and a \emph{switch move of rows} consists of exchanging the markings in a special pair of rows. Let $\gkprime$ be the grid diagram gotten by performing a switch move to the grid diagram $\gridg$, and let $L'$ be the link represented by $\gkprime$. Then $L$ and $L'$ are smoothly isotopic. This is because any switch can be represented by a sequence of commutations, stabilizations, and destabilizations. An example is given in Figure~\ref{fig:switchisgridmove}. 
 
 \begin{figure}[h]
    \centering
    \includegraphics[height=3cm]{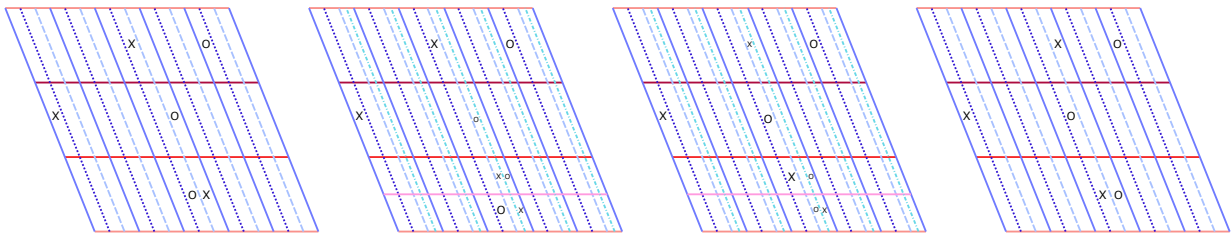}
    \caption{A switch of the second and third columns in $L(5,2)$ of grid index 3 represented as a stabilization, followed by a commutation of the third and fourth columns, and then a destabilization. }
    \label{fig:switchisgridmove}
\end{figure}

\begin{prop}\label{prop:switch}
Suppose $\gridg$ and $\gkprime$ are grid diagrams for an $\ell$-component link which are connected by a switch move. Then $(C^-(\gridg),\dxminus)$ and $(C^-(\gkprime),\dxminusprime)$ are homotopy equivalent as chain complexes over $\Ztwo[U_1,\ldots,U_\ell]$. 
\end{prop}

\begin{proof}
Perturb $\beta_i$ and $\beta_i'$ so that we can draw both $\gridg$ and $\gkprime$ on the same grid diagram, as in Section~\ref{sec:commutation}, with the perturbation understood as in Figure~\ref{fig:switchmove}. The same proof that the chain complexes associated to two grid diagrams connected by a commutation move are homotopy equivalent applies here, mutatis mutandis.  
\end{proof}

\begin{figure}[h]
    \centering
    \begin{tikzpicture}
    \node at (0,0) {\includegraphics[height=6cm]{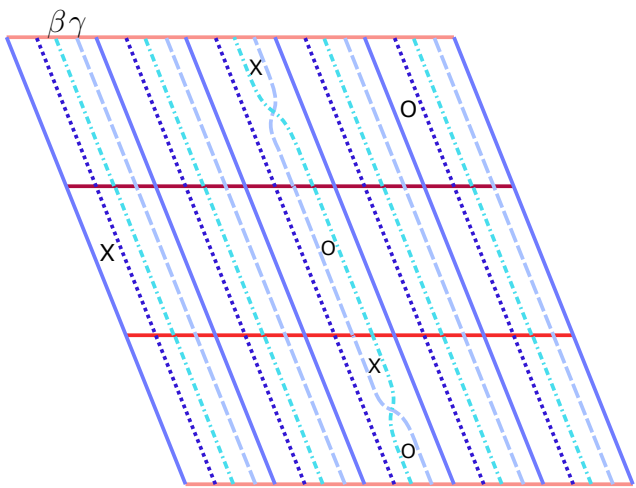}};
    \node at (-3.5,3.15) {$\beta$};
    \node at (-3.22,3.1) {$\gamma$};
    \end{tikzpicture}
    \caption{A switch of the second and third columns in $L(5,2)$ of grid index 3. Differing from a commutation, we have two markings in the same row which are being switched. We perturb $\beta$ and $\gamma$ to separate these within the row. }
    \label{fig:switchmove}
\end{figure}

%% file: Sections/Invariance/stabilization.tex
We will follow \cite{gridhomology} in their discussion of stabilization, including terminology. There are three steps to a \emph{stabilization move}. First, choose a region of $T^2-\alphas-\betas$ containing a marking, and remove the marking as well as the other markings in the associated row and column. Second, split the row and column in two by adding a curve to each of $\alphas$ and $\betas$. Finally, add markings to the two new rows and columns to form a grid diagram. One can check that there are four such ways to add markings to yield a valid grid diagram for each of the original choices of marking, so eight total stabilization moves. These are given with names in Figure~\ref{fig:8stabilizations}. 
The inverse of such a move is called a \emph{destabilization}. 

\begin{figure}[ht!]
    \begin{center}
    \begin{tikzpicture}
    \node[anchor=south west,inner sep=0] at (0,0) {\includegraphics[width=16cm]{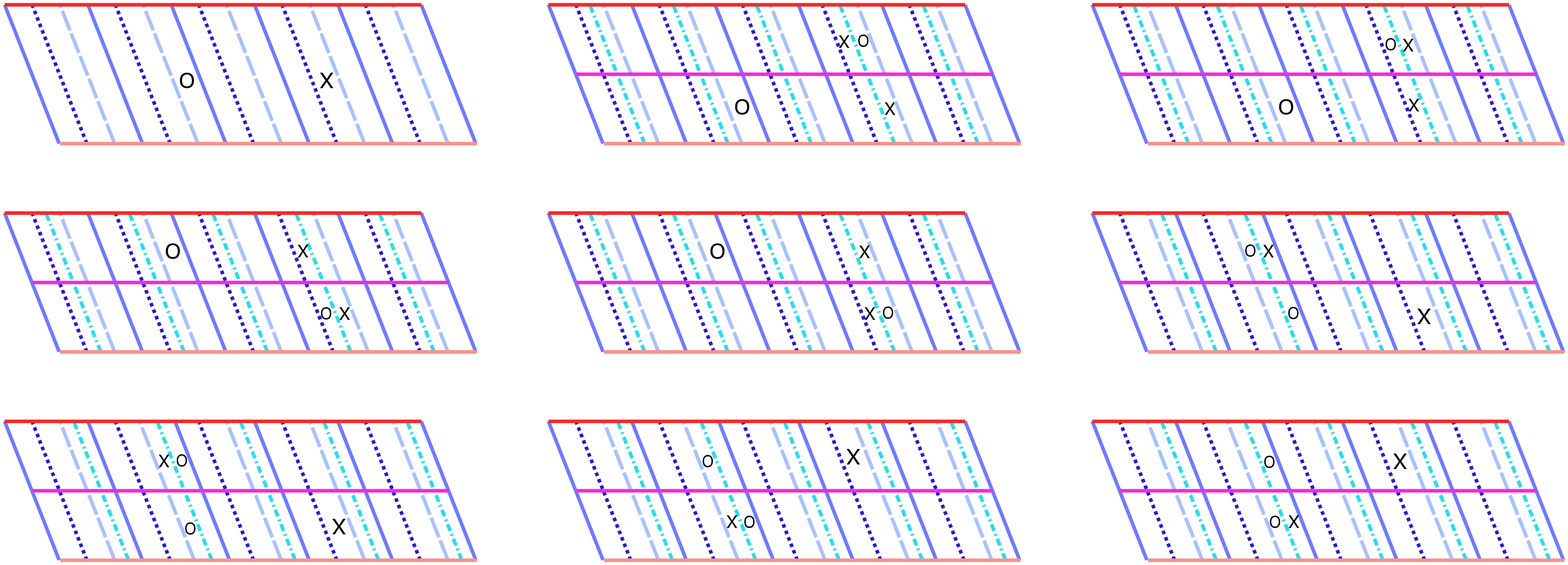}};
    \node at (8.2,4.1) {\small X:SW};
    \node at (13.7,4.1) {\small X:SE};
    \node at (8.2,1.9) {\small X:NW};
    \node at (8.2,-.2) {\small O:NE};
    \node at (2.5,1.9) {\small X:NE};
    \node at (13.7, 1.9) {\small O:SW};
    \node at (2.5, -.2) {\small O:SE};
    \node at (13.7, -.2) {\small O:NW};
    \end{tikzpicture}
    \caption{From the local picture of a portion of a grid diagram for $L(5,2)$ of grid index 3 shown in the top left corner, any of these eight stabilizations are possible, with names given. }
    \label{fig:8stabilizations}
    \end{center}
\end{figure}

We will prove the following:

\begin{prop}\label{prop:stabilization}
Suppose $\gridg$ and $\gkprime$ are grid diagrams for an $\ell$-component link which are connected by a stabilization. Then $(C^-(\gridg,\dxminus)$ and $C^-(\gkprime),\dxminusprime)$ are homotopy equivalent as chain complexes over $\Ztwo[U_1,\ldots,U_\ell]$. 
\end{prop}

We will first prove the above in the specific case of an $\XSW$ stabilization, which can be seen with notation in Figure~\ref{fig:stabilization}. 

For notation, let $\gridg$ be a grid diagram for $L$ with grid number $n$, and suppose the marking we chose was $X_{n-1}\in\XX$. Denote by $\alpha_n$ and $\beta_n$ the new curves in the $\XSW$ stabilization, and by $X_n$ and $O_n$ the new markings added in the $\XSW$ stabilization. Denote by $c$ the intersection point of $\alpha_n$ and $\beta_n$ at the lower right corner of the region containing $X_{n-1}$. Further, let $\genG$ be the generators of $C^-(\gridg)$, which we will denote $C$, and $\genG'$ be the generators of $C^-(\gkprime)$, which we will denote $C'$. 

\begin{figure}[ht]
    \centering
    \begin{tikzpicture}
    \node at (0,0) {\includegraphics[width=16cm]{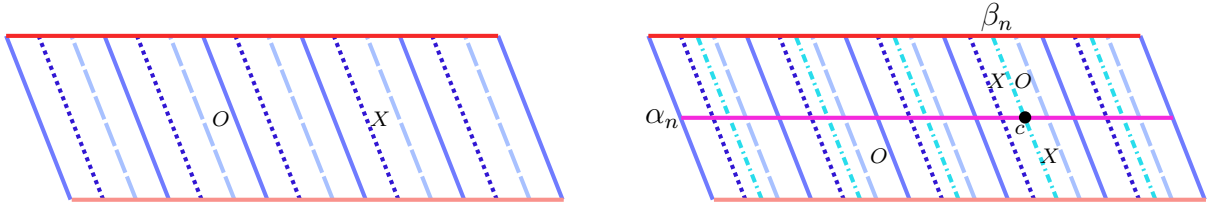}};
    \node at (-5.12,0) {\tiny $O$};
     \node at (-3,0) {\tiny $X$};
     \node at (3.63,-.5) {\tiny $O$};
     \node at (5.9,-.5) {\tiny $X$};
     \node at (5.2,.5) {\tiny $X$};
     \node at (5.55,.5) {\tiny $O$};
     \node at (5.5,-.15) {\tiny $c$};
     \node at (.75,0) {$\alpha_n$};
     \node at (5.2,1.35) {$\beta_n$};
    \end{tikzpicture}
    \caption{An example of $\XSW$ stabilization, including labels for the notation we will use throughout Section~\ref{sec:stabilization}.}
    \label{fig:stabilization}
\end{figure}

First we note that $C$ and $C'$ are chain complexes over different rings. To allow us to work uniformly over one ring, we formally introduce a variable $U_n$ to $C$, denoting the resulting chain complex $C[V_n]$. As a module this is $C\otimes_{\Ztwo [V_0,\ldots,V_{n-1}]}\Ztwo [V_0,\ldots,V_n]$. The action of $\Ztwo[V_0,\ldots,V_n]$ on this module is defined by by $a\cdot(\sum_i c_i\otimes d_i)=\sum_i c_i\otimes ad_i$ for $a\in \Ztwo[V_0,\ldots,V_n]$ and $\sum_i c_i\otimes d_i\in C\otimes \Ztwo[V_0,\ldots,V_n]$. 
The differential is defined by  $\dxminus(\sum c_i\otimes d_i)=\sum\dxminus(c_i)\otimes d_i$, and the grading can similarly be extended so that $V_n$ has Maslov grading $-2$ and Alexander grading $-1$, like the other variables. 

We would like to show that $C$ and $C'$ are quasi-isomorphic. We will show that there is a third complex $C''$ such that $C'$ is quasi-isomorphic to $C''$ and $C''$ is quasi-isomorphic to $C$. The second step we get for free. 

\begin{lem}[cf.~{\cite[Lemma 5.2.16]{gridhomology}}]\label{lem:freeconelemma}
Let $D$ be a trigraded chain complex over $\Ztwo [V_0,\ldots,V_{n-1}]$. Then $\Cone(V_n-V_{n-1}:D[V_n]\to D[V_n])$ is quasi-isomorphic as a trigraded $\Ztwo [V_0,\ldots,V_{n-1}]$ module to $D$. 
\end{lem}

Let $C'':=\Cone(V_n-V_{n-1}:C[V_n]\to C[V_n])$. Then by the above we have that $C''$ is quasi-isomorphic to $C$. Most of the work now goes into showing that $C'$ is quasi-isomorphic to $C''$. 

We note that for a trigraded chain complex $D$, following \cite{gridhomology}, we denote grading shift by $D\gop a,b,c\gcl$, so $D\gop a,b,c\gcl_{r,s,t}=D_{r-a,s-b,t-c}$. 

 We draw our notation from Figure \ref{fig:stabilization}, following \cite{gridhomology}. We can decompose $\genG'$ into those generators which contain $c$ and those which do not. Let $\bI:=\{\x\in\genG':c\in\x\}$ and $\mathbf{N}:=\{\x\in\genG':c\not\in\x\}$; then $\genG'=\bI\cup\mathbf{N}$, and as a $\Ztwo[V_0,\ldots,V_n]$ module, $C'=\tbI\oplus\tbN$, where $\tbI$ and $\tbN$ are the spans of $\bI$ and $\mathbf{N}$. For $\x\in \mathbf{N}$, $\y\in\bI$, we have that $\ePG(\x,\y)$ is empty, as any parallelogram connecting $\x$ to $\y$ must contain $X_{n-1}$ or $X_n$, so $\tbN$ is a subcomplex of $C'$. Let $\dnn$ be the restriction of $\dxminus$ to $\tbN$, $\din$ the restriction of $\dxminus$ to the domain $\tbI$ and codomain $\tbN$, and $\dii$ the restriction of $\dxminus$ to the domain and codomain $\tbI$. Therefore, as $$\dxminus=\begin{pmatrix} \partial_{\tbI}^{\tbI} & 0 \\
 \partial_{\tbI}^{\tbN} & \partial_{\tbN}^{\tbN}
 \end{pmatrix},$$ we get that $C'=\Cone(\partial_{\tbI}^{\tbN})$. Our goal, then, is to provide a quasi-isomorphism from $C'=\Cone(\din)$ to $C''=\Cone(V_n-V_{n-1})$. We will do this in components, and then appeal to the following algebraic lemma: 
 
 \begin{lem}[{\cite[Lemma 5.2.12]{gridhomology}}]\label{lem:conealgebraiclemma}
 A commutative diagram of chain complexes of trigraded $R$-modules, \[
 \begin{tikzcd}
 C\arrow[r,"f"]\arrow[d,"\phi"] & C'\arrow[d,"\phi'"]\\
 E\arrow[r,"g"]&E'
 \end{tikzcd}
 \]where the horizontal maps are homogenous of degree $(m,t,r)$, and the vertical maps are trigraded, induces a chain map $\Phi:\Cone(f)\to\Cone(g)$. If $\phi$ and $\phi'$ are quasi-isomorphisms, then so is $\Phi$. 
 \end{lem}
 
 We are looking to fill in the vertical maps on the following square with quasi-isomorphisms that make the square commute, at which point we can appeal to the above lemma: 
 \[
 \begin{tikzcd}
 (\tbI,\dii)\arrow[r,"\din"]\arrow[d,""] & (\tbN,\dnn)\arrow[d,""]\\
 (C[V_n]\gop1,1,0\gcl,\dxminus)\arrow[r,"V_n-V_{n-1}"]&(C[V_n],\dxminus).
 \end{tikzcd}\]
 
 \begin{lem}\label{lem:Iquasi-isomorphism}
 The identification of generators $\x\in \bI$ with $\x\setminus c\in \genG$ induces an isomorphism of trigraded chain complexes $e:(\tbI,\partial_{\tbI}^{\tbI})\to C[V_n]\gop1,1,0\gcl$ over $\Ztwo [V_0,\ldots,V_n]$. 
 \end{lem}
 
 Before we continue, let's address the gradings.  
 \begin{lem}\label{lem:Igradings}
 For $\x\in \genG$, let $\x'=\x\cup c\in \bI$. Then $\spincS(\x)=\spincS(\x')$, $\maslov(\x)=\maslov(\x')+1$ and $\alex(\x)=\alex(\x')+1$. 
 \end{lem}
 \begin{proof}
 Fix a fundamental domain so that $c$ is in the top right of the diagram. 
 
 Let's first consider $\spincS$. We see that $a_i^{\x}=a_i^{\x'}$ and $a_i^{\x_{\OO}}=a_i^{\x_{\OO}'}$ for $0\le i<n$ and $a_n^{\x}=a_n^{\x'}=a_n^{\x_{\OO}}=a_n^{\x_{\OO}'}=p-1$, so we get $\spincS(\x)=\spincS(\x')$. 
 
 Now let's consider $\maslov$. Let $O_i\in\OO$ and $x_i\in\x$. Denote by $\tO_i^k$ and $\tx_i^k$ the component in the strip $C_{p,q}^k$. 
 
 Recall that $\tilde{\maslov}$ is the signed sum of four terms. We can consider what happens to each of these terms for $\x\cup c$ and $\OO\cup O_n$ as compared to $\x$ and $\OO$. For example, each pair counted by $\I(\x\cup c,\x\cup c)$ which is not counted by $\I(\x,\x)$ is of the form $(\tx_i^r,\tc^s)$, $(\tc^r,\tx_i^s)$ for $r\le s$ and some $0\le i <n$, or of the form $(\tc^r,\tc^s)$, for $r<s$. When we consider these new pairs counted by each of the four terms, we see that they cancel in signed pairs, as follows: 
 
 \begin{align*}
 &\textbf{\underline{For $0\le r\le s<p$} }&&\textbf{\underline{For $0\le r< s<p$}}\\
     &(\tx_i^r,\tc^s)\longleftrightarrow(\tx_i^r,\tO_n^s)&&(\tc^r,\tc^s)\longleftrightarrow(\tc^r,\tO_n^s)\\
     &(\tc^r,\tx_j^s)\longleftrightarrow(\tO_n^r,\tx_j^s)&&(\tO_n^r,\tO_n^s)\longleftarrow(\tO_n^r,\tc^s)\\
     &(\tO_k^r,\tO_n^s)\longleftrightarrow(\tO_k^r,\tc^s)&&\\
     &(\tO_n^r,\tO_l^s)\longleftrightarrow(\tc^r,\tO_l^s)&&
 \end{align*}
 
 For each of the above pairs, the terms on the left are counted with positive sign in the definition of $\maslov$, and the terms on the right are counted with negative sign. 
 
 Unaccounted for in the above table are the pairs of the form $(\tc^r,\tO_n^r)$ for $0\le r<p$, which match with nothing, and of which there are $p$. Thus overall we have $\tilde{\maslov}(\x\cup c)=\tilde{\maslov}(\x)-p$, so $\maslov(\x)=\maslov(\x')+1$. A similar analysis of cancelling pairs yields that $\alex(\x)=\alex(\x')+1$. 
 \end{proof}
 
 \begin{proof}[Proof of Lemma \ref{lem:Iquasi-isomorphism}]

 This proof is essentially the same as in the case of knots in $S^3$, see {\cite[Lemma 5.2.18]{gridhomology}}. There is a bijection of generators between $C[V_n]$ and $\tbI$, and a bijection between empty parallelograms disjoint from $\XX$ in $\gridg$ and empty parallelograms disjoint from $\XX'$ and $c$ in $\gkprime$. Denote this bijection $e$. Note that for any $r\in\partial_{\tbI}^{\tbI}$, we have $n_{O_n}(r)=0$, and $n_{O_i}(r)=n_{O_i}(e(r))$ for all $0\le i<n$. That this isomorphism of chain complexes respects the gradings is the result of Lemma~\ref{lem:Igradings}. 
 \end{proof}
 
 Next we wish to show that we have a trigraded quasi-isomorphism from $(\tbN,\partial_{\tbN}^{\tbN})$ to $C[V_n]$ as chain complexes over $\Ztwo[V_0,\ldots,V_n]$, which we will do by giving an explicit trigraded chain homotopy equivalence between $(\tbN,\partial_{\tbN}^{\tbN})$ and $(\tbI\gop -1,-1,0\gcl,\partial_{\tbI}^{\tbI})$, and composing with the map $e$ above. 
 
 Define the map $\Phi_{X_n}^{\tbI}:\tbN\to\tbI$ by $$\Phi_{X_n}^{\tbI}(\x):=\sum_{\y\in\tbI}\sum_{\substack{r\in\ePG(\x,\y)\\\Int(r)\cap\XX=X_n}}V_0^{n_{O_0}(r)}\cdots V_{n}^{n_{O_{n}}(r)}\y. $$
 
 \begin{lem}\label{lem:phixni}
 The map $\Phi_{X_n}^{\tbI}$ is a chain map. 
 \end{lem}
 \begin{proof}
We wish to show $\Phi_{X_n}^{\tbI}\circ \dnn=\dii\circ \Phi_{X_n}^{\tbI}$. This map is very similar to the map used in the proof of Lemma~\ref{lem:componentvariablesarehomotopic}. As there, we are analyzing domains which are the juxtaposition of two embedded parallelograms, and possible placements of $X_n$ allowing the domains to contribute to either $\Phi_{X_n}^{\tbI}\circ\dnn$ or $\dii\circ\Phi_{X_n}^{\tbI}$. The restrictions given by the placement of $X_{n-1}$ and by the fact that the codomain of $\Phi_{X_n}^{\tbI}$ is $\tbI$ and not $C'$ make it feasible to present all such domain and placement combinations.  Figure~\ref{fig:phixnialternatedecompositions} contains those domains and placements for which one decomposition contributes to each of $\Phi_{X_n}^{\tbI}\circ\dnn$ and $\dii\circ\Phi_{X_n}^{\tbI}$. Figure~\ref{fig:phixnisamedecompositions} contains those domains and placements for which both decompositions are of the same type, and thus cancel in characteristic two.
\end{proof}

\begin{figure}[ht]
    \begin{center}
    \begin{tikzpicture}
    \node[anchor=south west,inner sep=0] at (0,0) {\includegraphics[height=5.5cm]{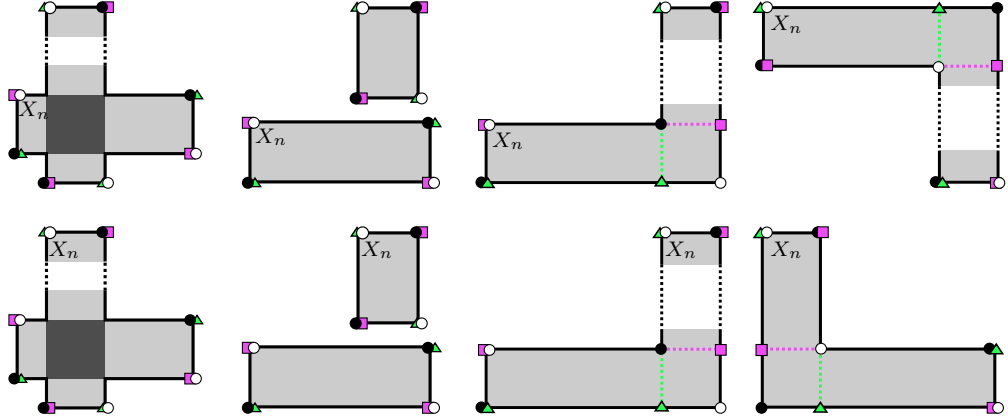}};
    \node[label=above right:{\tiny $X_n$}] at (-.15,3.65){};
    \node[label=above right:{\tiny $X_n$}] at (.25,1.8){};
    \node[label=above right:{\tiny $X_n$}] at (3,3.3){};
    \node[label=above right:{\tiny $X_n$}] at (6.15,3.25){};
    \node[label=above right:{\tiny $X_n$}] at (9.85,4.8){};
    \node[label=above right:{\tiny $X_n$}] at (4.41,1.8){};
    \node[label=above right:{\tiny $X_n$}] at (8.45,1.8){};
    \node[label=above right:{\tiny $X_n$}] at (9.85,1.8){};
    \end{tikzpicture}
    \end{center}
    \caption{Certain domains of Proposition~\ref{prop:d2=0} given with possible placements of a marking so as to be counted in the proof of Lemma~\ref{lem:phixni}. Each domain and marking combination has one decomposition which contributes to each term: $\Phi_{X_n}^{\tbI}\circ\partial_{\tbN}^{\tbN}$ and $\partial_{\tbI}^{\tbI}\circ\Phi_{X_n}^{\tbI}$.}
    \label{fig:phixnialternatedecompositions}
\end{figure}

\begin{figure}[ht]
    \begin{center}
    \begin{tikzpicture}
    \node[anchor=south west,inner sep=0] at (0,0) {\includegraphics[height=\figheight]{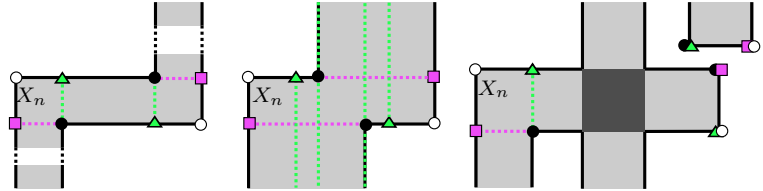}};
    \node[label=above right:{\tiny $X_n$}] at (-.2,.85){};
    \node[label=above right:{\tiny $X_n$}] at (2.95,.85){};
    \node[label=above right:{\tiny $X_n$}] at (5.95,.95){};
    \end{tikzpicture}
    \end{center}
    \caption{Certain domains of Proposition~\ref{prop:d2=0} given with possible placements of a marking so as to be counted in the proof of Lemma~\ref{lem:phixni}. Both decompositions of each domain and marking combination are of the same type, contributing to $\Phi_{X_n}^{\tbI}\circ\dnn$.}
    \label{fig:phixnisamedecompositions}
\end{figure}

 \begin{lem}\label{lem:phixndegree}
 The map is of tridegree $(-1,-1,0)$. 
 \end{lem}
 \begin{proof}
 This follows directly from Proposition \ref{prop:parallelogramgradings}. 
 \end{proof}
 
 The map $\Phi_{X_n}^{\tbI}$ is a chain homotopy equivalence. We will prove this by providing a homotopy inverse, and an explicit homotopy. The homotopy inverse is $\Phi_{O_n}:\tbI\to\tbN$, which is defined on generators by \[
 \Phi_{O_n}(\x)=\sum_{\y\in\tbN}\sum_{\substack{r\in\ePG(\x,\y)\\
 \Int(r)\cap\XX=\emptyset\\
 O_n\in r}}V_0^{n_{O_0}(r)}\cdots V_{n-1}^{n_{O_{n-1}}(r)}\y.
 \]
 
 \begin{lem}
 The map $\Phi_{O_n}$ is a chain map of tridegree $(1,1,0)$. 
 \end{lem}
 \begin{proof}
 This is a similar case analysis to Lemma \ref{lem:phixni}, analyzing possible placements of $O_n$ in the domains of Proposition \ref{prop:d2=0}. The grading follows directly from Proposition~\ref{prop:parallelogramgradings}. 
  
 \end{proof}
 
 \begin{lem}\label{lem:phixnphionidentity}
    On $\tbI$, the composition $\Phi_{X_n}^{\tbI}\circ\Phi_{O_n}$ is the identity.  
 \end{lem}
 
 \begin{proof}
 Once more, this is a case analysis of the domains of Proposition \ref{prop:d2=0}, analyzing possible placements of $X_n$ and $O_n$. Let $\phi$ be a domain connecting $\x$ to $\z$ for $\x,\z\in\genG$. It becomes clear that not a single domain in Figures \ref{fig:dsquared4}, \ref{fig:dsquared3}, or \ref{fig:dsquared2} permit placement of $X_n$ and $O_n$ in such a way that either decomposition is in $\Phi_{X_n}^{\tbI}\circ\Phi_{O_n}$. The only possibility, then, is that $\x=\z$ and $\phi$ is an annulus of width one. The annuli of height one are prohibited by virtue of either not containing $X_n$ or not containing $O_n$. Thus out of each generator $\x$ is a single annulus covering only the markings $X_n$ and $O_n$, decomposed in a unique way, connecting $\x$ to $\x$. Note that the map $\Phi_{O_n}$ is defined in such a way as to not count a coefficient coming from $O_n$, and neither $\Phi_{X_n}^{\tbI}$ nor $\Phi_{O_n,X_n}$ count a coefficient coming from $X_n$, so this map is exactly the identity. 
 \end{proof}

 \begin{lem}\label{lem:phixnquasiisomorphism}
 On $\tbN$, the composition $\Phi_{O_n}\circ\Phi_{X_n}^{\tbI}$ is homotopic to the identity. 
 \end{lem}
 
 \begin{proof}
 We will first define the homotopy equivalence between $\Phi_{O_n}\circ\Phi_{X_n}^{\tbI}$ and the identity. We will will denote it by $\Phi_{O_n,X_n}$. It is defined on generators $\x\in\tbN$ by \[
 \Phi_{O_n,X_{n}}(\x)=\sum_{\y\in\tbN}\sum_{
 \substack{r\in\ePG(\x,\y)\\
 \Int(r)\cap\XX=X_{n-1}\\
 O_n\in r}}
 V_0^{n_{O_0}(r)}\cdots V_{n-1}^{n_{O_{n-1}}(r)}\y.
 \]

 We will prove that \begin{equation}\label{eq:onxnhomotopy} \Phi_{O_n}\circ \Phi_{X_n}^{\tbI}+\partial_{\tbN}^{\tbN}\circ \Phi_{O_n,X_n}+\Phi_{O_n,X_n}\circ\partial_{\tbN}^{\tbN}=\Id_{\tbN}.\end{equation}
 
 Again, we are analyzing the domains in the proof of Proposition~\ref{prop:d2=0}, subject to the constraints given by the definitions of the maps above. Let $\phi$ be a domain connecting $\x$ to $\z$ for $\x,\z\in\genG$. 
 
 Figure~\ref{fig:phionxnalternate}, Figure~\ref{fig:phionxnsamedecompositions}, and Figure~\ref{fig:phionxnmatchingdecompositions} display each possible domain and placement of $O_n$ and $X_n$ which allow for exactly two decompositions counted by the left hand side of (\ref{eq:onxnhomotopy}). Their decompositions are described in the proof of Proposition~\ref{prop:d2=0}; as a result, they drop out of the left hand side of (\ref{eq:onxnhomotopy}) in characteristic two. 
 
 \begin{figure}[ht]
   
   \centering
   \begin{tikzpicture}
    \node at (0,0) {\includegraphics[height=8cm]{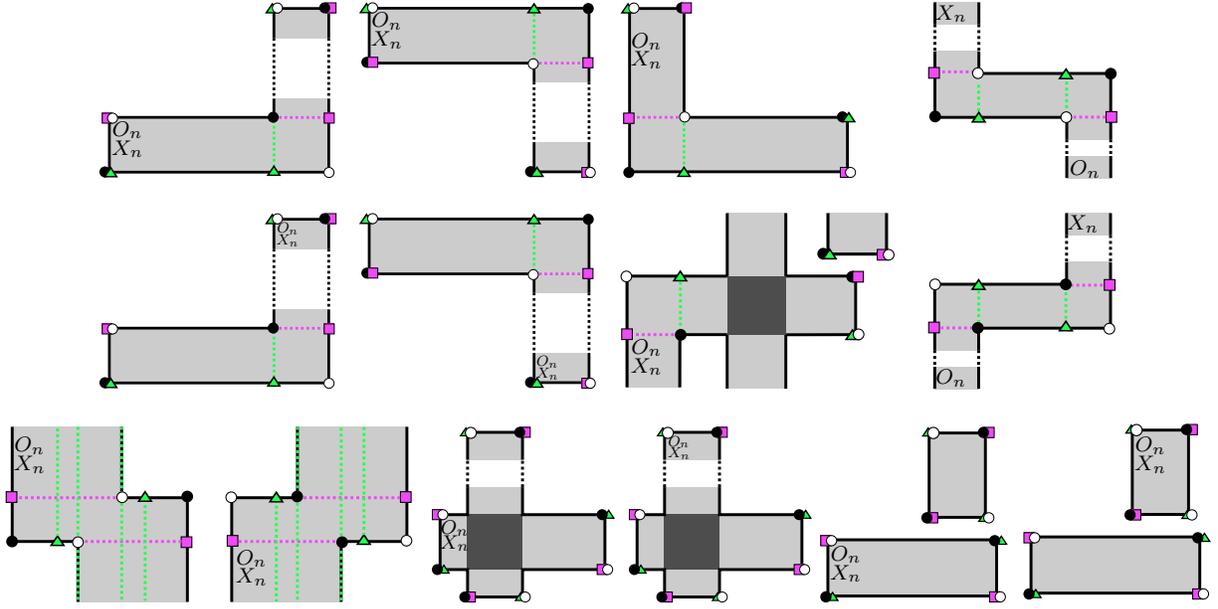}};
    \node at (-6.4,2.3) {\tiny $O_n$};
    \node at (-6.4,2.05) {\tiny $X_n$};
    \node at (-2.95,3.75) {\tiny $O_n$};
    \node at (-2.95,3.5) {\tiny $X_n$};
    \node at (0.5,3.48) {\tiny $O_n$};
    \node at (0.5,3.23) {\tiny $X_n$};
    \node at (6.33,1.78) {\tiny $O_n$};
    \node at (4.55,3.85) {\tiny $X_n$};
    \node[scale=0.7] at (-4.27,.99) {\tiny $O_n$};
    \node[scale=0.7] at (-4.27,.83) {\tiny $X_n$};
    \node[scale=0.7] at (-.82,-.78) {\tiny $O_n$};
    \node[scale=0.7] at (-.82,-.94) {\tiny $X_n$};
    \node at (0.5,-.6) {\tiny $O_n$};
    \node at (0.5,-.85) {\tiny $X_n$};
    \node at (4.55,-1) {\tiny $O_n$};
    \node at (6.3,1.05) {\tiny $X_n$};
    \node at (-7.7,-1.9) {\tiny $O_n$};
    \node at (-7.7,-2.15) {\tiny $X_n$};
    \node at (-4.75,-3.4) {\tiny $O_n$};
    \node at (-4.75,-3.65) {\tiny $X_n$};
    \node[scale=0.93] at (-2.05,-3) {\tiny $O_n$};
    \node[scale=0.93] at (-2.05,-3.2) {\tiny $X_n$};
    \node[scale=0.7] at (.93,-1.85) {\tiny $O_n$};
    \node[scale=0.7] at (.93,-2.) {\tiny $X_n$};
    \node at (3.15,-3.35) {\tiny $O_n$};
    \node at (3.15,-3.6) {\tiny $X_n$};
    \node at (7.2,-1.9) {\tiny $O_n$};
    \node at (7.2,-2.15) {\tiny $X_n$};
   \end{tikzpicture}
    \caption{Certain domains of Proposition~\ref{prop:d2=0} given with possible placements of $X_n$ and $O_n$ markings so as to be counted in the proof of Lemma~\ref{lem:phixnquasiisomorphism}. One decomposition contributes to each of $\dnn\circ\Phi_{O_n,X_n}$ and $\Phi_{O_n,X_n}\circ\dnn$. }
    \label{fig:phionxnalternate}
\end{figure}

\begin{figure}[ht]
 \centering
 \begin{tikzpicture}
  \node at (0,0) {\includegraphics[height=5.5cm]{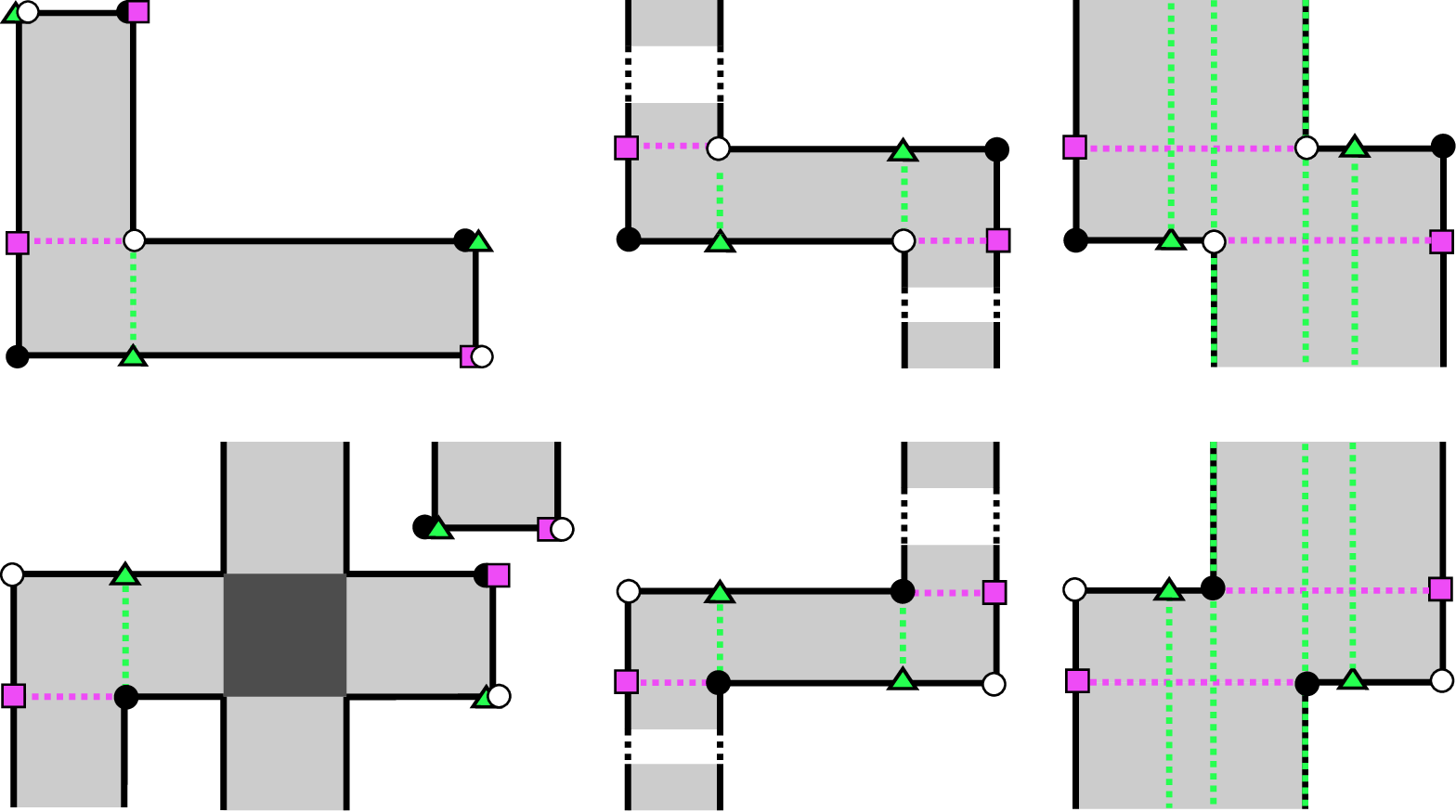}};
    \node at (-4.55,.9) {\tiny $O_n$};
    \node at (-4.55,.9-.25) {\tiny $X_n$};
    \node at (-.4,1.6) {\tiny $O_n$};
    \node at (-.4,1.6-.25) {\tiny $X_n$};
    \node at (2.6,1.6) {\tiny $O_n$};
    \node at (2.6,1.6-.25) {\tiny $X_n$};
    \node at (-4.6,-1.3) {\tiny $O_n$};
    \node at (-4.6,-1.3-.25) {\tiny $X_n$};
    \node at (-.45,-1.42) {\tiny $O_n$};
    \node at (-.45,-1.42-.25) {\tiny $X_n$};
    \node at (2.58,-1.42) {\tiny $O_n$};
    \node at (2.58,-1.42-.25) {\tiny $X_n$};
 \end{tikzpicture}
    \caption{Certain domains of Proposition~\ref{prop:d2=0} given with possible placements of a marking so as to be counted in the proof of Lemma~\ref{lem:phixnquasiisomorphism}. Both decompositions of each domain and marking combination are of the same type. For the top three domains, both decompositions contribute to $\dnn\circ\Phi_{O_n,X_n}$. For the bottom three domains, both decompositions contribute to  $\Phi_{O_n,X_n}^{\tbI}\circ\dnn$.}
    \label{fig:phionxnsamedecompositions}
\end{figure}

\begin{figure}[ht]
 \centering
 \begin{tikzpicture}
    \node at (0,0) {\includegraphics[height=5.5cm]{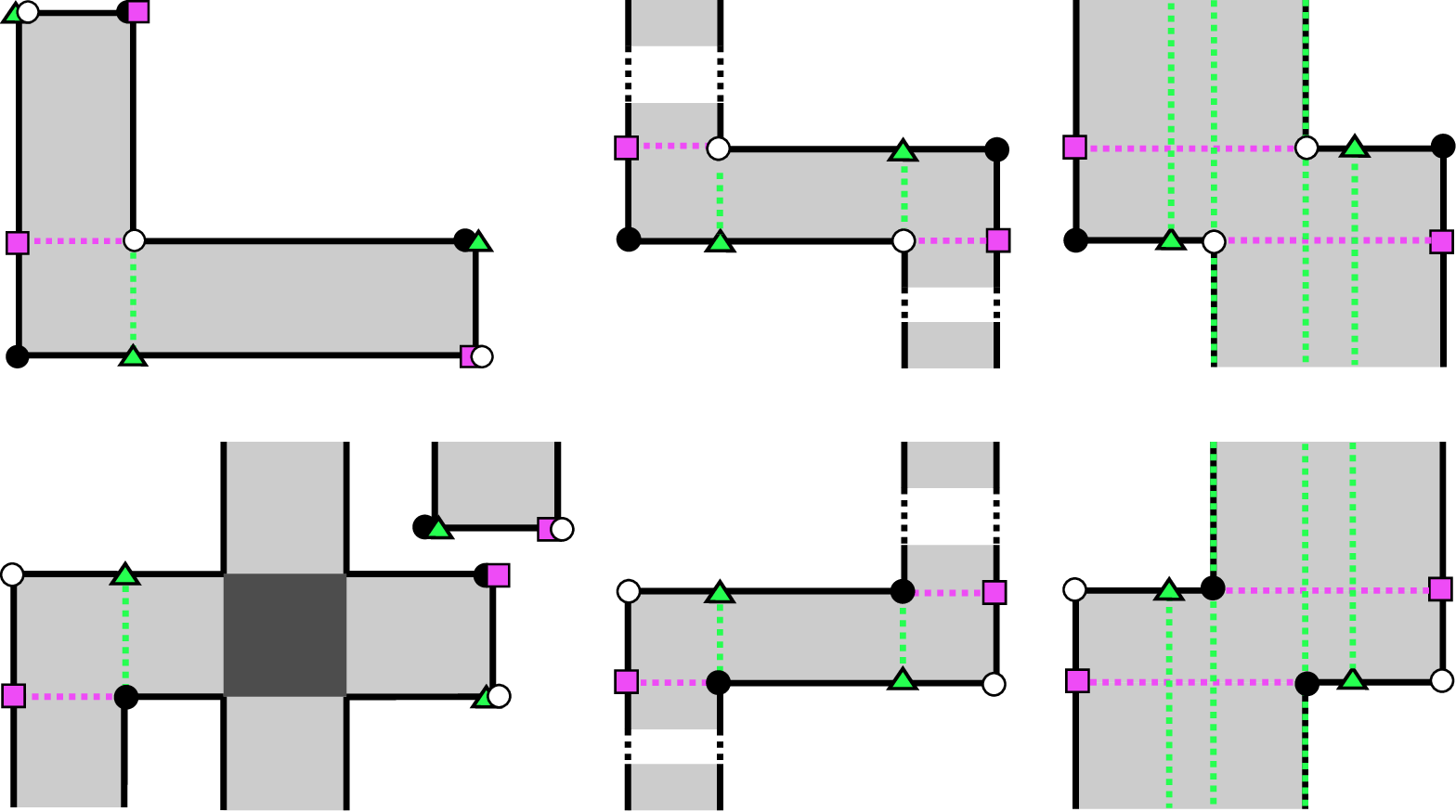}};
    \node at (-4.55,1.3) {\tiny $O_n$};
    \node at (-4.55,.9) {\tiny $X_n$};
    \node at (-.45,1.9){\tiny $O_n$};
    \node at (-.45,1.6){\tiny $X_n$};
    \node at (2.6,1.95) {\tiny $O_n$};
    \node at (2.6,1.55) {\tiny $X_n$};
    \node at (-4.6,-1.8) {\tiny $O_n$};
    \node at (-4.6,-2.2) {\tiny $X_n$};
    \node at (-.45,-1.7) {\tiny $O_n$};
    \node at (-.45,-2.05) {\tiny $X_n$};
    \node at (2.58,-1.7) {\tiny $O_n$};
    \node at (2.58,-2.05) {\tiny $X_n$};
 \end{tikzpicture}
    \caption{Certain domains of Proposition~\ref{prop:d2=0} given with possible placements of a marking so as to be counted in the proof of Lemma~\ref{lem:phixnquasiisomorphism}. Each has one decomposition contributing to $\Phi_{O_n}\circ\Phi_{X_n}^{\tbI}$. For the top three domains, the other decomposition contributes to $\dnn\circ\Phi_{O_n,X_n}$. For the bottom three domains, the other decomposition contributes to $\Phi_{O_n,X_n}\circ\dnn$.}
    \label{fig:phionxnmatchingdecompositions}
\end{figure}
 
 In the case where $\x=\z$, there is a unique width one annulus connecting $\x$ to $\x=\z$ which contains $O_n$ and $X_n$; which of the above three terms account for this annulus depends on the relative position of $O_n$, $X_n$, and components of $\x$. The three possible cases are given in Figure~\ref{fig:phionxnannuli}. As a result, we get that the left hand side of (\ref{eq:onxnhomotopy}) exactly equals the identity; any domain connecting $\x$ to $\z\ne\x$ has exactly two cancelling decompositions, and the unique annulus connecting $\x$ to $\z=\x$ has a unique decomposition counted with coefficient one.

\begin{figure}[ht]
    \centering
    \begin{tikzpicture}
    \node at (0,0) {\includegraphics[height=3.5cm]{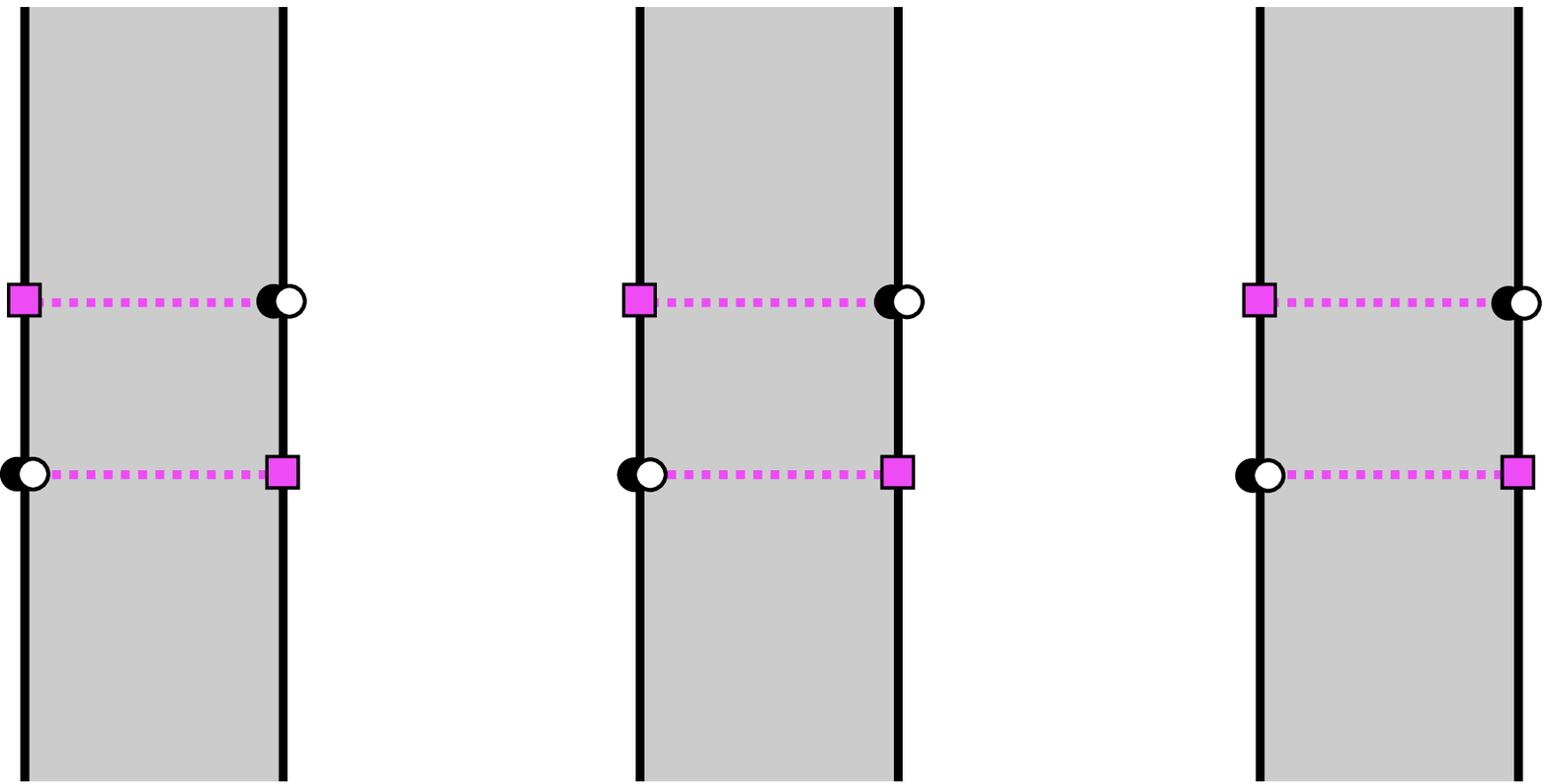}};
    \node at (-.35,.2) {\tiny $O_n$};
    \node at (-.35,-.1) {\tiny $X_n$};
    \node at (-3.1,.6) {\tiny $O_n$};
    \node at (-3.1,.2) {\tiny $X_n$};
    \node at (2.45,1.3) {\tiny $O_n$};
    \node at (2.45,1) {\tiny $X_n$};
    \end{tikzpicture}
    \caption{The relative locations of markings and generators on annuli counted in the proof of Lemma~\ref{lem:phixnquasiisomorphism}. Each has a unique decomposition contributing to one of the terms of (\ref{eq:onxnhomotopy}). The first contributes to $\Phi_{O_n}\circ\Phi_{X_n}^{\tbI}$, the second to $\dnn\circ\Phi_{O_n,X_n}$, and the third to $\Phi_{O_n,X_n}\circ\dnn$.  }
    \label{fig:phionxnannuli}
\end{figure}

 \end{proof}

 \begin{prop}\label{prop:XSWstabilization}
 Suppose that $\gkprime$ is gotten by performing an $\XSW$ stabilization at $X_{n-1}$. Then there is a quasi-isomorphism as trigraded chain complexes over $\Ztwo[V_0,\ldots,V_{n-1}]$ of $C^-(\gridg)$ and $C^-(\gkprime)$. 
 \end{prop}
 
 \begin{proof}
 We have \[
 \begin{tikzcd}
 \tbI\arrow[r,"\partial_{\tbI}^{\tbN}"]\arrow[d,"e"] & \tbN\arrow[d,"e\circ\Phi_{X_n}^{\tbI}"]\\
 C[V_n]\gop1,1,0\gcl\arrow[r,"V_n-V_{n-1}"]&C[V_n].
 \end{tikzcd}
 \]

The square commutes, as $\Phi_{X_n}^{\tbI}\circ\partial_{\tbI}^{\tbN}=V_n-V_{n-1}$, and $e$ is an isomorphism induced by a bijection of generators. The vertical maps are quasi-isomorphisms, by Lemmas~\ref{lem:Iquasi-isomorphism} and \ref{lem:phixnquasiisomorphism}. The horizontal maps are chain maps of degree $(-1,0,0)$, by Proposition \ref{prop:parallelogramgradings} and Lemma \ref{lem:phixndegree}. Then by Lemma \ref{lem:conealgebraiclemma}, we have an induced quasi-isomorphism as trigraded $\Ztwo[V_0,\ldots,V_n]$ modules from $C'=\Cone(\partial_{\tbI}^{\tbN}:\tbI\to\tbN)$ to $C''=\Cone(V_n-V_{n-1}:C[V_n]\gop1,1,0\gcl\to C[V_n])$, and by Lemma \ref{lem:freeconelemma}, we know $C''$ and $C$ are quasi-isomorphic as trigraded $\Ztwo[V_0,\ldots,V_{n-1}]$ modules. 
 \end{proof}

 Before we prove Proposition~\ref{prop:stabilization}, we need a combinatorial lemma. 
 
 \begin{lem}[{\cite[Lemma 3.2.2]{gridhomology}}]\label{lem:typeOistypeX}
 A stabilization of type $\ONE$ (respectively $\OSE$, $\ONW$, or $\OSW$) can be realized by a stabilization of type $\XSW$ (respectively $\XNW$, $\XSE$, or $\XNE$) followed by a sequence of commutation moves. 
 \end{lem}
 
 \begin{proof}[Proof of Proposition~\ref{prop:stabilization}]
 By the above lemma, it is enough to prove that if $\gridg$ and $\gkprime$ are connected by a stabilization of type $X$ then $(C^-(\gridg),\dxminus)$ and $(C^-(\gkprime),\dxminusprime)$ are homotopy equivalent. This is true for an $\XSW$ stabilization by Proposition~\ref{prop:XSWstabilization}. For a $\XNE$ stabilization the proof is essentially the same, as follows. We can still decompose $C'=\tbI\oplus\tbN$ with the same differential, and we have $\Phi_{X_n}^{\tbI}:\tbN\to\tbI$ and $\Phi_{O_n}:\tbI\to\tbN$ are chain maps which are inverse to each other on $\tbI$ and homotopy inverse to each other on $\tbN$ via homotopy $\Phi_{O_n,X_n}$. As a result, we can again apply Lemma~\ref{lem:conealgebraiclemma} to the diagram:  \[
 \begin{tikzcd}
 \tbI\arrow[r,"\partial_{\tbI}^{\tbN}"]\arrow[d,"e"] & \tbN\arrow[d,"e\circ\Phi_{X_n}^{\tbI}"]\\
 C[V_n]\gop1,1,0\gcl\arrow[r,"V_n-V_{n-1}"]&C[V_n].
 \end{tikzcd}
 \]
 
 A $\XSE$ stabilization is the composition of an $\XSW$ stabilization and a switch move. By the above and Proposition~\ref{prop:switch}, then, we have the desired result for $\XSE$ stabilizations. Similarly, a $\XNW$ stabilization is the composition of an $\XNE$ stabilization and a switch move, proving the result. 
 
 The fact that the respective chain complexes are quasi-isomorphic over $\Ztwo[V_0,\ldots,V_{n-1}]$, in combination with Lemma~\ref{lem:componentvariablesarehomotopic}, proves that they are quasi-isomorphic over $\Ztwo[U_1,\ldots,U_\ell]$. 
 \end{proof}

%% file: Sections/signassignments.tex
In this section we will recall the definition of a sign assignment on a grid diagram, and present a few results of \cite{celoria2015note}.

\begin{defn}[{\cite[Definition 3.1]{celoria2015note}}]

Given a grid diagram $\gridg$, let $\x,\tgen,\w,\z\in\genG$. Let $r_1\in \PG(\x,\tgen)$, $r_2\in\PG(\tgen,\z)$, $r_3\in\PG(\x,\w)$, and $r_4\in\PG(\w,\z)$. Then a \emph{sign assignment} on $\gridg$ is a function $$
\mathcal{S}:\PG(\gridg)\to\{\pm1\}
$$ 

satisfying the following three properties:\begin{itemize}
    \item[(S1)] If $r_1*r_2=r_3*r_4$ and $\tgen\ne\w$, then $\saS(r_1)\saS(r_2)=-\saS(r_3)\saS(r_4)$. 
    \item[(S2)] If $r_1*r_2$ is a row, then $\saS(r_1)\saS(r_2)=1$. 
    \item[(S3)] If $r_1*r_2$ is a column, then $\saS(r_1)\saS(r_2)=-1$. 
\end{itemize}
\end{defn}

In \cite{celoria2015note}, it is proved that on grid diagrams for links in lens spaces, sign assignments exist, and are unique (up to homology). 

\begin{thm}[{\cite[Theorem 1.1]{celoria2015note}}]\label{thm:signassignmentsdontmatter}
Sign assignments exist on all grids representing a link $L\subset L(p,q)$. Moreover, the sign-refined grid homology does not depend on the choice of a sign assignment. 
\end{thm}

The above theorem is constructive, but we will only use the above three properties of any sign assignment, so we refer the reader to \cite{celoria2015note} for an example of how to explicitly define a sign assignment. 

%% file: Sections/integralinvariance.tex
In \cite{celoria2015note}, the definition of grid homology for twisted toroidal grid diagrams was extended to integer coefficients, using a sign assignment. After recalling the definition, we prove that this homology is a link invariant. 

For $L\subset \lpq$, define $C^-(\gridg;\Z)$ as follows. Let $\gridg$ be a grid diagram for a link $L$ with $\ell$ components, and fix a sign assignment $\saS$ on $\gridg$. Let $C^-(\gridg;\Z)$ be the free $\Z[V_0,\ldots,V_{n-1}]$ module generated by $\genG$. For $\x\in\genG$, define \[\dxminus(\x):=\sum_{y\in\genG}\sum_{\substack{r\in\ePG(\x,\y)\\r\cap\XX=\emptyset}}\saS(r)(\prod_{i=0}^{n-1}V_i^{n_{O_i}(r)})\y.\] 
That this is a differential and thus defines a chain complex comes from the first axiom of a sign assignment, and the alternate decompositions in the proof of Proposition \ref{prop:d2=0}. 

We note that we really should denote the differential by $\partial_{\XX,\saS}^-$, but Theorem \ref{thm:signassignmentsdontmatter}, proven in \cite{celoria2015note}, proves that the differentials assigned to any two sign assignments on the same grid define quasi-isomorphic chain complexes. 

Furthermore, we will prove the link invariant is the quasi-isomorphism type of $C^-(\gridg;\Z)$ as a chain complex over $\Z[U_1,\ldots,U_\ell]$, where each variable corresponds to the action of a distinct component of the link. We need the following lemma. 

\begin{lem}\label{lem:componentvariablesarehomotopicsigned}
Suppose $O_i$ and $O_j$ are markings corresponding to the same component of $K$. Then multiplication by $V_i$ is chain homotopic to multiplication by $V_j$. 
\end{lem}
\begin{proof}
The proof is essentially the same as that of Lemma~\ref{lem:componentvariablesarehomotopic}, adapted to integral coefficients. Define $\Phi_{X_a}:C^-(\gridg)\to C^-(\gridg)$ on a generator $\x\in\genG$ by \[
\Phi_{X_a}(\x)=\sum_{\y\in\genG}\sum_{\substack{r\in\ePG(\x,\y)\\\Int(r)\cap\XX=X_a}}\saS(r)V_o^{n_{O_0}(r)}\cdots UV{n-1}^{n_{O_{n-1}}(r)}\y.
\] Extend this linearly to $C^-(\gridg;\Z)$. The only difference is that now the annulus of width one which connects a generator $\x\in\genG$ to itself is counted with coefficient $-1$, and the annulus of height one is counted with coefficient $1$. Thus, using $-\Phi_{X_a}$ as our homotopy, we get that $-\dxminus\circ\Phi_{X_a}-\Phi_{X_a}\circ\dxminus = V_i-V_j$. 
\end{proof}

To show that this is a link invariant, not just dependent on the grid representing the link, requires showing that the commutation and stabilization maps defined before are compatible with the sign assignments.

\subsection{Commutations}\label{sec:integralcommutation}
\input{Sections/IntegralInvariance/commutation}

\subsection{Switches}\label{sec:integralswitches}
\input{Sections/IntegralInvariance/switches}

\subsection{Stabilizations and destabilizations}
\input{Sections/IntegralInvariance/stabilization}

\subsection{Integral invariance}
With the above in hand, the main theorem follows. 

\begin{proof}[Proof of Theorem~\ref{thm:mainresultz}]
Suppose $\gridg$ and $\gkprime$ are two grid diagrams associated to the same link. By Theorem~\ref{thm:cromwell},  $\gridg$ and $\gkprime$  are connected by a sequence of commutations and (de)stabilizations. From Proposition~\ref{prop:signedcommutationinvariance} and Proposition~\ref{prop:signedstabilizationinvariance}, it follows that $(C^-(\gridg),\dxminus;\Z)$ and $(C^-(\gkprime),\dxminus;\Z)$ are quasi-isomorphic. 
\end{proof}

Corollary~\ref{cor:invariant} follows immediately. The quasi-isomorphism type $\CFKm(L,\lpq;\Z)$ is a link invariant by Theorem~\ref{thm:mainresultz}, and so the isomorphism type of $\HFKm(L,\lpq;\Z):=H_*(\CFKm(L,\lpq;\Z))$ is also.

%% file: Sections/IntegralInvariance/commutation.tex
Before we define our desired chain map, we need a way to assign signs to pentagons. We are equipped with a way to assign signs to parallelograms, so we define a map which takes a pentagon on the combined grid diagram and returns a parallelogram. Suppose, as in Section~\ref{sec:commutation}, that $\gridg$ and $\gkprime$ are two grid diagrams connected by a commutation of columns. Note that the only thing which has changed is the position of the markings, and a sign assignment doesn't depend on the position of the markings, so there is a bijection between sign assignments on $\gridg$ and on $\gkprime$. Fix a common sign assignment and denote it $\saS$. 

Let $a$ denote the intersection point of $\beta$ and $\gamma$ which is included in every pentagon counted by $\Phi_{\beta\gamma}$. Let $\x\in\genG$ and $\y'\in\genG'$. The component $y'\in \y'$ which lays on $\gamma$ can be slid along the $\alpha\in\alphas$ curve it occupies, without crossing any other $\betas$ curves, to a unique intersection point of that $\alpha$ and $\beta$, which we will call $y$. Define the generator $\y\in\genG$ by replacing $y'$ with $y$ in $\y'$. Bounded by this $\alpha$ arc, the arc of $\beta$ connecting $y$ to $a$, and the arc of $\gamma$ connecting $a$ to $y'$ is an embedded triangle, which we will denote $t_{\y'}$.  We could do the same process in reverse to $\x$ to get $\x'\in\genG'$ and $t_{\x}$. If we have $s\in\ePent(\x,\y')$, then either $s*t_{\y'}\in\ePG(\x,\y)$ or $t_{\x}*s\in\ePG(\x',\y')$. There is a natural one-to-one correspondence between $\ePG(\gridg)$ and $\ePG(\gkprime)$, which if both of the above parallelograms exists, maps them to each other. Thus we can define $R(s)$ to be this parallelogram, which we will take to live in $\ePG(\gridg)$. 

We call a pentagon a left (respectively right) pentagon if it has multiplicity one to the left of $a$ (respectively right), or, equivalently, if the above defined parallelogram is a left (respectively right) parallelogram. 

For $s\in\ePent(\x,\y')$, define a sign assignment on pentagons denoted $\saS_{\text{pent}}$ by \[
\saS_{\text{pent}}(s):=\begin{cases}
(-1)^{\maslov(\x)+1}\saS(R(s)) &\text{if $s$ is a left pentagon,}\\
(-1)^{\maslov(\x)}\saS(R(s)) &\text{if $s$ is a right pentagon.}
\end{cases}
\]

We will abuse notation and denote the above just $\saS$. 

Define $\Phi_{\beta\gamma,\saS}:C^-(\gridg)\to C^-(\gkprime)$ on generators $\x\in\genG$ by \[\Phi_{\beta\gamma,\saS}(\x)=\sum_{\y\in \genG'}\sum_{\substack{p\in \epbg(\x,\y)\\p\cap\XX=\emptyset}}\saS(p)V_0^{n_{O_0}(p)}\cdots V_{n-1}^{n_{O_{n-1}}(p)}\y. \]

We will abuse notation and leave the sign assignment implicit, denoting the above map by $\Phi_{\beta\gamma}$. 

\begin{lem}\label{lem:phibetagammasignedchainmap}
The map $\Phi_{\beta\gamma}$ is a chain map.
\end{lem}
\begin{proof}
This is similar to the proof for the case of knots in $S^3$ ({\cite[Lemma 15.3.3]{gridhomology}]}. 

As in Lemma~\ref{lem:phibetagamma}, we want to prove $\Phi_{\beta\gamma}\circ\dxminus=\dxminus\circ\Phi_{\beta\gamma}$, so we again consider domains which are the juxtaposition of a pentagon and a parallelogram or are the juxtaposition of a parallelogram and a pentagon. Each such domain is given with an alternate decomposition in the proof of Lemma~\ref{lem:phibetagamma}. For the domains of Figure~\ref{fig:pentagon4}, Figure~\ref{fig:pentagon3differentcomponent}, and Figure~\ref{fig:pentagon2alternate}, we have that the domain $\phi$ connecting $\x$ to $\z$ has alternate decompositions as $\phi=r*s$ and $\phi=s'*r'$ for some $r\in\ePG(\x,\tgen)$, $s\in\ePent(\tgen,\z)$, $s'\in\ePent(\x,\w')$, and $r'\in\ePG(\w',\z)$ for some generators $\tgen\in\genG$ and $\w\in\genG'$. We have that $r*R(s)=R(s')*r'$, so $\saS(r)\saS(R(s))=-\saS(R(s'))\saS(r')$. Both $s$ and $s'$ are right pentagons in the displayed domains, but the analysis is the same when they are both left pentagons. The Maslov degrees of the initial generators of $s$ and $s'$ differ by one, since $\dxminus$ is degree $-1$ with respect to Maslov grading, so we get that $\saS(r)\saS(s)=\saS(s')\saS(r')$, proving our desired equality in this case. 

For the domains of Figure~\ref{fig:pentagon3differentbetagamma} and Figure~\ref{fig:pentagon1horizontal}, we have the same decompositions as above. If, however, we have $\w\in\genG$ as defined by sliding $\w'$ along a curve $\alpha\in\alphas$ as described above, then $\w=\tgen$, so we no longer appeal to the first axiom of a sign assignment. Instead, we have that $\saS(r)\saS(R(s))=\saS(R(s'))\saS(r')$. The Maslov degrees of the initial generators of $s$ and $s'$ differ by one, and $s$ and $s'$ are pentagons of different sides. The effect of these differences cancels out, and we get $\saS(r)\saS(s)=\saS(s')\saS(r')$, again proving our desired equality in this case.

For the domains in Figure~\ref{fig:pentagon3samecomponent} and Figure~\ref{fig:pentagon2same}, we have that the domain $\phi$ has alternate decompositions as $\phi=r*s$ and $\phi=r'*s'$ for some $r\in\ePG(\x,\tgen)$, $s\in\ePent(\tgen,\z)$, $r'\in\ePG(\x,\w)$ and $s'\in\ePent(\w,\z)$ for some generators $\x,\tgen\in \genG$ (or has two decompositions as the juxtaposition of a pentagon and parallelogram, but the analysis is the same). In this case, we have $\saS(r)\saS(R(s))=-\saS(r')\saS(R(s'))$. The Maslov degrees of the initial generators of $s$ and $s'$ are the same, and they are same-sided pentagons, so $\saS(r)\saS(s)=-\saS(r')\saS(s')$. Therefore, these terms drop out from one side of our desired equality. 

For the domains in Figure~\ref{fig:pentagon1vertical}, regardless of the decomposition, after mapping each pentagon to its appropriate parallelogram on the diagram, we get that each decomposition forms a column, so is counted with coefficient $-1$. The pentagons, however, are on different sides. If the decompositions are of the same type, then, they are counted with opposite sign and cancel in the sum. If they are of opposite type, then the Maslov degrees of the initial generators differ by one, so we get that the decompositions are counted with the same sign, proving our desired equality. 
\end{proof}

We will show that $\Phi_{\beta\gamma}$ is a chain homotopy equivalence, with $-\Phi_{\gamma\beta}$ as a homotopy inverse, with a homotopy provided by $-\Hbgb$, after defining these maps with signs below. 

First, we need a way to assign signs to hexagons. Each hexagon in $\Hbgb$ contains an arc of $\gamma$ which borders a bigon between $\gamma$ and $\beta$; call that bigon $g$. We get that for $h\in\hexbgb(\x,\y)$, $g*h\in\ePG(\x,\y)$. Let $R(h)=g*h$, and define $\saS(h)=\saS(R(h))$, by an abuse of notation. Now we can define $\HH_{\beta\gamma\beta,\saS}:C^-(\gridg)\to C^-(\gridg)$ by \[\HH_{\beta\gamma\beta,\saS}(\x)=\sum_{\y\in\genG}\sum_{\substack{h\in\ehexbgb(\x,\y)\\\Int(h)\cap\XX=\emptyset}}\saS(h)V_0^{n_{O_0}(h)}\cdots V_{n-1}^{n_{O_{n-1}}(h)}\y.\]

We will abuse notation and denote the above by $\Hbgb$. 

\begin{lem}\label{lem:phibetagammasigned}
The map $\Phi_{\beta\gamma}$ is a chain homotopy equivalence. 
\end{lem}
\begin{proof}
Again, this is similar to the proof for the case of knots in $S^3$ ({\cite[Lemma 15.3.4]{gridhomology}]}.

We are trying to prove $\Id+\Phi_{\gamma\beta}\circ\Phi_{\beta\gamma}+\dxminus\circ \Hbgb+\Hbgb\circ\dxminus=0$ as maps on $C^-(\gridg)$. We can follow Lemma~\ref{lem:hbgb} and characterize a portion of the above equation as \begin{equation}\label{eq:hexhomotopysigned}
(\Phi_{\gamma\beta}\circ\Phi_{\beta\gamma}+\dxminus\circ \Hbgb+\Hbgb\circ\dxminus)(\x)=\sum_{\z\in\genG}\sum_{\substack{\phi\in\epi(\x,\z)\\\Int(\phi)\cap\XX=\emptyset}}N(\phi)V_0^{n_{O_0}(\phi)}\cdots V_{n-1}^{n_{O_{n-1}}(\phi)}\z, 
\end{equation}
where this time, $N(\phi)$ is the signed number of decompositions of $\phi$. We wish to show $N(\phi)=0$ for $m\in\{1,2,3,4\}$, and $N(\phi)=-1$ for $m=0$, as we will then have that the above is equal to $-\Id$, as desired. 

The cases for $1\le m\le 4$ are very similar to the proof of Lemma~\ref{lem:phibetagammasignedchainmap}, so we omit the analysis here. In summary, after straightening each pentagon and hexagon involved in a decomposition of a domain $\phi$, we end up in one of two cases. The first case is $\phi$ has a pair of juxtapositions of parallelogram satisfying the conditions of the first property of a sign assignment in which case the decompositions are counted with opposite sign and we have $N(\phi)=0$. In the second case, the supports of the parallelograms are not distinct, in which case the pentagons must come from different signs, so the cancelling sign comes from the definition of the sign assignment on pentagons, and again we have $N(\phi)=0$. 

For the case $m=0$, we have the domains of Figure~\ref{fig:hex0}, each of which after straightening forms a column, so is counted with coefficient $-1$. Since hexagons are counted with the same sign as their straightened parallelograms, we have that any decomposition of a column contributing to $\dxminus\circ \Hbgb$ or $\Hbgb\circ\dxminus$ is counted with coefficient $-1$. For decompositions of columns contributing to $\Phi_{\gamma\beta}\circ\Phi_{\beta\gamma}$, both pentagons are on the same side. If they are both right (respectively left) pentagons, the sign differs by a factor of $(-1)^{\maslov(\x)}(-1)^{\maslov(\Phi_{\beta\gamma}(\x))}$ (respectively $(-1)^{\maslov(\x)+1}(-1)^{\maslov(\Phi_{\beta\gamma}(\x))+1}$. Since $\Phi_{\beta\gamma}$ preserves Maslov grading, we have that the sign differs by a factor of 1, so the column is still counted with coefficient $-1$. 

Thus, as desired we have (\ref{eq:hexhomotopysigned}) counts exactly one domain coming out of each generator $\x\in\genG$, which connects it to $\x$ with coefficient $-1$, so this map is the negative identity. 
\end{proof}

\begin{prop}\label{prop:signedcommutationinvariance}
Suppose $\gridg$ and $\gkprime$ are grid diagrams for a link with $\ell$ components which are connected by a commutation move. Then $C^-(\gridg)$ and $C^-(\gridg)$ homotopy equivalent as chain complexes over $\Z[U_1,\ldots,U_\ell]$. 
\end{prop}
\begin{proof}
Suppose the diagrams differ by a commutation of columns. Then the chain map $\Phi_{\beta\gamma}$ defined above is a chain homotopy equivalence, via the homotopy $-\Hbgb$, via Lemma \ref{lem:phibetagammasigned}, and has homotopy inverse $-\Phi_{\gamma\beta}$. That this map is trigraded is the content of Lemma \ref{lem:pentagongrading}; it does not change in the signed setting. If instead $\gridg$ and $\gkprime$ differ by a commmutation of rows, the proof is in essence the same, with the above diagrams rotated. We note that the definition of a sign assignment on pentagons would rely in that case on above vs. below pentagons, instead of left vs. right pentagons. 
\end{proof}

%% file: Sections/IntegralInvariance/switches.tex
Again, understanding a switch as in Figure~\ref{fig:switchmove}, the same proof that the chain complexes associated to two grid diagrams over $\Z$ connected by a commutation are homotopy equivalent carries over to the case of a switch. 

\begin{prop}\label{prop:signedswitchinvariance}
Suppose $\gridg$ and $\gkprime$ are grid diagrams for a link with $l$ components which are connected by a switch move. Then $(C^-(\gridg;\Z),\dxminus)$ and $(C^-(\gkprime;\Z),\dxminus)$ are homotopy equivalent as chain complexes over $\Z[U_1,\ldots, U_\ell]$. 
\end{prop}

%% file: Sections/IntegralInvariance/stabilization.tex
As in Section~\ref{sec:stabilization}, suppose $\gridg$ is a grid diagram for a link $L$ in a lens space $\lpq$, and suppose $\gkprime$ is the grid diagram given by performing an $\XSW$ stabilization at $X_{n-1}$. Let $\saS_{\gkprime}$ be any sign assignment on $\gkprime$; this restricts to a sign assignment on $\gridg$, as follows. For $\x,\y\in\genG$, and $r\in\PG(\x,\y)$, we have that there is a corresponding rectangle $r'\in \PG(\x',\y')$, where $\x'=\x\cup c,\y'=\y\cup c\in\genG'$. Define $\saS_{\gridg}(r)=\saS_{\gkprime}(r')$. This clearly satisfies the three properties for a sign assignment. Note that this is the first time that it has been used that a sign assignment assigns signs to all parallelograms, not just empty ones, as even if $r$ is empty, $r'$ need not be. Throughout, we will call both these sign assignments $\saS$. We also will follow the notation of Section~\ref{sec:stabilization}, letting $C$ denote $C^-(\gridg)$ and $C'$ denote $C^-(\gkprime)$.

Recall the definition of $\bI:=\{\x\in\genG':c\in\x\}$, where $c$ is the unique intersection point of $\alpha_n$ and $\beta_n$ which abuts regions containing $X_n$, $O_n$, and $X_{n-1}$. 

First, note that when considered with sign, $C'$ is not equal to $\Cone(\din)$. 
 Instead, we have \[C'=\Cone(\begin{tikzcd}
 (\tbI,-\dii)\arrow[r,"\din"] & (\tbN,\dnn)
 \end{tikzcd}
 )
\] Note that indeed the above are still chain complexes, and by counting alternate pairs, we get that $\din$ is indeed a chain map. We still wish to appeal to Lemma~\ref{lem:conealgebraiclemma}, so need to define quasi-isomorphisms between the summands of $C'$ and $C'':=\Cone(V_n-V_{n-1}:C[V_n]\to C[V_n])$. These will just be appropriately signed versions of the maps of Section~\ref{sec:stabilization}.

\begin{lem}\label{lem:signedIquasi-isomorphism}
The map defined on generators $\x\in\bI$ by $e_{\Z}(\x)=(-1)^{\maslov(\x)}(\x\cup c)$ and extended linearly to a map from $(\tbI,-\dii)$ to $(C[V_n]\gop 1,1,0\gcl,\dxminus)$ induces an isomorphism of trigraded chain complexes over $\Z[V_0,\ldots,V_n]$. 
 \end{lem}
 
 \begin{proof}
 With the definition of the induced sign assignment in hand, and noting that $\dii$ is degree $-1$ with respect to the Maslov grading, this follows from the proofs of Lemma~\ref{lem:Iquasi-isomorphism} and Lemma~\ref{lem:Igradings}.
 \end{proof}
 
  Define the map $\Phi_{X_n,S}^{\tbI}:\tbN\to\tbI$ on generators $\x\in\mathbf{N}$ by $$\Phi_{X_n,S}^{\tbI}(\x):=\sum_{\y\in\tbI}\sum_{\big\{\substack{r\in\eRect(\x,\y)\\\Int(r)\cap \XX=X_n}\big\}}\saS(r)V_0^{n_{O_0}(r)}\cdots V_{n}^{n_{O_{n}}(r)}\y. $$ 
 
 Extend this map linearly to $\tbN$.

 \begin{lem}
  The map $\Phi_{X_n,\saS}^{\tbI}$ is a chain map of tridegree $(-1,-1,0)$. 
 \end{lem}
 \begin{proof}
This requires showing $\Phi_{X_n,\saS}^{\tbI}\circ\dnn=-\dii\circ\Phi_{X_n,\saS}^{\tbI}$. Again, this is a case analysis of domains and possible placements of $X_n$. These are analyzed in the proof of Lemma~\ref{lem:phixni}. For the domains of Figure~\ref{fig:phixnisamedecompositions}, both decompositions contribute to the same side of the above equality, and are counted with opposite sign according to the first axiom of a sign assignment, so they cancel out. For the domains of Figure~\ref{fig:phixnialternatedecompositions}, the decompositions contribute to opposite sides of the above equality. The difference in sign from the sign assignment cancels with the negative sign to prove our desired equality. 
 \end{proof}
 
 The map $e_{\Z}\circ\Phi_{X_n,\saS}^{\tbI}$ is a chain homotopy equivalence between $\tbN$ and $C[V_n]$. To see this, we again will define an explicit homotopy inverse and homotopy. 
 
 Define the map $\Phi_{O_n,\saS}:\tbI\to\tbN$ on generators $\x\in\bI$ by \[
\Phi_{O_n,\saS}(\x)=\sum_{\y\in\tbN}\sum_{\substack{r\in\ePG(\x,\y)\\
 \Int(r)\cap\XX=\emptyset\\
 O_n\in r}}\saS(r)V_0^{n_{O_0}(r)}\cdots V_{n-1}^{n_{O_{n-1}}(r)}\y.
\]

Extend this map linearly to $\tbI$. 

 \begin{lem}
 The map $-\Phi_{O_n,\saS}$ is a chain map of tridegree $(1,1,0)$, and on $\tbI$, the composition $-\phixnisigned\circ\phionsigned$ is the identity. 
 \end{lem}
 \begin{proof}
 That $-\phionsigned\circ\dii=\dnn\circ\phionsigned$ is seen by a similar signed case analysis to the above. The grading result follows from Proposition~\ref{prop:parallelogramgradings}.
 
 Analyzing possible placements of $X_n$ and $O_n$ in the domains of Proposition~\ref{prop:d2=0}, we see the only domains counted are the vertical annuli of width one. If a vertical annulus $\phi$ has a decomposition as $\phi=r_1*r_2$, the sign assignment tells us $\saS(r_1)\saS(r_2)=-1$. In the above, then, the annulus is counted with sign +1. The map $\phionsigned$ is defined so as to not count the coefficient coming from $O_n$, so we get the composite is equal to the identity.
 \end{proof}

 Define $\phionxnsigned$ on a generator $\x$ of $\mathbf{N}$ by \[
\phionxnsigned(\x)=\sum_{\y\in\tbN}\sum_{
 \substack{r\in\ePG(\x,\y)\\
 \Int(r)\cap\XX=X_{n-1}\\
 O_n\in r}}
 \saS(r)V_0^{n_{O_0}(r)}\cdots V_{n-1}^{n_{O_{n-1}}(r)}\y. 
\]
Extend this map linearly to $\mathbf{N}$. 

\begin{lem}
 On $\tbN$, the composition $-\phionsigned\circ\phixnisigned$ is homotopic to the identity. 
\end{lem}
\begin{proof}
We will prove the homotopy is $-\phionxnsigned$. Therefore, we wish to show $-\Id=\phionsigned\circ\phixnisigned+\dnn\circ\phionxnsigned+\phionxnsigned\circ\dnn$. Any domain that is counted on the right hand side has two alternate decompositions counted with opposite sign, cancelling out in sum, as in Lemma~\ref{lem:phixnquasiisomorphism}, except for the annuli, which have a unique decomposition, counted with sign -1. 
\end{proof}

\begin{prop}\label{prop:xswsigned}
Suppose that $\gkprime$ is gotten by performing an $\XSW$ stabilization at $X_{n-1}$. Then there is a quasi-isomorphism as trigraded chain complexes over $\Z[V_0,\ldots,V_{n-1}]$ of $C^-(\gridg)$ and $C^-(\gkprime)$. 
\end{prop}
\begin{proof}
We have \[
 \begin{tikzcd}
 (\tbI,-\dii)\arrow[r,"\din"]\arrow[d,"e"] & (\tbN,\dnn)\arrow[d,"e\circ\phixnisigned"]\\
 (C[V_n]\gop1,1,0\gcl,\dxminus)\arrow[r,"V_n-V_{n-1}"]&(C[V_n],\dxminus).
 \end{tikzcd}\]
 
 This square commutes. 
 The vertical maps are quasi-isomorphisms, so by Lemma~\ref{lem:conealgebraiclemma} we get that the induced map between $C^-(\gkprime)=\Cone(\din)$ and $C''=\Cone(V_n-V_{n-1})$ is a quasi-isomorphism. Lemma~\ref{lem:freeconelemma} says that $C''$ is quasi-isomorphic to $C^-(\gridg)$. 
\end{proof}

\begin{prop}\label{prop:signedstabilizationinvariance}
Suppose $\gridg$ and $\gkprime$ are two grid diagrams for a link with $\ell$ components which are connected by a stabilization move. Then $(C^-(\gridg),\dxminus;\Z)$ and $(C^-(\gkprime),\dxminusprime;\Z)$ are quasi-isomorphic as chain complexes over $\Z[U_1,\ldots,U_\ell]$. 
\end{prop}

\begin{proof}
We have proven the statement for stabilizations of type $\XSW$ in Proposition~\ref{prop:xswsigned}. The argument that it is true for a stabilization of type $\XNE$ follows just as in the proof of Proposition~\ref{prop:stabilization}. Similarly, a stabilization of type $\XSE$ (respectivel $\XNW$) is the composition of a stabilization of type $\XSW$ (respectively $\XNE$) and a switch, so by the above and Proposition~\ref{prop:signedswitchinvariance}, we have the desired result for stabilizations of type $X$. By Lemma~\ref{lem:typeOistypeX} and Proposition~\ref{prop:signedcommutationinvariance}, we have the desired result for stabilizations of type $O$. 
\end{proof}